\documentclass[12pt]{amsart}
\DeclareMathAlphabet{\mathpzc}{OT1}{pzc}{m}{it}
\usepackage{euscript}
\usepackage{amssymb}
\usepackage{amsmath}
\usepackage{epic}
\usepackage{graphics}
\usepackage{epsfig}
\usepackage{color}
\usepackage{stmaryrd}

\usepackage{picture}

\usepackage{epstopdf} 

\usepackage{tikz}
\usetikzlibrary{decorations.pathreplacing}
\usetikzlibrary{matrix,arrows}

\usepackage{amscd,euscript}
\usepackage[frame,cmtip,curve,arrow,matrix,line,graph]{xy}
\usepackage{yfonts}

\numberwithin{equation}{section}

\newtheoremstyle{my}{1.5em}{.5em}{\em}{}{\sc}{.}{0.5em}{}

\newtheorem{thm}{Theorem}[section]
\newtheorem{Theorem}[thm]{Theorem}
\newtheorem*{Theorem*}{Theorem} 

\newtheorem{corollary}[thm]{Corollary}
\newtheorem*{corollary*}{Corollary}
\newtheorem{Lemma}[thm]{Lemma}

\newtheorem{prop}[thm]{Proposition}

\newtheorem*{conjecture*}{Conjecture}

\newtheorem*{question*}{Question}
\newtheorem{defn}[thm]{Definition}

\newtheorem*{definitions*}{Definitions}

\newtheorem*{rem*}{Remark}
\newtheorem{Remark}[thm]{Remark}
\newtheorem{assumption}[thm]{Assumption}

\newtheorem*{remark*}{Remark}
\newtheorem{remarks}[thm]{Remarks}
\newtheorem*{remarks*}{Remarks}
\newtheorem*{example*}{Example}
\newtheorem{Example}[thm]{Example}
\newtheorem*{examples*}{Examples}
\newtheorem{examples}[thm]{Examples}
\newtheorem*{convention*}{Convention}
\newtheorem*{conventions*}{Conventions}

\newtheorem*{exercise*}{Exercise}
\newtheorem*{bibliographical-note*}{Bibliographical note}

\newtheorem{lemma}[thm]{Lemma}
\newenvironment{pf}{\paragraph{Proof}}{\qed\par\medskip}
\newtheorem{remark}[thm]{Remark}

\newcommand{\Acknowledgements}{{\em Acknowledgements.} }

\newcommand{\bR}{\mathbb{R}}

\newcommand{\bC}{\mathbb{C}}

\newcommand{\bP}{\mathbb{P}}

\newcommand{\half}{{\textstyle\frac{1}{2}}}

\newcommand{\Sym}{\operatorname{Sym}}

\newcommand{\Hom}{\mathrm{Hom}}

\newcommand{\Aut}{\operatorname{Aut}}

\newcommand{\Diff}{\mathrm{Diff}}

\newcommand{\A}{\mathcal{A}}

\newcommand{\oPP}{\oplus\mathbb{P}}

\numberwithin{equation}{section}
\renewcommand{\leq}{\leqslant}
\renewcommand{\geq}{\geqslant}
\newcommand{\isom}{\cong}
\newcommand{\tensor}{\otimes}

\newcommand{\blob}{{\scriptscriptstyle\bullet}}

\newcommand{\lRa}[1]{\xrightarrow{\ #1\ }}

\newcommand{\lra}{\longrightarrow}

\newcommand{\QQ}{\mathbb{Q}}
\newcommand{\blank}{-}
\newcommand{\D}{{\mathcal{D}}}

\newcommand{\PP}{\operatorname{\mathbb P}}
\newcommand{\C}{\mathbb C}
\newcommand{\RR}{\operatorname{R}}

\newcommand{\B}{\mathcal B}
\newcommand{\CC}{\mathcal C}
\newcommand{\OO}{\mathcal O}
\renewcommand{\P}{\mathcal{P}}

\newcommand{\Z}{\mathbb{Z}}
\newcommand{\R}{\mathbf R}

\newcommand{\constant}{\operatorname{constant}}

\newcommand{\Tw}{\operatorname{Tw}}

\newcommand{\Ker}{\operatorname{Ker}}
\newcommand{\Ext}{\operatorname{Ext}}

\newcommand{\Mod}{\operatorname{\mathcal{M}}}

\newcommand{\op}{\operatorname{op}}
\renewcommand{\sc}{\operatorname{sc}}
\newcommand{\Tri}{\operatorname{Tri}}
\newcommand{\ord}{\diamond}
\newcommand{\Triord}{\operatorname{Tri}^{\ord}}

\newcommand{\Zer}{\mathrm{Zer}}
\newcommand{\Pol}{\mathrm{Pol}}
\newcommand{\Quad}{\operatorname{Quad}}
\newcommand{\Stab}{\operatorname{Stab}}

\newcommand{\CY}{CY$_3$ }
\newcommand{\M}{\operatorname{M}}
\newcommand{\T}{{T}}

\newcommand{\TT}{\mathcal{T}}
\newcommand{\FF}{\mathcal F}
\renewcommand{\Im}{\operatorname{Im}}
\renewcommand{\H}{\bar{\mathcal{H}}}
\renewcommand{\op}{{\circ}}

\newcommand{\MCG}{\operatorname{MCG}}

\newcommand{\Sph}{\operatorname{Sph}}
\newcommand{\cM}{\mathcal{M}}
\newcommand{\hS}{\Hat{S}}
\newcommand{\m}{{ {m}}}
\newcommand{\Tilt}{\operatorname{Tilt}}
\renewcommand{\SS}{{\textsection}}
\renewcommand{\S}{\mathbb{S}}
\renewcommand{\M}{\mathbb{M}}

\newcommand{\res}{\operatorname{Res}}
\newcommand{\hs}{\Hat{H}(\phi)}
\newcommand{\hse}{\Hat{H}^e(\phi)}
\newcommand{\hsinput}[1]{\Hat{H}(\phi_{#1})}
\newcommand{\Crit}{\operatorname{Crit}}

\newcommand{\h}{\mathfrak{h}}
\newcommand{\Reach}{\operatorname{Aut}_{\triangle}}
\newcommand{\uReach}{\mathpzc{{Aut}}_{\triangle}}

\newcommand{\Ex}{\operatorname{Exch}}
\newcommand{\Nil}{\operatorname{Nil}_\triangle}
\newcommand{\All}{\Reach}
\newcommand{\uAll}{\uReach}
\newcommand{\uTwist}{\mathpzc{{Sph}}_\triangle}
\newcommand{\uAllo}{\mathpzc{{Aut}}^{0}_\triangle}
\newcommand{\Allo}{{\operatorname{\Aut}}^{0}_\triangle}
\newcommand{\Quadorb}{\operatorname{Quad}_{\heartsuit}}

\newcommand{\barC}{\bar{C}}

\title{Quadratic differentials as stability conditions}
\author{Tom Bridgeland} 
\thanks{During the writing of this paper T.B. was supported by All Souls College, Oxford.}
\address{Tom Bridgeland, School of Mathematics and Statistics, University of Sheffield, Hicks Building, Hounsfield Road, S3 7RH, England.}
\author{Ivan Smith}
\thanks{I.S. was partially supported by a grant from the European Research Council.}
\address{Ivan Smith, Centre for Mathematical Sciences, Cambridge, CB3 0WB, England.}
\begin{document}

\begin{abstract}
We prove that moduli spaces of meromorphic quadratic differentials with simple zeroes on compact Riemann surfaces can be identified with spaces of stability conditions on a class of \CY triangulated categories defined using quivers with potential associated to triangulated surfaces.  We relate the finite-length trajectories of such  quadratic differentials to the stable objects of the corresponding stability condition.
\end{abstract}

\maketitle \thispagestyle{empty}

\begin{footnotesize}
\setcounter{tocdepth}{1}
\tableofcontents
\end{footnotesize}
\parindent0em
\parskip1em


\section{Introduction}\label{Sec:Intro}

In this paper we prove that spaces of stability conditions on a certain class of  triangulated categories  can be identified with  moduli spaces of meromorphic quadratic differentials. The relevant categories  are Calabi-Yau of dimension three (CY$_3$), and are described using quivers with potential associated to triangulated surfaces. The observation that spaces of abelian and quadratic differentials have similar properties to spaces of stability conditions was first made by  Kontsevich and  Seidel several years ago. On the one hand, our results provide some of  the first  descriptions of spaces of stability conditions on  \CY categories, which is the case of most interest in physics.
On the other, they give a precise link  between the  trajectory structure of   flat surfaces and the theory of wall-crossing and Donaldson-Thomas invariants.

Our results  can  also be viewed  as a first step towards a mathematical understanding of the  work of physicists Gaiotto, Moore and Neitzke \cite{GMN1,GMN2}. Their paper \cite{GMN1}   describes a remarkable interpretation of the Kontsevich-Soibelman wall-crossing formula for Donaldson-Thomas invariants   in terms of hyperk{\"a}hler geometry. In the sequel \cite{GMN2} an extended example  is described, relating to parabolic Higgs bundles of rank two. The mathematical objects studied in the present paper are very closely related to their physical counterparts in \cite{GMN2}, and some of our basic constructions are taken directly from that paper. 
We hope to return to the relations with  Hitchin systems and cluster varieties in a future publication.   In another direction, the CY$_3$ categories appearing in this paper also arise as Fukaya categories of certain quasi-projective Calabi-Yau threefolds.  That relation is the subject of a sequel paper \cite{Smith-Quiver}.

In this introductory section we shall first recall some basic facts about quadratic differentials on Riemann surfaces. We then describe the simplest examples of the  categories we shall be studying, before giving a summary of our main result in that case, together with a very brief sketch of how it is proved. We then  state the other version of our result involving quadratic differentials with higher-order poles. We conclude by  discussing the relationship between the finite-length trajectories of a quadratic differential and the  stable objects of the corresponding stability condition.

As a matter of notation, the triangulated categories  we consider here are most naturally labelled by combinatorial data consisting of a smooth surface $\S$ equipped with a collection of marked points $\M\subset\S$, all considered up to diffeomorphism. Initially $\S$ will be closed, but in the second form of our result $\S$ can have non-empty boundary. The quadratic differentials we consider live on Riemann surfaces $S$ whose underlying smooth surface is obtained from $\S$ by collapsing each  boundary component to a  point. To avoid confusion, we shall try to preserve the notational distinction whereby $\S$ refers to a smooth surface, possibly with boundary, whereas $S$ is  always a Riemann surface, usually compact. All these surfaces will be assumed  to be connected.

We fix an algebraically closed field $k$ throughout.


\subsection{Quadratic differentials}

A meromorphic quadratic differential $\phi$ on a Riemann surface $S$ is a meromorphic section of the  holomorphic line bundle $\omega_S^{\tensor 2}$. We emphasize that all the differentials considered in this paper will be assumed to have simple zeroes. Two quadratic  differentials $\phi_1,\phi_2$ on Riemann surfaces $S_1,S_2$  are considered to be    equivalent if there is a holomorphic isomorphism $f\colon S_1\to S_2$ such that $f^*(\phi_2)=\phi_1$.

Let $\S$ be a compact, closed, oriented surface,  with a non-empty  set of  marked  points $\M\subset \S$. We assume that if $g(\S)=0$ then $|\M|\geq 3$. Up to diffeomorphism the pair $(\S,\M)$ is determined by the genus $g=g(\S)$ and the number $d=|\M|>0$ of marked points. We use this combinatorial data to specify a union of strata in the space of meromorphic quadratic differentials; this will be less trivial later when we allow $\S$ to have boundary. 

By a quadratic differential on $(\S,\M)$ we shall mean a pair $(S,\phi)$, where $S$ is a compact and connected Riemann surface  of genus $g=g(\S)$, and $\phi$ is a meromorphic quadratic differential with simple zeroes and exactly $d=|\M|$ poles, each one of order $\leq 2$. Note that every equivalence class of such differentials contains pairs $(S,\phi)$ such that $\S$ is the underlying smooth surface of $S$, and $\phi$ has poles precisely at the points of $\M$.

 A quadratic differential $(S,\phi)$ of this form determines a double cover $\pi\colon \Hat{S}\to S$, called the spectral cover, branched precisely at the zeroes   and simple poles of $\phi$. This cover has the property that
\[\pi^*(\phi)=\psi\tensor \psi\]
 for some globally-defined meromorphic 1-form $\psi$.
We write $\hS^\op\subset \hS$ for the complement of the poles of $\psi$.
The hat-homology group of the differential $(S,\phi)$ is defined to be \[\hs=H_1(\Hat{S}^\op;\Z)^-\]
 where 
 the superscript  indicates  the anti-invariant part for the action of the covering involution. The 1-form $\psi$ is  holomorphic on $\hS^\op$ and anti-invariant, and hence defines a de Rham cohomology class, called the period of $\phi$, which we choose to view as a group homomorphism
 \[Z_\phi\colon \hs\to \C, \qquad \gamma\mapsto \int_\gamma \psi.\]

There is a complex orbifold $\Quad(\S,\M)$ of dimension \[n=6g-6+3d\] parameterizing equivalence-classes of quadratic differentials on $(\S,\M)$. We call a quadratic differential complete if it has no simple poles; such differentials form a dense open subset $\Quad(\S,\M)_0\subset \Quad(\S,\M)$. 

The homology groups $\hs$ form a local system over the orbifold $\Quad(\S,\M)_0$. A slightly subtle point is that this local system does not extend over $\Quad(\S,\M)$, but rather has  monodromy of order 2 around each component of the divisor parameterizing differentials with a simple pole. It therefore defines a local system on an orbifold $\Quadorb(\S,\M)$ which has larger automorphism groups along this divisor. There is a natural map \[\Quadorb(\S,\M)\to \Quad(\S,\M),\]which is an isomorphism over the open subset $\Quad(\S,\M)_0$, and which induces an isomorphism on coarse moduli spaces. Fixing a free abelian group $\Gamma$ of rank $n$, we can also consider an unramified cover 
\[\Quad^\Gamma(\S,\M)\lra \Quadorb(\S,\M)\]
of framed quadratic differentials, consisting of equivalence classes of  quadratic differentials as above, equipped with a  local trivalization $\Gamma \isom \hs$ of the hat-homology local system.

In Section \ref{proof} we shall prove the following result, which is a variation on the usual existence of period co-ordinates in spaces of quadratic differentials. For this we need to assume that $(\S,\M)$ is not a torus with a single marked point.

\begin{thm}
\label{pete}
 The space of framed differentials $\Quad^\Gamma(\S,\M)$ is a complex manifold, and there is a local homeomorphism 
\begin{equation}
\label{per1}\pi\colon \Quad^\Gamma(\S,\M)\lra \Hom_\Z(\Gamma,\C),\end{equation}
obtained by composing the framing and the period. 
\end{thm}

In the excluded case the space $\Quad^\Gamma(\S,\M)$ is not a manifold because it has generic automorphism group $\Z_2$.


\subsection{Triangulations and quivers}
\label{fir}
  Suppose again that $\S$ is a compact, closed, oriented surface  with a non-empty  set of   marked points $\M\subset \S$.    For the purposes of the following discussion   we will assume that if $g(\S)=0$ then $|\M|\geq 5$.

By a non-degenerate ideal triangulation of $(\S,\M)$ we mean  a triangulation  of $\S$ whose vertex set is precisely $\M$ and  in which every vertex has valency at least 3.
To each such triangulation $T$ there is an associated quiver $Q(T)$ whose vertices are  the midpoints of the edges of $T$, and whose arrows are obtained by inscribing a small clockwise 3-cycle inside each face of $T$, as illustrated  in Figure \ref{fig1}. 
\begin{figure}[ht]
\begin{center}
\includegraphics[scale=0.45]{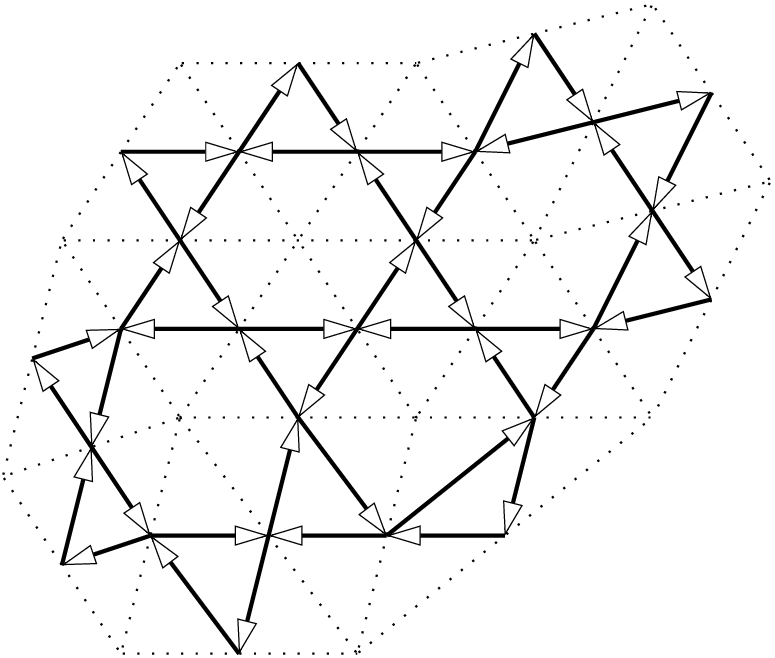}
\end{center}
\caption{Quiver associated to a triangulation.\label{fig1}}
\end{figure}

There are two  obvious systems of cycles in $Q(T)$, namely a clockwise 3-cycle $T(f)$ in each face $f$, and an anticlockwise cycle $C(p)$ of length at least 3 encircling each point $p\in \M$. We  define a potential $W(T)$ on $Q(T)$ by taking the sum
\[W(\T)=\sum_{f} T(f) - \sum_{p}  C(p).\]

Consider the derived category of  the  complete Ginzburg  algebra \cite{Ginzburg,KY} of the quiver with potential $(Q(T),W(T))$ over  $k$, and let $\D(T)$ be the full subcategory consisting of modules with finite-dimensional cohomology. It is a \CY triangulated category of finite type over $k$, and comes equipped with a canonical t-structure, whose  heart $\A(T)\subset \D(T)$   is equivalent to the category of finite-dimensional modules for the completed Jacobi algebra of  $(Q(T),W(T))$.

Suppose that two non-degenerate ideal triangulations $T_i$ are related by a flip,  in which the diagonal of a quadilateral is replaced by its opposite diagonal, as in Figure \ref{stana}. The point of the above definition is that 
the resulting quivers with potential $(Q(T_i),W(T_i))$ are related by a mutation at the vertex corresponding to the edge being flipped; see Figure \ref{stana}. It follows from general results of Keller and Yang \cite{KY} that there is a  distinguished pair of $k$-linear triangulated equivalences $\Phi_\pm\colon \D(T_1)\isom \D(\T_2)$. 

\begin{figure}[ht]
\begin{center}
\includegraphics[scale=0.45]{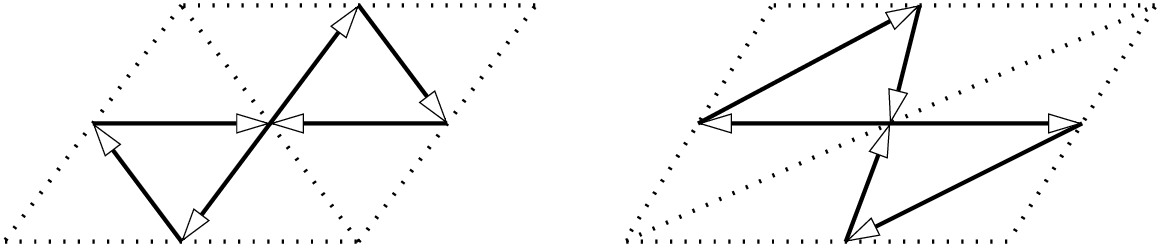}
\end{center}
\caption{Effect of a flip\label{stana}. }
\end{figure}

Labardini-Fragoso \cite{LF1} extended the correspondence between ideal triangulations and quivers with potential so as  to encompass a larger class of  triangulations containing  vertices of valency $\leq 2$. He then proved the much more difficult result that flips    induce mutations  in this more general context.
Since any two ideal triangulations are related by a finite chain of flips, it follows  that up to $k$-linear triangulated equivalence, the category $\D(T)$ is independent of the chosen triangulation.
We loosely use the notation  $\D(\S,\M)$  to denote any triangulated category $\D(T)$  defined by  an ideal triangulation $T$ of  the marked  surface $(\S,\M)$.

 
 \subsection{Stability conditions}
 \label{tw}
A stability condition on a triangulated category $\D$ is a pair $\sigma=(Z,\P)$ consisting of a group homomorphism $Z\colon K(\D)\to \C$ called the central charge, and an $\R$-graded collection of objects  \[\P=\bigcup_{\phi\in\RR}\P(\phi)\subset \D\]  known as the  semistable objects, which together satisfy some axioms (see Section \ref{stabsumm}). 

For simplicity, let us assume that the Grothendieck group $K(\D)$ is free of some finite rank $n$. There is then a complex manifold  $\Stab(\D)$ of dimension $n$ whose points are stability conditions on $\D$ satisfying a further  condition known as the support property. The map 
\begin{equation}\label{per2} \pi\colon \Stab(\D)\lra \Hom_\Z(K(\D),\C)\end{equation}
taking a stability condition to its central charge is a local homeomorphism. The manifold $\Stab(\D)$ carries a natural action of the group $\Aut (\D)$  of triangulated autoequivalences of $\D$. 

Now suppose that  $(\S,\M)$ is a compact, closed, oriented surface with marked points, and let $\D$ be the  \CY triangulated category $\D(\S,\M)$ defined in the last subsection. 
There is  a distinguished connected component \[\Stab_\triangle(\D)\subset \Stab(\D),\] containing stability conditions whose heart is one of the standard hearts $\A(T)\subset \D(T)$ discussed above.
We write \[\Reach(\D)\subset \Aut(\D)\]  for the subgroup of autoequivalences of $\D$ which preserve this component. We also define $\uReach(\D)$ to be the quotient of $\Reach(\D)$ by the subgroup of autoequivalences which act trivially on $\Stab_\triangle(\D)$.

The first form of our main result is

\begin{Theorem}  \label{Thm:Main0} 
Let $(\S,\M)$ be a compact, closed, oriented surface with marked points.  Assume that  one of the following two conditions holds
\begin{itemize}
\item[(a)] $g(\S)=0$ and  $|\M| >5$; 
\item[(b)] $g(\S)>0$ and  $|\M| > 1$. 
\end{itemize}
Then   there is an isomorphism of complex orbifolds \[
\Quadorb(\S,\M)\isom  \Stab_\triangle(\D)/ \,\uReach(\D).
\] 
 \end{Theorem}

The assumption on the number of punctures in the $g(\S)=0$ case of Theorem \ref{Thm:Main0} comes from a similar restriction in a crucial result of Labardini-Fragoso \cite{LF4}. We conjecture that the  conclusion  of the Theorem  holds with the weaker assumptions that $|\M|>1$ and that if $g(\S)=0$ then $|\M|>3$. The case of a once-punctured surface is special in many respects, and we leave it for future research; see Section \ref{therest} for more comments on this. The case  of a three-punctured sphere is also special, and is treated in  Section \ref{threepunctures}.


\subsection{Horizontal strip decomposition}
\label{hori}

The main ingredient in the proof of Theorem \ref{Thm:Main0} is the statement that a generic point of the space $\Quad(\S,\M)$  determines an ideal triangulation of the surface $(\S,\M)$, well-defined up to the action of the mapping class group. We learnt this idea from Gaiotto, Moore and Neitzke's work  \cite[Section 6]{GMN2}, although in retrospect, it is an immediate consequence of well-known results in the theory of quadratic differentials. 

Away from its critical points (zeroes and poles),
a quadratic differential $\phi$ on a Riemann surface $S$  induces a flat metric, together with a foliation known as the horizontal foliation. One way to see this is to  write $\phi=dz^{\tensor 2}$ for some  local co-ordinate $z$, well-defined up to $z\mapsto \pm z+\constant$. The metric is then  given by pulling back the Euclidean metric on $\C$ using $z$, and the horizontal foliation is given by the lines $\Im(z)=\constant$. 

The  integral curves  of the horizontal foliation are called  trajectories. 
The  trajectory structure  near a simple zero and a generic double pole are illustrated in Figure \ref{zeropole}.
\begin{figure}[ht]
\begin{center}
\includegraphics[scale=0.4]{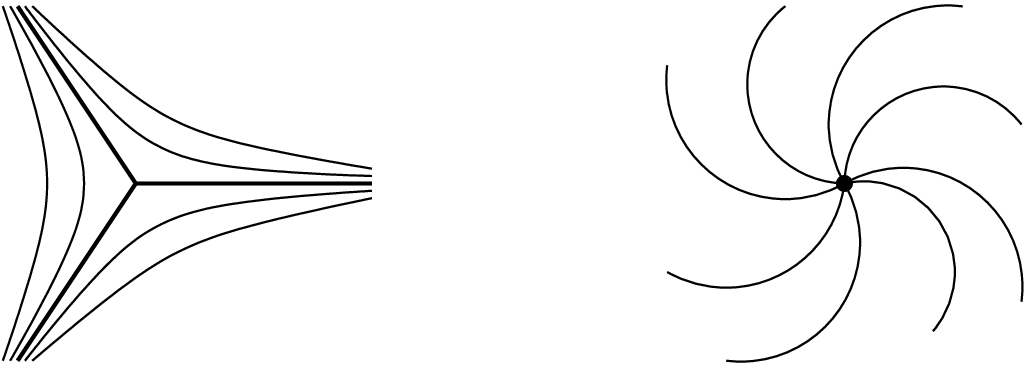}
\end{center}
\caption{Local trajectory structure at a simple zero and a generic double pole.\label{zeropole}}
\end{figure}
Note that generic double poles behave like black holes: any trajectory passing beyond a certain event horizon eventually falls into the pole. Thus for a generic differential one expects all trajectories to tend towards a double pole in at least one direction.

In the flat metric on $S$ induced by $\phi$, any pole  of  order $\geq 2$ lies at infinity.
Therefore, assuming that $S$ is compact, any finite-length trajectory $\gamma$  is either a simple closed curve  containing no critical points of $\phi$,  or is a simple arc which tends to a finite critical point of $\phi$ (a zero or simple pole) at either end. In the first case $\gamma$ is called a closed trajectory, and moves in an annulus of such trajectories known as a ring domain.
In the second case we call $\gamma$  a saddle trajectory. Note that the endpoints of  a saddle trajectory $\gamma$ could well coincide; when this happens  we call $\gamma$  a closed sadddle trajectory.

The  boundary of a ring domain  has two components, and each boundary component usually consists of unions of saddle trajectories. There is one other possibility however: a ring domain may  consist of closed curves encircling a double pole $p$ with real residue; the point $p$ is then one of the boundary components.   We call such ring domains degenerate. 
\begin{figure}[ht]
\begin{center}
\includegraphics[scale=0.4]{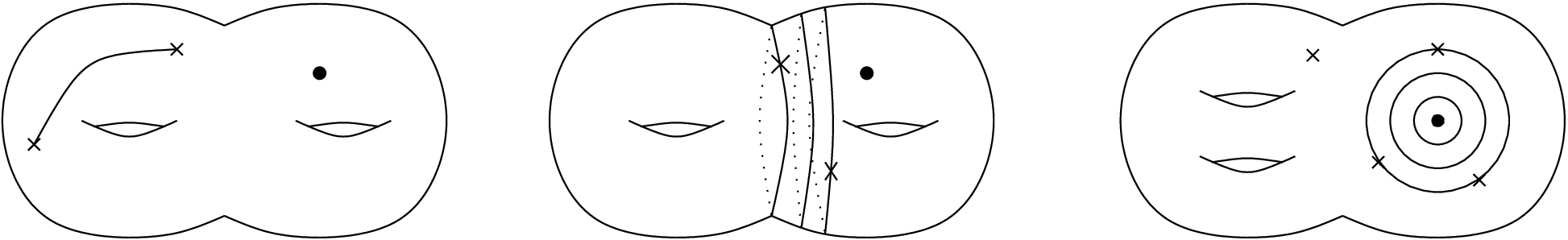}
\end{center}
\caption{A saddle trajectory, a ring domain and a degenerate ring domain.}
\end{figure}

There is a dense open subset  $B_0\subset \Quad(\S,\M)$ consisting of  differentials $(S,\phi)$ with no simple poles and no finite-length trajectories; we call such differentials saddle-free. For saddle-free differentials, each of the three horizontal trajectories leaving a given zero eventually tend towards a double pole. These separating trajectories divide the surface $S$  into a union of cells, known as horizontal strips (see Figure \ref{dualgraph}). 
\begin{figure}[ht]
\begin{center}
\includegraphics[scale=0.4]{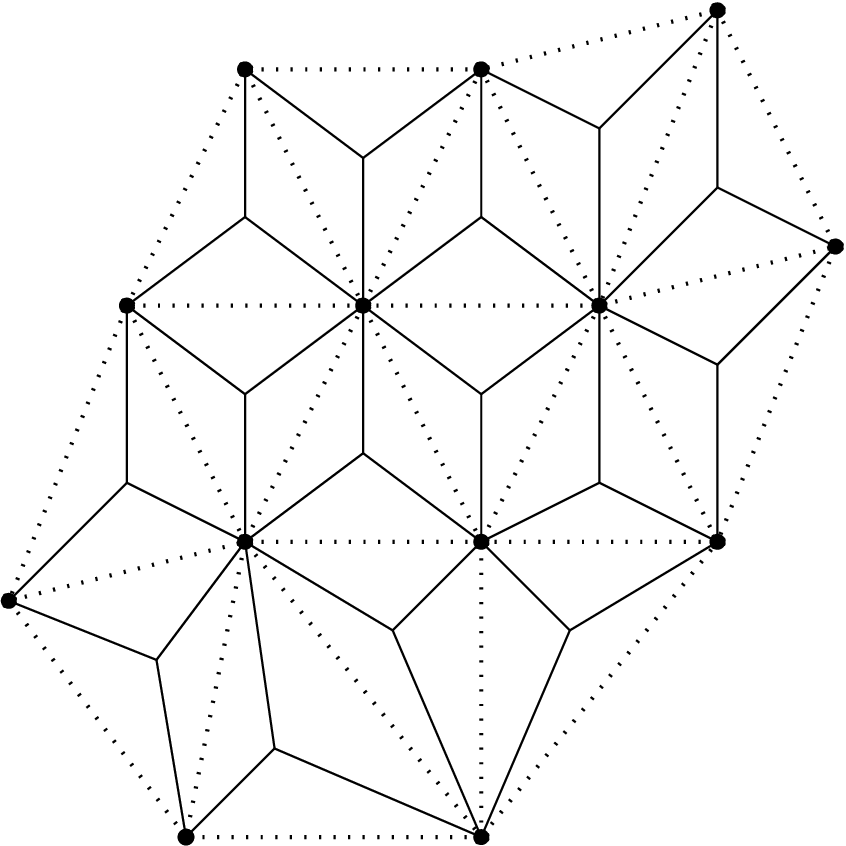}
\end{center}
\caption{The  separating (solid) and generic    trajectories (dotted) for a saddle-free differential; the black dots represent double poles\label{dualgraph}.}
\end{figure}
 Taking a single generic trajectory from each horizontal strip gives a  triangulation  of the surface $S$, whose vertices lie at the poles of $\phi$, and this then induces an ideal triangulation $T$ of the surface $(\S,\M)$,  well-defined up to the  action of the mapping class group. This is what is referred to as the WKB triangulation in \cite{GMN2}. 

The dual graph to the collection of separating trajectories is precisely the quiver $Q(T)$ considered before. In particular,  the vertices of  $Q(T)$ naturally correspond to the horizontal strips of $\phi$.
In each horizontal strip $h_i$ there is a unique homotopy class of arcs $\ell_i$ joining the two zeroes of $\phi$ lying on its boundary. Lifting $\ell_i$ to the spectral cover gives a class $\alpha_i\in \hs$, and taken together, these classes  form a basis. There is thus a natural isomorphism \[\nu\colon K(\D(T))\to \hs,\]
  which sends the class  of the simple module $S_i$ at a vertex of $Q(T)$, to the class $\alpha_i$  defined by the corresponding  horizontal strip $h_i$.

Using the isomorphism $\nu$, the period of $\phi$ can  be interpreted as  a group homomorphism $Z_\phi\colon K(\D(T))\to \C$. More concretely, this is given by
\[Z_\phi(S_i)=2\int_{\ell_i} \sqrt{\phi} \in \C,\]
where the sign of $\sqrt{\phi}$ is chosen so that $\Im Z_\phi(S_i)>0$.  We thus have a triangulated category $\D(T)$, with its canonical heart $\A(T)$, and a compatible central charge $Z_\phi$. This is precisely the data needed to define a stability condition on $\D(\T)$.

We refer to the connected components of the open subset $B_0$ as chambers; the horizontal strip decomposition and the triangulation $T$ are constant in each chamber, although the period $Z_\phi$ varies. As one moves from one chamber to a neighbouring one, the triangulation $T$ can undergo a flip. Gluing the stability conditions obtained from all these chambers using the Keller-Yang equivalences $\Phi_\pm$ referred to  above eventually leads to a proof of Theorem  \ref{Thm:Main0}.


\subsection{Higher-order poles}

We can extend Theorem \ref{Thm:Main0} to cover quadratic differentials with  poles of order $>2$. Such differentials correspond to stability conditions on categories defined by triangulations of surfaces with boundary. For this reason it will  be convenient to  also  index the relevant moduli spaces of differentials by such surfaces, as we now explain.

A marked, bordered surface $(\S,\M)$ is a pair consisting of a compact, oriented, smooth surface $\S$, possibly with boundary, together with a collection of marked points $\M\subset \S$, such that every boundary component of $\S$ contains at least one point of $\M$. The marked points $\bP\subset \M$ lying in the interior of $\S$ are called punctures. We shall always assume that   $(\S,\M)$  is not one of the following:
\begin{enumerate}
\item[(i)] a sphere with $\leq 2$ punctures;
\smallskip
\item[(ii)]an unpunctured disc with $\leq 2$ marked points on its boundary.
\end{enumerate}
These excluded surfaces have no ideal triangulations, and so our theory would be vacuous in these cases.

The trajectory structure of a quadratic differential $\phi$ near a higher-order pole is illustrated in Figure \ref{polefive};
 just as with double poles there is an event horizon beyond which all  trajectories tend to the pole, but at a pole of  order $k+2$ there are, in addition,  $k$ distinguished tangent vectors along which all trajectories enter.

\begin{figure}[ht]
\begin{center}
\includegraphics[scale=0.25]{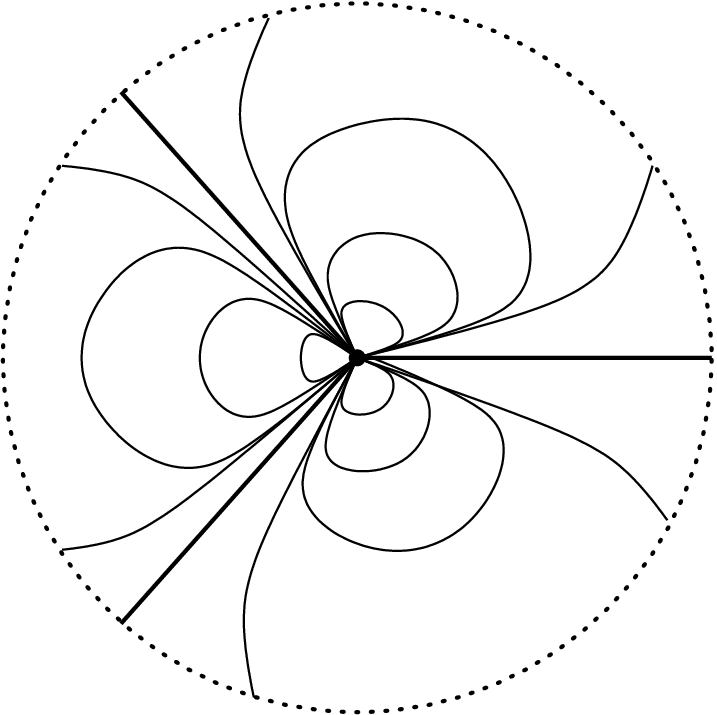}
\end{center}
\caption{Local trajectory structure at a pole of order $5$\label{polefive}.}
\end{figure}

A meromorphic quadratic differential $\phi$ on a compact Riemann surface $S$  determines a marked, bordered surface $(\S,\M)$ by the following construction.   To define the surface $\S$ we take the  underlying smooth surface of $S$ and perform  an oriented real blow-up  at each pole of $\phi$ of order $\geq 3$.  The marked points $\M$ are then the poles of $\phi$ of order $\leq 2$, considered as points of the interior of $\S$, together with the points on the boundary of $\S$ corresponding to the distinguished tangent directions. 

Let us now fix a marked, bordered surface $(\S,\M)$. 
Let $\Quad(\S,\M)$ denote the space of equivalence classes of  pairs $(S,\phi)$, consisting of a compact Riemann surface $S$, together with a meromorphic quadratic differential $\phi$ with simple zeroes,  whose associated marked bordered surface is diffeomorphic to $(\S,\M)$.

More concretely,  the pair $(\S,\M)$ is determined up to diffeomorphism by the genus $g=g(\S)$, the number of punctures $p=|\bP|$, and a collection of integers $k_i\geq 1$ encoding the number of marked points on each boundary component of $\S$. The space $\Quad(\S,\M)$ then consists of equivalence classes of pairs $(S,\phi)$ consisting of a meromorphic quadratic differential  $\phi$ on a compact Riemann surface $S$ of genus $g$, having $p$ poles of order $\leq 2$, a collection of higher-order poles with multiplicities $k_i+2$, and simple zeroes.

The space $\Quad(\S,\M)$ is a complex orbifold of dimension 
\[n=6g-6+3p+\sum_i (k_i+3).\]
We  can define the spectral cover $\pi\colon \hS\to S$, the hat-homology group $\hs$, and the spaces $\Quad^\Gamma(\S,\M)$ and $\Quadorb(\S,\M)$ exactly as before. We can also prove the analogue of Theorem \ref{pete}   in this more general setting.

The theory of  ideal triangulations of marked bordered surfaces has been developed for example in \cite{FST}. The results of Labardini-Fragoso \cite{LF1} apply equally well in this more general situation, so exactly as before, there is a \CY triangulated category $\D= \D(\S,\M)$, well-defined up to $k$-linear equivalence, and a distinguished connected component $\Stab_\triangle(\D)$.  
  
The second form of our main result is

\begin{Theorem}  \label{Thm:Main1} 
Let $(\S,\M)$ be a marked bordered surface  with non-empty boundary. 
Then there is an isomorphism of complex orbifolds \[
 \Quadorb(\S,\M)\isom  \Stab_\triangle (\D) /\,\uReach (\D).
\] \end{Theorem}

There are six degenerate cases which have been suppressed  in the statement of Theorem \ref{Thm:Main1}.  Firstly, if $(\S,\M)$ is one of the following three surfaces
\begin{itemize}
\item[(a)] a once-punctured disc with 2  or 4 marked points on the boundary; 
\item[(b)]  a 
 twice-punctured disc with 2 marked points on the boundary;
\end{itemize}
then Theorem \ref{Thm:Main1} continues to hold, but only if we replace $\uReach(\D)$ by a certain index 2 subgroup $\uReach^{\operatorname{allow}}(\D)$.  The  basic reason for this is that a triangulation $T$ of such a  surface is not determined up to the action of the mapping class group by the associated quiver $Q(T)$. Secondly, if $(\S,\M)$ is one of the following three surfaces
\begin{itemize}
\item[(c)] an unpunctured disc with 3 or 4 marked points on the boundary; 
\item[(d)]  an annulus with one marked point on each boundary component; 
\end{itemize}
then the space $\Quad(\S,\M)$ has a generic  automorphism group which must first be killed to make Theorem \ref{Thm:Main1} hold. These exceptional cases are treated in more detail in Section \ref{therest}.

Particular choices of the data $(\S,\M)$ lead to quivers of interest in representation theory. See Section \ref{applications} for  some examples of this. In particular, we can recover in this way some recent results of T. Sutherland \cite{Suth,Sutherland}, who used different methods to compute the spaces of numerical stability conditions on the categories $\D(\S,\M)$ in all cases in which these spaces are two-dimensional.


\subsection{Saddle trajectories and stable objects}

In the course of proving the Theorems stated above, we will in fact prove a stronger result, which gives a direct correspondence between the finite-length trajectories of a quadratic  differential and the stable objects of the corresponding stability condition.

To describe this correspondence in more detail,  fix a 
 marked bordered surface $(\S,\M)$  satisfying the assumptions of one of
our main theorems, and let $\D=\D(\S,\M)$ be the corresponding triangulated category.
Let $\phi$ be a meromorphic differential on a compact Riemann surface $S$ defining a point $\phi\in \Quad(\S,\M)$,  and  let $\sigma\in \Stab_\triangle(\D)$ be the  corresponding stability condition, well-defined up to the action of the group  $\All(\D)$.
We shall say that the differential $\phi$ is generic if for any two hat-homology classes $\gamma_i\in \hs$
\[\R\cdot Z_\phi(\gamma_1)=\R\cdot  Z_\phi(\gamma_2) \implies \Z\cdot\gamma_1=\Z\cdot \gamma_2.\]
Generic  differentials form a dense subset of $\Quad(\S,\M)$, and for simplicity we shall restrict our attention to these.

 To state the result, let us denote by  $\cM_\sigma(0)$  the moduli space  of objects in $\D$ that are stable in the stability condition $\sigma$ and of phase 0. This space can be identified with a moduli space of stable representations of a finite-dimensional algebra, and hence by work of King \cite{king}, is represented by a   quasi-projective scheme over $k$.

\begin{thm}
\label{greenlees}
Assume that $\phi$ is generic.
Then $\cM_\sigma(0)$ is smooth, and each of its connected components is either a point, or is isomorphic to the projective line $\PP^1$. Moreover,   there are bijections 
\begin{eqnarray*}
\big\{\text{0-dimensional  components of  }\cM_\sigma(0)\big\}&{\leftarrow}{\rightarrow}&\big\{\text{non-closed saddle trajectories of $\phi$}\big\}; \\
\big\{\text{1-dimensional components of  }\cM_\sigma(0)\big\}&\leftarrow\rightarrow&\big\{\text{non-degenerate ring domains  of $\phi$}\big\}. 
\end{eqnarray*}
\end{thm}

Note that with our conventions, all trajectories are assumed to be  horizontal, and  correspond to stable objects of phase 0. In particular, a stability condition $\sigma$ has a stable object of phase 0  precisely if the corresponding differential $\phi$ has a  finite-length trajectory. Stable objects of more general phases $\theta$ correspond in exactly the same way  to finite-length straight arcs which meet the horizontal foliation at a constant angle $\pi \theta$. This more general statement follows immediately from Theorem \ref{greenlees}, because the isomorphisms of our main theorems are compatible with the natural $\C^*$-actions on both sides.

 Standard results in Donaldson-Thomas theory  imply that  the two types of moduli spaces  appearing in  Theorem \ref{th} contribute  $+1$ and $-2$ respectively to the  BPS  invariants, although we do not include the proof of this here. These  exactly match the contributions to the BPS invariants described in \cite[Section 7.6]{GMN2}.  In physics terminology, non-closed saddle trajectories correspond to BPS hypermultiplets, and non-degenerate  ring domains to BPS  vectormultiplets.

It is a standard open question in the theory of flat surfaces to characterise or  
constrain the hat-homology classes which contain saddle  connections. Theorem \ref{greenlees}  relates this to the similar problem of identifying the classes in the Grothendieck group which support  stable objects. Here one has the powerful technology of Donaldson-Thomas invariants and the Kontsevich-Soibelman wall-crossing formula \cite{KS}, which  in principle allows one to determine how the spectrum of stable objects changes as the stability condition varies.   It would be interesting to see whether these techniques can be usefully applied to the theory of flat surfaces.


\subsection {Structure of the paper} The paper   splits naturally into three parts.

The first part, consisting of Sections 2--6, is concerned with spaces of meromorphic quadratic differentials.  Section 2 reviews basic notions concerning quadratic differentials, and introduces orbifolds $\Quad(g,m)$ parameterizing   differentials with simple zeroes and fixed pole orders. Section 3 consists of well-known  material on the trajectory structure of quadratic differentials. Section 4 is devoted to  proving that the period map on $\Quad(g,m)$ is a local isomorphism. Section  5 studies the stratification of the space $\Quad(g,m)$ by the number of separating trajectories. Finally, Section 6 introduces the spaces $\Quad(\S,\M)$ appearing above, in which  zeroes of the differentials are allowed to collide with  the double poles. 

The second part, comprising  Sections 7--9, is concerned with \CY triangulated categories, and more particularly, the categories $\D(\S,\M)$ described above. Section 7 consists of  general material on quivers with potential, t-structures, tilting and stability conditons. Section 8 introduces the basic combinatorial properties of ideal and tagged triangulations. Section 9 contains a more detailed study of the  categories $\D(\S,\M)$, including their autoequivalence groups, and gives a precise correspondence relating t-structures on $\D(\S,\M)$ to tagged triangulations of the  surface $(\S,\M)$.

The geometry and algebra come together in the last part, which comprises Sections 10--12. Section 10 describes the WKB triangulation associated to a saddle-free differential, and the way it changes as one passes between neighbouring chambers.  Section 11 contains the proofs of our main results identifying spaces of stability conditions with spaces of quadratic differentials. We finish in Section 12 with some illustrative examples.

The reader is advised to start with \SS\SS  2--3, the first half of \SS 6,  and \SS \SS 7--9, since these contain the essential definitions and are the least technical. It may also help to look at some of the examples in \SS 12.

\noindent \Acknowledgements    Thanks  most of all to Daniel Labardini-Fragoso, Andy Neitzke and Tom Sutherland, all of whom have been enormously helpful. Thanks too to Sergey Fomin,  Bernhard Keller, Alastair King,  Howard Masur,  Michael Shapiro and Anton Zorich for helpful conversations and correspondence. This paper owes a significant debt to the work of Davide Gaiotto, Greg Moore and Andy Neitzke \cite{GMN2}. 





\section{Quadratic differentials}\label{Sec:Differentials}

  We begin by summarizing some of the  basic properties of  meromorphic quadratic differentials on Riemann surfaces. This material is mostly well-known, although  we were unable to find any references dealing with the moduli spaces of  differentials with higher-order poles that we shall be using. Our standard reference for quadratic differentials is   Strebel's book \cite{Strebel}.

\subsection{Quadratic differentials} \label{Sec:Flat}

 Let $S$ be a Riemann surface, and let  $\omega_S$ denote its holomorphic cotangent bundle. A \emph{meromorphic quadratic differential}  $\phi$ on $S$ is a meromorphic  section  of the line bundle $\omega_S^{\tensor 2}$.  Two such  differentials $\phi_1,\phi_2$ on  surfaces $S_1,S_2$  are said to be    equivalent if there is a biholomorphism  $f\colon S_1\to S_2$ such that $f^*(\phi_2)=\phi_1$.

In terms of a local co-ordinate $z$ on $S$ we can write a quadratic differential $\phi$ as
 \[\phi(z)=\varphi(z)\, dz\tensor dz\]
 with $\varphi(z)$ a  meromorphic function. 
 We write  $\Zer(\phi),\Pol(\phi)\subset S$ for the subsets of zeroes and poles of $\phi$ respectively. The subset  $\Crit(\phi)=\Zer(\phi)\cup\Pol(\phi)$ is the set  of \emph{critical points}  of  $\phi$. 

At a point of $S \setminus \Crit(\phi)$ there is a  distinguished local co-ordinate $w$, uniquely defined up to transformations of the form $w \mapsto \pm w + \constant$, with respect to which
\[
\phi(w)\,  = \, dw \otimes dw.
\]
 In terms of an arbitrary local co-ordinate $z$ we have
$w=\int \sqrt{\varphi(z)}\,dz$.

A quadratic differential $\phi$ determines two structures on $S \setminus \Crit(\phi)$, namely a flat metric (called the $\phi$-metric) and a foliation (the horizontal foliation). The $\phi$-metric is defined locally by pulling back the  Euclidean metric on $\C$ using a distinguished co-ordinate $w$. The  horizontal foliation is given in terms of a distinguished co-ordinate by the lines  $\Im (w)=\constant$.   

The $\phi$-metric  and the horizontal foliation  on $S\setminus \Crit(\phi)$  together determine both the complex structure on $S$ and the  differential $\phi$. 
  Note that the set of quadratic differentials on a fixed surface $S$ has  a natural $S^1$ -action given by scalar multiplication : $\phi \mapsto e^{i\pi\theta}\cdot \phi$. This action has no effect on the $\phi$-metric, but alters which in the circle of foliations defined by $\Im (w/e^{i\pi\theta}) = \constant$ is regarded as being horizontal.

In terms of a local co-ordinate $z$ on $S$,  the length of a smooth path $\gamma$  in the $\phi$-metric  is 
\begin{equation}
\label{old}
\ell_\phi(\gamma)\ = \ \int_{\gamma} \, |\varphi(z)|^{1/2} \, |dz|.
\end{equation}
It is important to divide the critical set into a disjoint union \[\Crit(\phi)=\Crit_{<\infty}(\phi)\cup \Crit_\infty(\phi),\]
where $\Crit_{<\infty}(\phi)$ consists of \emph{finite critical points}, namely zeroes and simple poles, and $\Crit_\infty(\phi)$ consists of \emph{infinite critical points}, that is poles of order $\geq 2$.
We write  
 \begin{equation*} 
 S^\op=S\setminus \Crit_{\infty}(\phi)
 \end{equation*} for the complement of the infinite critical points.  

Note that the integral \eqref{old} is well-defined for curves passing through points of $\Crit_{<\infty}(\phi)$. This gives the surface $S^\op$   the structure of a metric space, in which the distance between two points $p,q\in S^\op$ is
  the infimum of the lengths of smooth curves in $S^\op$ connecting $p$ to $q$. The  topology on $S^\op$ defined by this metric   agrees with the standard one induced from the surface $S$.  


\subsection{GMN differentials}

All the quadratic  differentials considered in this paper live on compact surfaces and have simple zeroes and at least one pole. Since it will be convenient to eliminate certain degenerate situations we make the following definition.

\begin{defn}
\label{gm}
A GMN differential is a meromorphic quadratic differential $\phi$ on a compact, connected Riemann surface $S$ such that
\begin{itemize}
\item[(a)] $\phi$ has simple zeroes,
\item[(b)] $\phi$ has at least one pole,
\item[(c)] $\phi$ has at least one finite critical point.
\end{itemize}
\end{defn}

Condition (c)   excludes polar types  $(2,2)$ and $(4)$ in genus 0; differentials of these  types have unusual trajectory structures, and infinite automorphism groups.

Given a GMN differential $(S,\phi)$ we write $g$ for the genus of the surface $S$ and $d$ for  the number of poles of $\phi$. 
 The  \emph{polar type} of $\phi$ is the unordered collection of $d$ integers $m=\{m_i\}$ giving the orders of the poles of $\phi$.
We define
\begin{equation}
\label{n}n=6g-6+\sum_{i=1}^d (m_i+1),\end{equation}
A GMN differential $(S,\phi)$ is said to be \emph{complete} if $\phi$ has no simple poles, or in other words, if all $m_i\geq 2$. This is exactly the case in which the $\phi$-metric on $S\setminus\Pol(\phi)$   is complete. At the opposite extreme, the differential $(S,\phi)$  is said to have   \emph{finite area} if $\phi$ has only  simple poles, that is if  all $m_i=1$.


\subsection{Spectral cover and periods}
\label{go}

Suppose that $\phi$ is a 
GMN differential on a compact Riemann  surface $S$, with poles of order $m_i$ at points $p_i\in S$.  We can alternatively view $\phi$ as a holomorphic section
 \begin{equation}
 \label{g}{\varphi}\in H^0(S,\omega_S(E)^{\tensor 2}), \qquad E=\sum_i  \Big\lceil\frac {m_i}{2} \Big\rceil \cdot p_i,\end{equation}
 with simple zeroes at both the zeroes  and the odd order poles of $\phi$. The \emph{spectral cover}\footnote{The terminology ``spectral cover" fits with that used in the literature on Higgs bundles, cf. \cite{Hitchin}.}  of $S$ defined by   $\phi$  is the compact Riemann surface
 \[\hS=\big\{(p,l(p)): p\in S,\, l(p)\in L_p\text{ such that }l(p)\tensor l(p)={\varphi}(p)\big\}\subset L,\]
where $L$ is the total space of the line bundle $\omega_{S}(E)$.
This is a manifold because ${\varphi}$ has simple zeroes.

The obvious projection map $\pi\colon \hS\to S$ is a double cover, branched precisely over the zeroes and the odd order poles of the original meromorphic differential $\phi$. There is a covering involution $\tau\colon \hS\to \hS$,  commuting with the map $\pi$. The surface $\hS$ is connected because Definition \ref{gm} implies that $\pi$  has at least one branch point.  

We define the hat-homology group of the differential $\phi$ to be
\[\hs=H_1(\hS^\op;\Z)^-,\]
 where  $\hS^{\op}=\pi^{-1}(S^\op)$, and the superscript denotes the anti-invariant part for the action of the covering involution $\tau$.

\begin{lemma}
\label{train}

The group $\hs$ is free of rank $n$ given by \eqref{n}. \end{lemma}

\begin{pf}  The Riemann-Hurwitz formula applied to the spectral cover $\pi\colon \hS\to S$ implies that
\begin{equation}\label{hatg}2\hat{g}-2=2(2g-2)+(4g-4+\sum_{i=1}^d m_i )+ (d-e),\end{equation}
where $\hat{g}$ is the genus of $\hS$, and $e$ is the number of even $m_i$. The group $H_1(S^\op;\Z)$ is free of rank
$2g+d-s-1$, where $s$ is the number of simple poles. Similarly, using equation \eqref{hatg}, and noting that each even order pole has two inverse images in $\hS$, the  group $H_1(\hS^\op;\Z)$ is free of rank \[r=2\hat{g}+d+e-s-1=8g-6+\sum_{i=1}^d m_i + 2d-s-1.\]
Since the invariant part of $H_1(\hS^\op;\Z)$ can be identified with $H_1(S^\op;\Z)$, the anti-invariant part $H_1(\hS^\op;\Z)^-$ is therefore free of rank $n$.
\end{pf}

The spectral cover $\Hat{S}$ comes equipped with a tautological section $\psi$ of the line bundle $\pi^*( \omega_S(E))$ satisfying $\pi^*(\varphi)=\psi\tensor \psi$ and $\tau^*(\psi)=-\psi$.  There is a canonical map $\eta\colon \pi^*(\omega_S)\to \omega_{\Hat{S}}$ and we can form the composition
\[\OO_{\Hat{S}} \lRa{\psi} \pi^*(\omega_S(E)) \lRa{\eta(\Hat{E})} \omega_{\Hat{S}}(\Hat{E}),\]
where  $\Hat{E}=\pi^{-1}(E)$. This defines  a meromorphic 1-form on $\Hat{S}$, which we  also denote   by $\psi$.

 Since the canonical map $\eta$ vanishes at the branch-points of $\pi$, the differential  $\psi$ is regular at the inverse images of the simple poles of $\phi$, and hence restricts to a  holomorphic 1-form on the open subsurface $\hS^\op$. 
By construction $\psi$ is anti-invariant for the action of the covering involution $\tau$, and therefore defines a de Rham cohomology class
\[[\psi]\in H^1(\hS^\op;\C)^-\]
 called the period of $\phi$.  We choose to view this instead as a group homomorphism
\[Z_\phi\colon \hs\to \C.\]


\subsection{Intersection forms} \label{if} Consider a GMN differential $\phi$ on a Riemann surface $S$, and its spectral cover $\pi\colon \hS\to S$. Write \[\Hat{D}_\infty=\pi^{-1}(\Crit_\infty(\phi)).\] Thus $\hS^\op=\Hat{S}\setminus \Hat{D}_\infty$. 
There are canonical maps of homology groups
\[ H_1(\hS^\op;\Z) =H_1(\hS\setminus \Hat{D}_\infty;\Z) \lRa{g} H_1(\hS;\Z)\lRa{h} H_1(\hS,\Hat{D}_\infty;\Z).\]

The intersection form on $H_1(\hS;\Z)$ is a non-degenerate, skew-symmetric pairing, and induces a degenerate skew-symmetric form
\[H_1(\hS^\op;\Z)\times H_1(\hS^\op;\Z)\to \Z,\]
 which we also call the intersection form, and write as $(\alpha,\beta)\mapsto \alpha\cdot \beta$.
On the other hand, Lefschetz duality gives a non-degenerate pairing
\begin{equation}
\label{wine}
\langle-,-\rangle \colon  H_1(\hS\setminus \Hat{D}_\infty;\Z)\times H_1(\hS,\Hat{D}_\infty;\Z)\to \Z.\end{equation}
These bilinear forms restrict to the anti-invariant eigenspaces for the actions of the covering involutions. 

For each pole $p\in S$ of $\phi$ of even order  there is an associated \emph{residue  class} \[ \beta_p\in H_1(\hS^\op;\Z)^-,\]
well-defined up to sign. It is  obtained by taking the inverse image  under $\pi$ of  a small  loop  in $S^\op$ encircling the point $p$, and then orienting the two connected components so that the resulting class is anti-invariant.

  The \emph{residue} of $\phi$ at $p$ is defined to be
\begin{equation} \label{Eqn:Residue}
\res_p(\phi)=Z_\phi(\beta_p) = \pm 2 \int_{\delta_p} \sqrt{\phi},
\end{equation}
and is well-defined up to sign.

\begin{lemma}
\label{mov}
The classes $\beta_p\in H_1(\hS^\op;\Z)^-$ are a $\mathbb{Q}$-basis for the kernel of the intersection form.
\end{lemma}

\begin{pf} If $p\in S$ is an even order pole of $\phi$, let $\{s_p,t_p\}$ be the classes in $H_1(\hS^\op;\Z)$ defined by small clockwise loops around the two inverse images of $p$ in the spectral cover $\hS$. Similarly, if $p\in S$ is a pole of odd order $\geq 3$,  let $u_p\in H_1(\hS^\op;\Z)$ be the class defined by a small loop around the single inverse image of $p$.   Standard topology of surfaces shows that there is an exact sequence
\[0\lra \Z \lRa{i}  \Z^{\oplus k}\lRa{f} H_1(\hS^\op;\Z)\lRa{h} H_1(\hS;\Z)\lra 0,\]
where the map $h$ is induced by the inclusion $\hS^\op\subset \hS$, the map $f$ sends the generators to the classes  $s_p$, $t_p$ and $u_p$ respectively, and the image of $i$ is spanned by the element $(1,1,\dots ,1)$.
 
  The covering involution exchanges $s_p$ and $t_p$, and  fixes $u_p$, and we have $\beta_p=\pm(s_p-t_p)$.
Since the image of the map $i$ lies in the invariant part of $H_1(\hS;\Z)$, the  elements $\beta_p$ are linearly independent.
The intersection form  on $H_1(\hS;\Z)^-$ is non-degenerate, so the kernel of the induced form on $H_1(\hS^\op;\Z)^-$  is precisely the kernel of the surjective map  \[h^-\colon H_1(\hS^\op;\Z)^-\to H_1(\hS;\Z)^-.\] The group $H_1(\hS;\Z)^-$ has rank $2(\hat{g}-g)$, which by \eqref{hatg}   is equal to  $n-e$, where $e$ is the number of even order poles of $\phi$. Thus the kernel of $h^-$ is spanned over $\QQ$  by the $e$ elements $\beta_p$. 
\end{pf}


\subsection{Moduli spaces}

We now consider moduli spaces of GMN differentials of fixed polar type.
For this purpose we fix a genus $g\geq 0$ and an unordered   collection of $d\geq 1$ positive integers $m=\{m_i\}$.

Define  $\Quad(g,m)$ to be the set of equivalence-classes of 
 pairs $(S,\phi)$ consisting 
 of a compact, connected Riemann surface  $S$ of genus $g$, equipped with a GMN differential $\phi$ having   polar type $\m=\{m_i\}$.

\begin{prop}
\label{rain}
The space $\Quad(g,\m)$ is either empty, or is a connected complex orbifold of dimension $n$ given by \eqref{n}.
\end{prop}

\begin{pf}
Let $\Mod(g,d)$ be the moduli stack of compact Riemann surfaces of genus $g$ with an ordered set of $d$ marked points $(p_1,\cdots,p_d)$. This is a smooth, connected algebraic stack of finite type over $\C$.   Choose an ordering of the integers $m_i$, and let $\Sym({\m})\subset \Sym(d)$ be the subgroup of the symmetric group consisting of permutations $\sigma$ such that $m_{\sigma(i)}=m_i$.

At each point of $\Mod(g,d)/\Sym({\m})$  there is a Riemann surface $S$ equipped with a  well-defined divisor $D=\sum_i m_i p_i$.
The spaces  of global sections $H^0(S,\omega_S^{\tensor 2}(D))$ fit together to form  a vector bundle 
\begin{equation}
\label{hs}
\mathcal{H}(g,m)\lra \Mod(g,d)/\Sym(m).\end{equation}
To see this, note first that if $g=0$ then we can assume that the divisor $D$ has degree at least 4, since otherwise the vector spaces are all zero, and the space $\Quad(g,m)$ is empty. 
 Serre duality therefore gives
\[H^1(S,\omega_S^{\tensor 2}(D))\isom H^0(S,\omega_S(D)^\vee)^*=0\]
which proves the claim.
It then follows using Riemann-Roch  that the rank of the bundle  \eqref{hs} is $3g-3+\sum_{i=1}^d  m_i$.
 
The stack $\Quad(g,\m)$ is  the Zariski open subset of  $\mathcal{H}(g,m)$ consisting of sections with simple zeroes disjoint from the points $p_i$.  Since $\Mod(g,d)$ is connected of dimension $3g-3+d$, the stack $\Quad(g,\m)$ is either empty, or is smooth and connected of  dimension $n$.

The final step is to show that the automorphism groups of the relevant quadratic differentials are finite. This claim is clear if $g\geq 1$ or $d\geq 3$, because the same property holds for $\Mod(g,d)$ (a curve of genus $g\geq 2$ has a finite automorphism group; a curve of genus $1$ has finitely many automorphisms fixing a given point). When $g=0$ the claim is also clear if the total number of critical points is $\geq 3$. Since there is at least one pole, and the number of zeroes is $\sum m_i-4$, the only other possibilities are polar types  $(1,3)$, $(4)$, $(5)$ and $(2,2)$.

In the first three of these cases there is a single quadratic differential up to equivalence, namely $\phi= z^k\, dz^{\tensor 2}$ with $k=-1,0,1$ respectively.  The corresponding automorphism groups are  $\{1\}$, $\Z_2\ltimes\C$ and $\Z_3$ respectively.  In the remaining case $(2,2)$ the possible differentials are  $\phi=r\,dz^{\tensor 2}/z^2$ for $r\in \C^*$. Each of these differentials has automorphism group $\Z_2\ltimes\C^*$. By Definition \ref{gm}(c), a GMN differential  must have a zero or a simple pole; this exactly excludes the troublesome cases $(2,2)$ and $(4)$.  
\end{pf}

\begin{Example}
\label{nosu}
Consider the case $g=1, m=(1)$. The corresponding space  $\Quad(g,m)$ is empty, even though the  expected dimension is $n=2$. Indeed, this space parameterizes pairs $(S,\phi)$, where $S$ is a Riemann surface of genus 1, and $\phi$ is a meromorphic differential on $S$ having only a simple pole. On the surface $S$  the bundle $\omega_S$ is trivial, so $\phi$ defines a meromorphic function  with a single simple pole. The Riemann-Roch theorem shows  that no such function exists. 
\end{Example}

We shall often abuse notation by referring to the points of the space $\Quad(g,m)$ as GMN differentials, and by denoting such a point simply by $\phi\in \Quad(g,m)$. This is shorthand for the statement that $\phi$ is a GMN differential on a compact Riemann surface $S$, such that the equivalence class of the pair $(S,\phi)$ defines a point of the space $\Quad(g,m)$.

The homology groups $H_1(S^\op;\Z)^-$ form a local system over the orbifold $\Quad(g,\m)$ because we can realise the spectral cover construction in families, and the Gauss-Manin connection gives a flat connection in the resulting bundle of anti-invariant homology groups. 
Often in what follows we will be studying a small analytic neighbourhood \[\phi_0\in U\subset \Quad(g,m)\] of a fixed differential $\phi_0$. Whenever we do this we will tacitly assume that $U$ is contractible, and  use the Gauss-Manin connection to identify the hat-homology groups of all differentials in $U$.


\subsection{Framings and the period map}
\label{fr}

As in  the last section, we fix a genus $g\geq 0$ and a  collection of $d\geq 1$ positive integers $m=\{m_i\}$.  Let us also fix a free abelian group $\Gamma$ of rank $n$ given by \eqref{n}. 

As before, we consider pairs $(S,\phi)$ consisting of  a Riemann surface $S$ of genus $g$, equipped with a GMN differential $\phi$ of polar type $m=\{m_i\}$. 
A $\Gamma$-framing of such a pair  $(S,\phi)$ is  an isomorphism of groups
\[\theta\colon \Gamma\to  \hs.\]

Suppose $(S_i,\phi_i)$ for $i=1,2$  are two quadratic differentials as above, and $f\colon S_1\to S_2$ is an  isomorphism such that $f^*(\phi_2)=\phi_1$. Then $f$ lifts to an isomorphism $\Hat{f}\colon \hS_1^\op\to \hS_2^\op$, which is unique if we insist that it also satisfies $\Hat{f}^*(\psi_2)=\psi_1$, where $\psi_i$ are the distinguished  1-forms defined in Section \ref{go}. 

Let $\Quad^\Gamma(g,\m)$ be the set of equivalence classes of  triples $(S,\phi,\theta)$ consisting of a  compact, connected   Riemann surface $S$ of genus $g$  equipped with a GMN  differential $\phi$ of polar type $m=\{m_i\}$ together with a $\Gamma$-framing $\theta$.
We define triples $(S_i,\phi_i,\theta_i)$  to be equivalent if there is an isomorphism  $f\colon S_1\to S_2$  such that $f^*(\phi_2)=\phi_1$ and  such that the distinguished lift $\Hat{f}$ makes the following diagram commute
\begin{equation}\label{look}\xymatrix@C=1.5em{ &\Gamma\ar[dl]_{\theta_1}\ar[dr]^{\theta_2}\\ \hsinput{1} \ar[rr]^{\Hat{f}_*}  && \hsinput{2}
 }\end{equation}

We can define families of framed differentials in the obvious way, and  the forgetful map \begin{equation}\label{oo}\Quad^\Gamma(g,\m)\to \Quad(g,\m)\end{equation}  is then an unbranched cover. Thus the set  $\Quad^\Gamma(g,\m)$ is naturally a complex orbifold. 
The group $\Aut(\Gamma)$ of  automorphisms of the group $\Gamma$ acts on  $\Quad^\Gamma(g,\m)$, and the quotient orbifold is precisely $\Quad(g,\m)$.
Note that $\Quad^\Gamma(g,m)$ will not usually be connected, because the monodromy of the local system of hat-homology groups preserves the intersection form, and hence cannot relate all different framings of a given differential. But since all  such framings \emph{are} related by the action of  $\Aut(\Gamma)$, the  different connected components of $\Quad^\Gamma(g,m)$ are  all isomorphic.

The period of a framed GMN differential $(S,\phi,\theta)$ can  be viewed as a map  $Z_\phi\circ \theta\colon \Gamma \to \C$. This gives a well-defined \emph{period map}
\begin{equation}
\label{periodmap}\pi\colon \Quad^\Gamma(g,\m) \to \Hom_\Z(\Gamma,\C).\end{equation}
In Section \ref{per} we shall prove that, with the exception of the six special cases considered in the next section, the space  $\Quad^\Gamma(g,m)$ is  a complex manifold, and  the period map $\pi$ is a local homeomorphism.


\subsection{Generic automorphisms}

In certain special cases  the orbifolds $\Quad(g,m)$  and $\Quad^\Gamma(g,m)$ have non-trivial generic automorphism groups. In this section we classify the polar types when this  occurs.

\begin{lemma}
\label{fatman}
The generic automorphism group of a point of $\Quad(g,m)$ is trivial, with the  exception of the  polar types \[(5); \quad (6);\quad (1,1,2);\quad (3,3);\quad (1,1,1,1),\]
in genus $g=0$, and the polar type $m=(2)$ in genus $g=1$.
\end{lemma}

\begin{pf}
Suppose first that if $g=0$ then $d\geq 5$, and that if $g=1$ then $d\geq 2$. 
With these assumptions the  stack $\Mod(g,d)/\Sym(d)$ parameterizing compact Riemann surfaces of genus $g$  
with an unordered collection of  $d$  marked points has trivial generic automorphism group.\footnote{Consider the case when $g\geq 2$. In order for the automorphism group of a marked curve to be non-trivial the points $p_i$ must be permuted by some automorphism of the curve. Since the automorphism group of  such a curve is finite \cite[Ex. IV.5.2]{Har} this is a non-generic condition. The statement in genus $1$ is similar using the set of points $\{p_i-p_j\}$ and the fact that the group of automorphisms modulo  translations is finite. The genus 0 case is easily dealt with explicitly.}  The same is therefore true of the stack $\Mod(g,d)/\Sym(m)$ appearing in the proof of Proposition \ref{rain}. The space $\Quad(g,m)$ is an open  
subset of a vector bundle over this stack, so again, the generic automorphism group is trivial.

Consider the case $g=1$ and $d=1$. The stack $\Quad(g,m)$ then parameterizes pairs consisting of a Riemann surface $S$ of genus 1,   together with a meromorphic function on $S$ with simple zeroes and a single pole, necessarily of order $m\geq 2$.  For a generic such surface $S$, the group of automorphisms preserving the pole is generated by a single involution,  and using Riemann-Roch it is easily seen that if $m\geq 3$ then the zeroes of the generic such function are not permuted by this involution.

When $g=0$  Riemann-Roch shows that there exist differentials with any given configuration of zeroes and poles, providing only that the number $k$ of zeroes is equal to $\sum m_i - 4$. Thus if a generic point  $\phi\in \Quad(0,m)$ has non-trivial automorphisms, then $|\Crit(\phi)|\leq 4$. Moreover, if $|\Crit(\phi)|=4$ then the critical points must consist of two pairs  of the same type, since the generic automorphism group of $\Mod(0,4)/\Sym(4)$ acts on the marked points via permutations of type $(ab)(cd)$ (see e.g. \cite[Section 2.5]{JS}). If $|\Crit(\phi)|=3$ then at least two of the critical points must be of the same type.

Suppose that the generic point of $\Quad(0,m)$ does have non-trivial automorphisms. Since there is at  
least one pole, we must have $0 \leq k \leq 3$. We cannot have  $k = 3$   since there  
would then be 4 critical points whose types do not match in pairs. If $k = 2$  
 there must be two poles of the same degree, giving  
the $(3,3)$ case, or a single pole, giving the $(6)$ case. If $k = 1$  there  
must be just one pole, which 
 gives the case $(5)$, since if there were 2 poles they would have to have the same degree. Finally, if $k =  
0$ we get the cases $(1,1,2)$ and $(1,1,1,1)$, since the cases $(2,2)$ and $(4)$ have already been excluded by the defintion of a GMN differential, and the case $(1,3)$ leads to a single differential with trivial automorphism group, as discussed in the proof of Proposition \ref{rain}.
\end{pf}

\begin{examples}
\label{spec}
We consider  differentials $(S,\phi)\in \Quad(g,m)$ corresponding to some of the exceptional cases  in the statement of Lemma \ref{fatman}.

\begin{itemize}

\item[(a)] 
Consider the case $g=0$ and $m=(1,1,2)$. Taking the simple poles to be at $\{0,\infty\}\in \PP^1$ we can write any such differential in the form \[\phi(z)=\frac{c\,  dz^{\tensor 2}}{z(z-1)^2}\] for some $c\in \C^*$. Thus $\phi$ is invariant under the automorphism $z\mapsto 1/z$. The spectral cover $\hS$ is again $\PP^1$ with co-ordinate $w=\sqrt{z}$ and covering involution $w\mapsto -w$.  The automorphism $z\mapsto 1/z$ lifts to the  automorphism $w\mapsto 1/w$ of the open subsurface $\hS^\op=\PP^1\setminus \{\pm 1\}$ and acts trivially on the hat-homology group, which is $H_1(\hS^\op;\Z)=\Z$. Thus every element of $\Quad^\Gamma(g,m)$ has automorphism group $\Z_2$. 
\smallskip

\item[(b)]
Consider the case  $g=0$, $m=(3,3)$. Any such differential  is of the form 
\[\phi(z)=(tz+2s+tz^{-1})\, \frac{dz^{\tensor 2}}{z^2}, \]
for constants $s\in \C$ and $t\in \C^*$ with $s \pm t\neq 0$, and is invariant under $z\mapsto 1/z$. The  spectral cover $\hS$ has genus 1.  The open subset $\hS^\op$ is the complement of 2 points, the inverse images of the poles of $\phi$. The  automorphism $z\mapsto 1/z$ of $\PP^1$ lifts to a translation by a 2-torsion point  of $\hS$. It acts trivially on the hat-homology group, which is  $H_1(\hS;\Z)=\Z^{\oplus 2}$. Thus every point of $\Quad^\Gamma(g,m)$ has automorphism group $\Z_2$.

\smallskip
\item[(c)]

Consider the case  $g=0$, $m=(1,1,1,1)$. Such differentials are of the form \[\phi(z)=\frac{dz^{\tensor 2}}{p_4(z)},\] where $p_4(z)$ is a monic polynomial of degree 4 with distinct roots, and are invariant under any automorphism of $\PP^1$ permuting these roots. The  spectral cover $\hS$ has  genus 1. The  automorphisms of $\PP^1$ preserving $\phi$ lift to translations by 2-torsion points of $\hS$. These automorphisms  act trivially on  the hat-homology group, which is  $H_1(\hS;\Z)=\Z^{\oplus 2}$. Thus every point of $\Quad^\Gamma(g,m)$ has automorphism group $\Z_2^{\oplus 2}$. 
\end{itemize}
\end{examples}

In each of the other cases of Lemma \ref{fatman} the orbifold $\Quad^\Gamma(g,m)$ also has non-trivial generic automorphism group. The case  $g=0$, $m=(5)$ is elementary, and the case $g=0$, $m=(6)$ is very similar to Example \ref{spec}(a). The case $g=1$, $m=(2)$ is treated in Example \ref{ex} below.





\section{Trajectories and geodesics}
\label{hor}

In this section we  focus on  the global   trajectory structure of  a fixed quadratic differential, and  the basic properties of  the geodesic arcs of the associated flat metric. This material is all well-known, but since it  forms the basis for much of what follows we thought it worthwhile to give a fairly detailed treatment. The reader can find proofs and further explanations in   Strebel's book \cite{Strebel}.

 \subsection{Trajectories}
 \label{l}

 Let $\phi$ be a meromorphic quadratic differential on a compact Riemann surface $S$.   A \emph{straight arc} in $S$   is a smooth path $\gamma\colon I\to S\setminus \Crit(\phi)$, defined on an open interval $I\subset \R$, which makes a constant angle $\pi\theta$ with the horizontal foliation. In terms of a distinguished local co-ordinate $w$ as in Section \ref{Sec:Flat} the condition is  that the function $\Im (w/e^{i\pi\theta})$ should be  constant along $\gamma$. 
 The phase $\theta$ of a straight arc is a well-defined element of $\R/\Z$; in the case  $\theta=0$ the arc is said to be \emph{horizontal}.

We make the convention that all straight arcs are parameterized by arc-length in the $\phi$-metric.  Straight arcs differing by a reparameterization (necessarily of the form $t\mapsto \pm t+\constant$) will be regarded as being the same. A straight arc is  called \emph{maximal} if it is not the restriction of a straight arc defined on a larger interval.
 A  maximal horizontal straight arc is called a  \emph{trajectory}. Every point of $S\setminus \Crit(\phi)$ lies on a unique trajectory, and any two  trajectories are either disjoint or coincide.

 We define a \emph{saddle trajectory} to be a trajectory $\gamma$ whose domain of definition is a finite  interval $(a,b)\subset \R$. Since $S$ is compact, we can then extend $\gamma$  to a continuous path $\gamma\colon [a,b]\to S$, whose  endpoints $\gamma(a)$ and $\gamma(b)$ are   finite critical points of $\phi$. We tend not to distinguish between the saddle trajectory $\gamma$ and its closure. By a \emph{closed saddle trajectory} we mean a saddle trajectory whose  endpoints  coincide.
 
More generally,  a \emph{saddle connection} is a maximal straight arc of some phase $\theta$ whose domain of definition is a finite interval.  Thus a saddle trajectory is a horizontal saddle connection, and a  saddle connection  of phase $\theta$ is a saddle trajectory   for the rotated differential $e^{-i\pi\theta}\cdot \phi$.

If  a trajectory  $\gamma$ intersects itself, then it must be periodic, and have  domain $I=\R$. In this situation we usually restrict the domain of $\gamma$ to a primitive period $[a,b]\subset \R$, and refer to $\gamma$ as a \emph{closed trajectory}. By a \emph{finite-length trajectory} we mean either a closed trajectory or a  saddle trajectory.


\subsection{Hat-homology classes}

Let us again fix a meromorphic quadratic differential $\phi$ on a compact Riemann surface $S$.
 The inverse image of the horizontal foliation of $S\setminus\Crit(\phi)$ under the covering map $\pi$ defines a horizontal foliation on $\hS\setminus\pi^{-1} \Crit(\phi)$. In more detail, the  1-form $\psi$  of Section \ref{go} can be written  locally as $\psi=d\hat{w}$, and the  horizontal foliation of $\hS$ is then given by the lines $\Im (\hat{w}) =\constant$. This   foliation can be canonically oriented by insisting
that $\psi$ evaluated on the tangent vector to the oriented foliation  should lie in $\R_{>0}$ rather than $\R_{<0}$.  Note that since $\psi$ is anti-invariant, the covering involution  $\tau$ preverses the horizontal foliation on $\hS$, but reverses its orientation.

Suppose that $\gamma\colon [a,b]\to S$ is a finite-length trajectory. The inverse image $\pi^{-1}(\gamma)$  is then a closed curve in the spectral cover $\hS$, which could be disconnected (if $\gamma$ is a closed trajectory), or singular (if $\gamma$ is a closed saddle trajectory, see Figure \ref{Fig:Imprimitive}). In all cases we orient $\pi^{-1}(\gamma)$ according to the orientation discussed in the previous paragraph.  Since the covering involution flips this orientation, we obtain a  class $\Hat{\gamma}\in \hs$  called  the \emph{hat-homology} class\footnote{With this definition it is not necessarily the case that $\Hat{\gamma}$ is primitive,  cf. Figure \ref{Fig:Imprimitive}. In the literature one often sees a more complicated definition of the hat-homology class of a saddle trajectory which boils down to taking the unique primitive multiple of  our $\Hat{\gamma}$.}
 of the trajectory $\gamma$.  
 Note that, by definition, it satisfies $Z_\phi(\Hat{\gamma}) \in \R_{>0}$.

 \begin{figure}[ht]
\begin{center}
\includegraphics[scale=0.4]{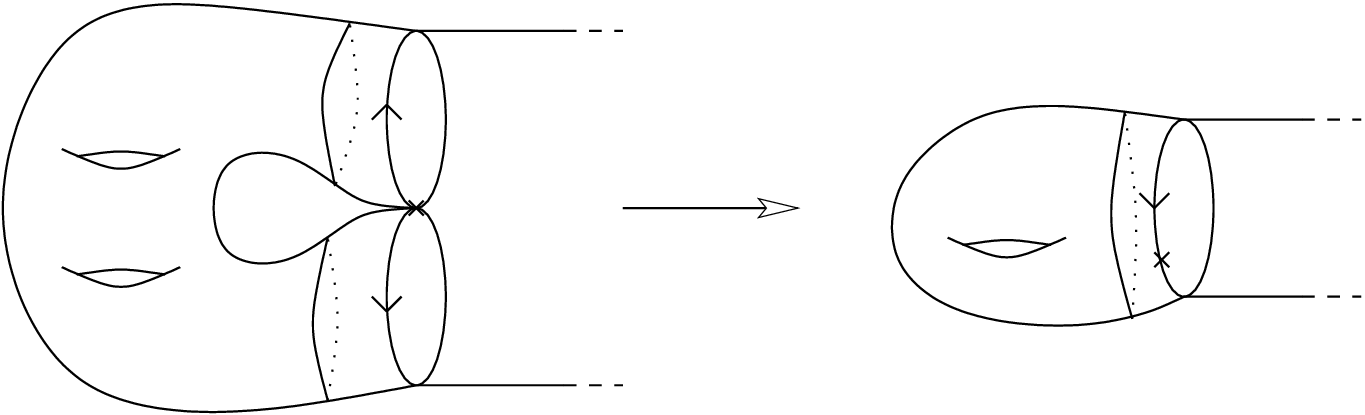}
\end{center}
\caption{A closed saddle trajectory $\gamma$, and its preimages $\gamma^{\pm}$ in the spectral cover, whose union define its (imprimitive) hat-homology class\label{Fig:Imprimitive}.}
\end{figure}

Similar remarks apply to maximal straight arcs of finite-length and nonzero phase $\theta$. The only difference is that we orient the inverse image of the arcs on $\hS$ by insisting that $\psi$ evaluated on the tangent vector should have positive imaginary part. This means that the corresponding hat-homology classes  have periods $Z_\phi(\Hat{\gamma})$ lying in the upper half-plane.  


\subsection{Critical points}
\label{crit}

We now describe the local  structure of the horizontal foliation  near a  critical point  of  a meromorphic quadratic differential, following  Strebel \cite[\SS 6\,]{Strebel}.  

Let $\phi$ be a meromorphic quadratic differential on a Riemann surface $S$. Suppose first that $p\in \Crit_{<\infty}(\phi)$ is either a simple pole of $\phi$, in which case we set $k=-1$, or a zero of some order $k\geq 1$. Then there are  local co-ordinates $t$ such that
\[
\phi(t)\,  = c^2\cdot t^{k}\, dt^{\tensor 2}, \qquad c=\half(k+2).
\]
At nearby points of  $S\setminus\{p\}$, a distinguished local co-ordinate is 
$w=t^{\frac{1}{2}(k+2)}$. The local trajectory structure  is illustrated in the cases $k=\pm 1$ in Figure \ref{zero}.

\begin{figure}[ht]
\begin{center}
\includegraphics[scale=0.4]{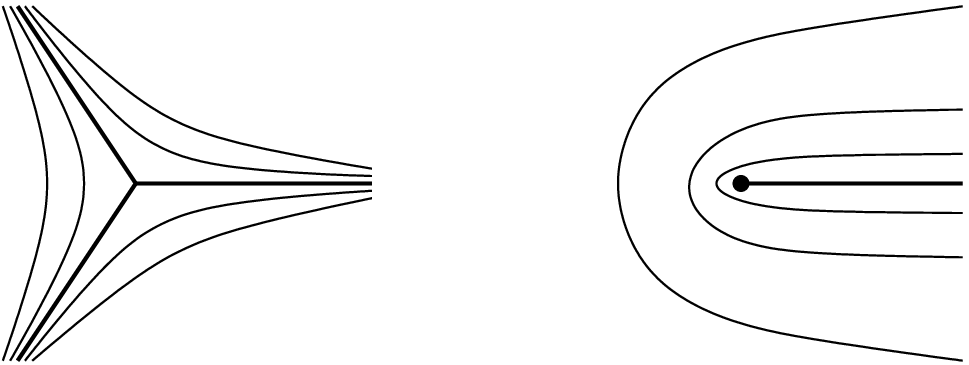}
\end{center}
\caption{Local trajectory structures at a simple zero and a simple pole\label{zero}.}
\end{figure}

Note that three horizontal rays emanate from each simple zero;  this trivalent structure will be the basic reason for the link with triangulations.

Next suppose that $p\in \Crit_{\infty}(\phi)$ is a  pole of order 2. Then there are local co-ordinates $t$ such that \[\phi(t)=r \, \frac{dt^{\otimes 2}}{t^2},\] for some well-defined constant $r\in \bC^*$. The residue of $\phi$ at $p$ is 
\begin{equation}
\label{gotcha}Z_\phi(\beta_p)=\res_p(\phi)=\pm 4\pi i \sqrt{r},\end{equation}
 and is well-defined up to sign.

At nearby points of $S\setminus\{p\}$ any branch of the function $w=\sqrt{r} \,\log(t)$  is  a distinguished local co-ordinate,
and the structure of the  horizontal foliation near  $p$ is determined by the residue  as follows:
  \begin{itemize}
\item[(i)] if $\res_p(\phi)\in \bR$  the foliation is by concentric circles centred on the pole;\smallskip
  \item[(ii)] if $\res_p(\phi)\in i\bR$ the foliation is by radial arcs emanating from the double pole;\smallskip
\item[(iii)] if $\res_p(\phi)\notin \bR\cup i\bR$  the leaves of the foliation are logarithmic spirals which wrap onto the pole.
\end{itemize}
\begin{figure}[ht] \begin{center}
\includegraphics[scale=0.9]{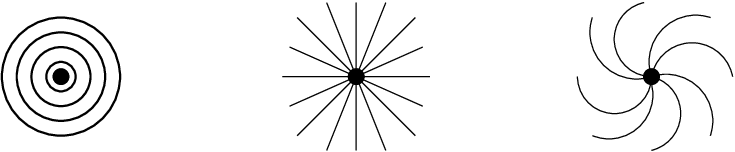}
\caption{Local trajectory structures at a double pole\label{Fig:LocalTrajectories}.}
\end{center} \end{figure}
These three cases are illustrated in Figure \ref{Fig:LocalTrajectories}. 
In cases (ii) and (iii) there is a neighbourhood $p\in U\subset S$  such that any trajectory entering $U$ tends to $p$.

Finally, suppose that $p\in \Crit_\infty(\phi)$ is a pole of order $m>2$. If $m$ is odd, there are  local co-ordinates $t$ such that
\[
\phi(t)\,  = c^2\cdot t^{-m}\, dt^{\tensor 2}, \qquad c=\half(2-m).
\]
as before. If $m\geq 4$ is even, there are local co-ordinates $t$ such that
\[\phi(t) = \big(c t^{-m/2} + \frac{b}{t}\big)^2 dt^{\tensor 2}, \quad c=\half(2-m).\]
The residue of $\phi$ at $p$  is then \[Z_\phi(\beta_p)=\res_p(\phi)=\pm 4\pi i b,\]
and is well-defined up to sign.
 
\begin{figure}[ht]
\begin{center}
\includegraphics[scale=0.4]{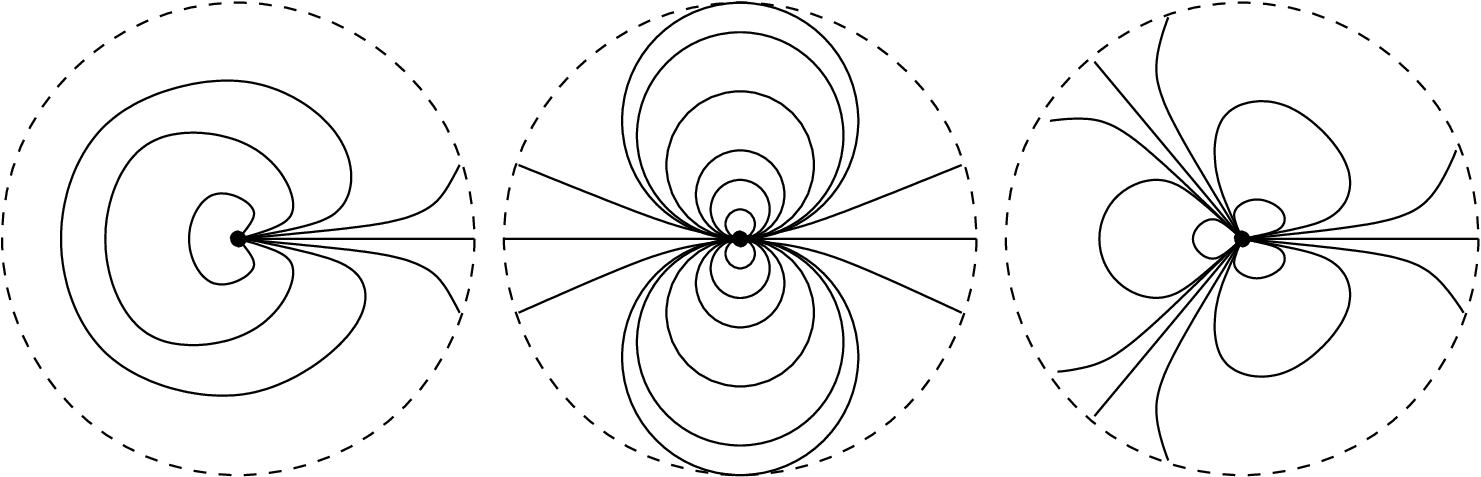}
\end{center}
\caption{Local trajectory structures at poles of order $m=3,4,5$\label{pole}.}
\end{figure}

The trajectory structure in these cases is  illustrated in Figure \ref{pole}.
 There is  a neighbourhood $p\in U\subset S$ and a collection of  $m-2$ distinguished tangent directions $v_i$ at $p$, such that any trajectory entering $U$ eventually tends to  $p$ and  becomes asymptotic to  one of the $v_i$. 
 

\subsection{Global trajectories}

Let $\phi$ be a GMN differential on a  compact Riemann surface $S$. We now consider  the global structure of the horizontal foliation of $\phi$, again following Strebel \cite[\SS\SS \,9--11]{Strebel}.
Every  trajectory of  $\phi$ falls into exactly one of the following categories:
\begin{enumerate}
\item \emph {saddle trajectories} approach  finite critical points  at both ends; \smallskip
\item \emph{separating trajectories}\footnote{These trajectories do not separate the surface: we call them separating because in the generic saddle-free situation considered in Section \ref{sadfree} the separating trajectories divide the surface into a disjoint union of cells.} approach critical points  at each end, one finite and one infinite;\smallskip
\item \emph{generic trajectories} approach  infinite critical points  at both ends;\smallskip
\item \emph{closed trajectories} are simple closed curves in $S\setminus\Crit(\phi)$;\smallskip
\item \emph{recurrent trajectories} are recurrent in at least one direction.\end{enumerate}
Since only finitely many horizontal arcs emerge from each finite critical point, the number of saddle trajectories and separating trajectories is finite.
 Removing these  from $S$, together with the critical points $\Crit(\phi)$,  the remaining open surface splits as  a disjoint union of connected components which can be classified  as follows\footnote{See \cite[Section 11.4]{Strebel}. Strictly speaking the decomposition is into \emph{maximal} horizontal strips, half-planes etc, but since all such domains we consider will be maximal, we  drop the qualifier. Recall that we have outlawed various degenerate cases: by assumption $\phi$ has at least one finite critical point, and at least one pole.}  
\begin{enumerate}
\item A \emph{half-plane} is equivalent to the upper half-plane \[\{z\in \C :\Im(z)>0\}\subset \C\] equipped with the differential $dz^{\tensor 2}$. It is swept out by generic trajectories which connect a fixed pole of order $m>2$ to itself. The boundary  is made up of saddle trajectories and separating trajectories.\smallskip

\item
 A \emph{horizontal strip} is equivalent to a region \[\{z\in \C:a<\Im(z)<b\}\subset \C, \]
  equipped with the differential $dz^{\tensor 2}$.  It is swept out by generic trajectories connecting two (not necessarily distinct) poles of arbitrary order $m\geq 2$. Each component of the boundary  is made up of  saddle trajectories and separating trajectories. \smallskip

\item
 A \emph{ring domain}  is equivalent to a region \[\{z\in \C:a<|z|<b\}\subset \C^*,\]
  equipped with the differential $r\, dz^{\tensor 2}/z^2$ for some $r\in \R_{<0}$.   It is swept out by closed trajectories. Each component of the boundary  is either made up  of saddle trajectories or is a single double pole of $\phi$ with  real residue.\smallskip

\item A \emph{spiral domain} is defined to be the interior of the closure of a recurrent trajectory. The only fact  we shall need is that the boundary of a spiral domain is  made up of saddle trajectories.  In particular there are no infinite critical points in  the closure of a spiral domain.
\end{enumerate}

A ring domain $A$ will be called  \emph{degenerate} if  one of its boundary components consists of a double pole $p$. The residue $\res_p(\phi)$ is then necessarily real, and $A$ consists of closed trajectories encircling $p$. Conversely, any double pole $p$ with real residue is contained in a degenerate ring domain.
A ring domain $A$  will be called \emph{strongly non-degenerate} if its boundary consists of two,  pairwise disjoint, simple closed curves on $S$.    Not all non-degenerate ring domains are strongly non-degenerate;  for example,  in the  case of finite area  differentials, there is a dense subspace of $\Quad(g,m)$ consisting of differentials which have a single dense ring domain  \cite[Theorem 25.2]{Strebel}. 


\subsection{Saddle-free differentials}
\label{sadfree}

 We say that a GMN differential  is \emph{saddle-free} if it has no saddle trajectories. The following simple but crucial observation comes from \cite[\SS  6.3]{GMN2}.

\begin{lemma}\label{prop:NoClosed}
If a GMN differential $\phi$ is saddle-free, and $\Crit_\infty(\phi)$ is non-empty, then $\phi$ has no closed or recurrent trajectories.
\end{lemma}

\begin{pf}
Since $\Crit_\infty(\phi)$ is non-empty the surface $S$ cannot be the closure of a spiral domain. 
On the other hand, the boundary of a spiral domain consists of saddle trajectories. Thus there can be no spiral domains, and hence no recurrent trajectories.  Similarly the boundary of a ring domain must contain saddle trajectories, except for the case when both boundary components are double poles with  real residue. This can only occur when $g=0$ and the polar type is $m=(2,2)$; such differentials are not GMN since they have no finite critical points.
\end{pf}

Let $\phi$ be a saddle-free GMN differential such that $\Crit_\infty(\phi)$ is non-empty. Removing the finitely many separating trajectories from $S\setminus\Crit(\phi)$ gives an open surface which is a disjoint union of   horizontal strips and half-planes swept out by generic trajectories. 
\begin{figure}[ht]
\begin{center}
\includegraphics[scale=0.4]{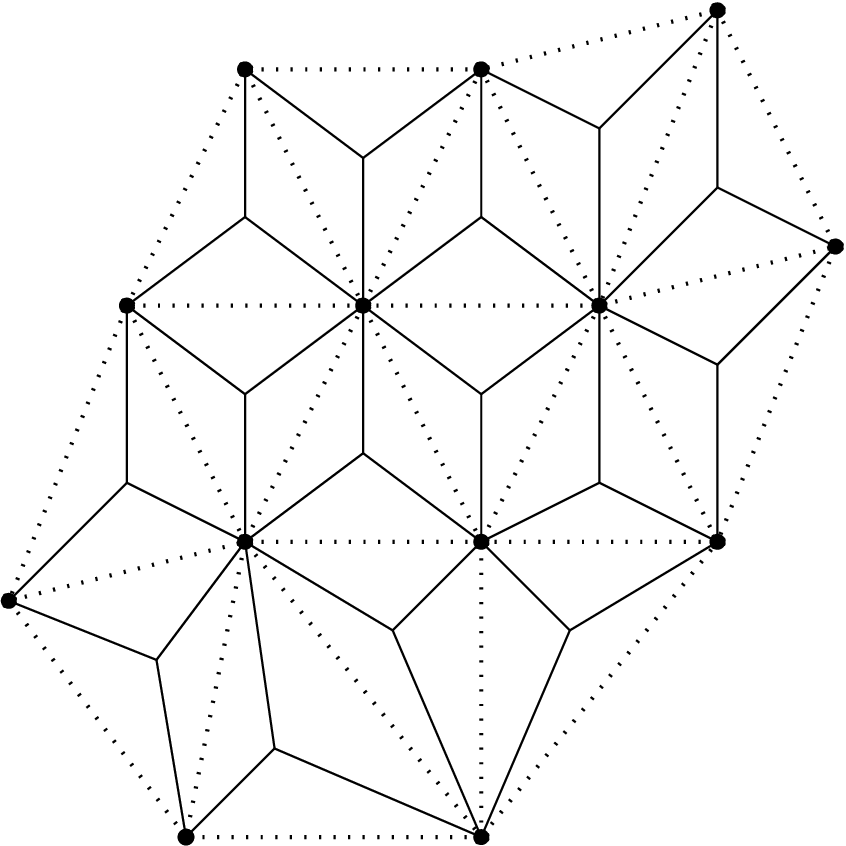}
\end{center}
\caption{The generic (dotted) and  separating trajectories (solid) for a  saddle-free GMN differential having only double poles. All horizontal  strips in the picture are non-degenerate.  \label{dualgraph2}}
\end{figure} 

Each of the two components of the boundary of a horizontal strip contains exactly one finite critical point of $\phi$. \begin{figure}[ht]
\label{twocells}
\begin{center}
\includegraphics[scale=0.4]{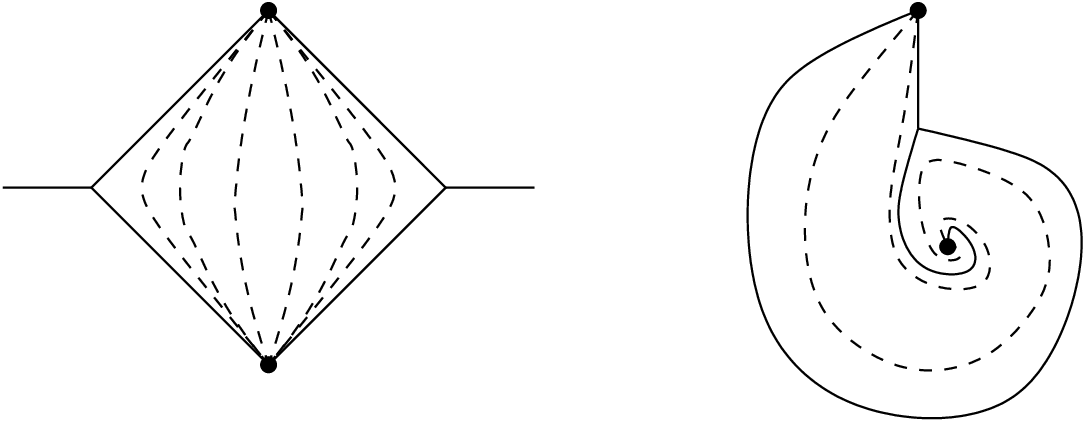}
\end{center}
\caption{Two types of strip, regular and degenerate.\label{Fig:Degenerate}}
\end{figure}If these are both zeroes, then embedded in the surface there are  two possibilities, depending on whether the two zeroes are distinct or coincide;  we call the corresponding strips \emph{regular} or \emph{degenerate} respectively. 
These two possibilities are illustrated in Figure \ref{Fig:Degenerate};  note though that  the two double poles in the first of these pictures could well coincide on the surface.

\begin{figure}[ht]
\label{twocells}
\begin{center}
\includegraphics[scale=0.4]{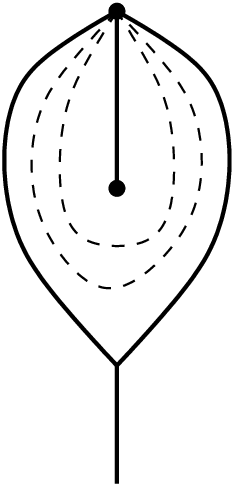}
\end{center}
\caption{Horizontal strip with a simple pole on its boundary; the simple pole is in the centre of the diagram with a double pole above and a simple zero below.\label{Fig:SimplePoleStrip}}
\end{figure}

A horizontal strip containing a simple pole in one of its boundary components is almost always of the form illustrated in Figure \ref{Fig:SimplePoleStrip}. The one exception occurs in genus 0 and polar type $(1,1,2)$: the moduli space of such differentials consists of a single $\C^*$-orbit, and the trajectory structure for a generic  element consists of a single horizontal strip  containing two simple poles in  its boundary.

\subsection{Standard saddle connections}
\label{ihope}

Let $\phi$ be a  saddle-free GMN differential on a Riemann surface $S$, and assume that $\Crit_\infty(\phi)$ is non-empty. The interior of each horizontal strip  is  equivalent to a strip in $\C$ equipped with the differential $dz^{\tensor 2}$. 
  In each  such strip $h$ there is a unique saddle connection $\ell_h$ connecting the two finite critical points  on the opposite sides of the strip.   

\begin{figure}[ht]
\begin{center}
\includegraphics[scale=0.6]{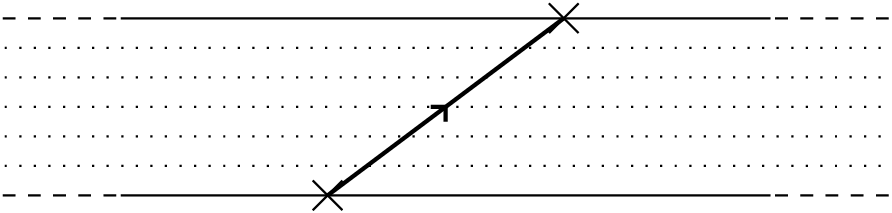}
\end{center}
\caption{A horizontal strip in $\C$ with its standard saddle connection.}
\end{figure}

Since $\phi$ is saddle-free,  $\ell_h$ must have nonzero phase.   As in Section \ref{l}, there is  an associated hat-homology class $\alpha_h\in\hs$, which by definition  satisfies $\Im Z_\phi(\alpha_h)>0$.
 We call the arcs $\ell_h$ the \emph{standard saddle connections} of the differential $\phi$. The
  classes $\alpha_h$  will be called the \emph{standard saddle classes}. 
  
\begin{lemma}
\label{bassist}
The standard saddle classes $\alpha_h$  form a basis for the group $\hs$.
\end{lemma}

\begin{pf}
In each horizontal strip $h_i$ we can choose a generic trajectory and then take one of its two lifts    to the spectral cover  to give a class $\delta_{h_i}$ in the relative homology group of
\eqref{wine}. The intersection number $\langle \alpha_{h_i},\delta_{h_j}\rangle$ is then nonzero precisely if  $h_i=h_j$, in which case it is $\pm 1$. Thus the elements $\alpha_{h_i}$ are linearly independent.
Lemma \ref{train} states that the group $\hs$ is free of rank $n$ given by equation \eqref{n}. To complete the proof it will be enough to show that this is also the number of horizontal strips of $\phi$.

By a transverse orientation of a separating trajectory we mean a continuous choice of normal direction; for each separating trajectory there are two possible choices.
We orient the separating trajectories in the boundary of a horizontal strip by taking the inward pointing normal direction.
Each horizontal strip then has four transversally oriented separating trajectories in its closure; for a degenerate strip, two of these consist of different orientations of the same trajectory. Similarly, each half-plane has two such oriented trajectories. Moreover, every oriented separating trajectory occurs as the boundary of exactly one half-plane or horizontal strip.

Let $x$ be the number of horizontal strips, and $s$ the number of simple poles. Three horizontal arcs emanate from each  zero, and one from each simple pole, and each of these forms the end of a separating trajectory. Each pole of order $m\geq 3$ is surrounded by $m-2$ half-planes, so the total number of these is $s+\sum_{i=1}^d (m_i-2)$. Thus
we get an equality
\[4x+2s+2\sum_{i=1}^d (m_i-2)=6 (4g-4+\sum_{i=1}^d  m_i)+2s.\]
Simplifying this expression gives $x=n$.\end{pf}


\subsection{Geodesics}
\label{geod}

Let $\phi$ be a meromorphic quadratic differential on a  Riemann surface $S$. Recall from Section \ref{Sec:Flat} that $\phi$ induces a  metric space structure on the open subsurface $S^\op=S\setminus\Crit_\infty(\phi)$.  A \emph{$\phi$-geodesic}  is defined to be a locally-rectifiable path $\gamma\colon [0,1]\to S^\op$   which is locally length-minimizing.  Note that it is not assumed that $\gamma$ is  the shortest path between its endpoints. 

It follows immediately from the definition of the  $\phi$-metric that any straight arc is a $\phi$-geodesic, and that conversely, in a neighbourhood of a non-critical point of $\phi$, any geodesic is a straight arc. Using the canonical co-ordinate systems of Section \ref{crit}, it is easy to determine the local behaviour of geodesics near a finite critical point of $\phi$. Here  we briefly summarize the results of this analysis, and refer the reader to Strebel \cite[ \SS 8]{Strebel} for more details.

 In a neighbourhood of a zero $p$ of $\phi$ of order $k$,  any two points  are joined by a unique geodesic, which is also the shortest curve in $S^\op$ connecting these points. This unique geodesic  is either a straight arc not passing through $p$, or is composed of  two radial straight arcs emanating from $p$. This second situation occurs precisely if the  angle between the radial arcs is $\geq 2\pi/(k+2)$.
\begin{center}
\begin{figure}[ht]
\includegraphics[scale=0.5]{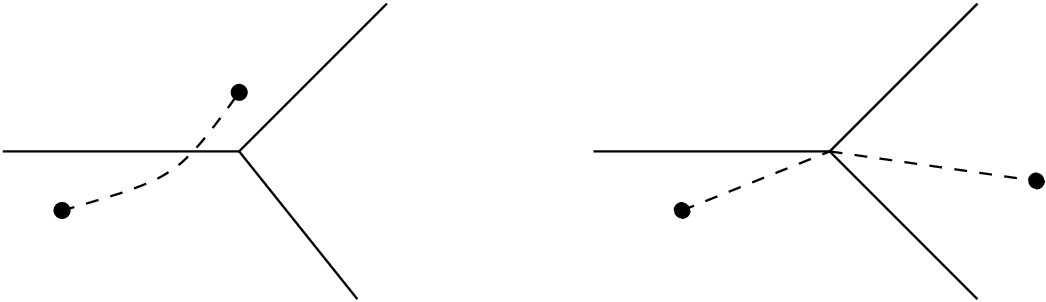}
\caption{Geodesic segments near a simple zero.\label{Fig:LocalNearZero}}
\end{figure}
\end{center}

 In  a neighbourhood of a simple pole $p$ of $\phi$, any two points are connected by at least one  geodesic, but uniqueness of geodesics fails: some pairs of points are connected by more than one straight arc. Moreover, a geodesic need not be the shortest path between its endpoints: it is length-minimizing locally, but  not necessarily globally. Note however, that no geodesic contains the point $p$ in its interior: the only geodesics passing through $p$ begin or end there.
 \begin{center}
\begin{figure}[ht]
\includegraphics[scale=0.5]{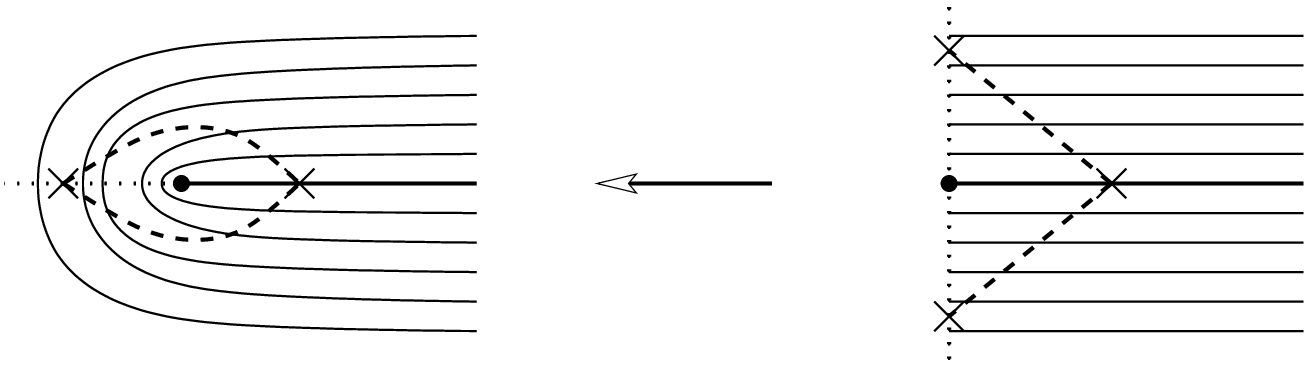}
\caption{Geodesic segments near a simple pole, and their inverse images under the square-root map. Note that the pulled-back differential has a non-critical point at the inverse image of the pole.\label{Fig:Nonunique}}
\end{figure}
\end{center}

From these local descriptions, it immediately follows that any geodesic in $S^\op$ is a union of (closures of) straight arcs, joined at zeroes of $\phi$. In particular, any geodesic connecting points of $\Crit_{<\infty}(\phi)$ is a union of saddle connections.
 Of course, the phases of the constituent saddle connections will usually be different.    


\subsection{Gluing surfaces along geodesics}
\label{cp}

It will be useful in what follows to glue  Riemann surfaces equipped with quadratic differentials along closed curves made up of unions of saddle connections.  We will use some particular examples of this construction in Sections \ref{rs} and \ref{busp} below.

Consider a topological surface $\S$ with boundary. By a  quadratic differential on $\S$ we  simply mean a quadratic differential on the interior of $\S$, that is  a  quadratic differential  on a Riemann surface  whose underlying topological surface is the interior of $\S$.  We say that two such surfaces  $\S_i$ equipped with differentials $\phi_i$ are equivalent if there is a homeomorphism $f\colon \S_1\to \S_2$ which restricts to a biholomorphism on the interiors and    satisfies $f^*(\phi_2)=\phi_1$.

Given an integer $k\geq 0$ we denote by $\mathbb{V}_k\subset \C$  the closed sector bounded by the rays of argument 0 and $2\pi(k+1)/(k+2)$. We equip the interior of $\mathbb{V}_k$ with the differential \[\phi_k(t)=c^2\cdot t^k\, dt^{\tensor 2}, \quad c=\frac{1}{2} (k+2).\] Thus, for example, $\mathbb{V}_0\subset \C$ is the closed upper half-plane equipped with the standard differential $\phi_0(t)=dt^{\tensor 2}$ on its interior. In general the differential $\phi_k$ extends holomorphically over a neighbourhood of the boundary of $\mathbb{V}_k$, and when $k>0$, the boundary $\partial \mathbb{V}_k$ then consists of two horizontal trajectories of $\phi_k$ meeting at a zero of order $k$. 

Note that the map $t^{k+2}=z^2$ gives an equivalence
\begin{equation}
\label{meme}
(\C\setminus \mathbb{V}_k,\phi_k(t)) \isom (\h, dz^{\tensor 2}).\end{equation}
Thus a copy of $\mathbb{V}_k$ can be glued to a copy of $\mathbb{V}_0$ in such a way that the differentials $\phi_k$ and $\phi_0$ on the interiors extend to a well-defined differential on $\C$.

If $\phi$ is a quadratic differential on a topological surface $\S$ with boundary, we say that the pair   $(S,\phi)$  has a \emph{gluable boundary} if  each point $x\in \partial \S$ has a neighbourhood which is equivalent to a neighbourhood of $0\in \mathbb{V}_k$ for some $k\geq 0$. In particular it follows  that the boundary $\partial \S$ is either  a union of saddle trajectories or a single closed trajectory. Note, however, that the gluable boundary condition  is a much stronger statement: if $z\in \partial\S$ is a zero of $\phi$ of order $k$,  then there are $k+2$ horizontal trajectories in $\S$ emanating from $z$, two of which lie in  the boundary. 

Suppose that  $S$ is a Riemann surface equipped with a meromorphic differential $\phi$ having simple zeroes, and  that $\gamma\subset S$ is a separating simple closed curve which is either a closed trajectory or a union of saddle trajectories. Cutting the underlying topological surface $\S$ along $\gamma$ we can view it  as a union of two surfaces  with boundary $\S_\pm$ glued along the curve $\gamma$. The assumption that $\phi$ has simple zeroes then  immediately implies that the pairs $(S_\pm,\phi|_{S_\pm})$ have gluable boundaries in the sense described above.
 
Conversely, suppose  that $\S_\pm$ are two smooth, oriented surfaces with boundary, each with a single boundary component $\partial \S_\pm$, and each equipped with  a meromorphic quadratic differential $\phi_\pm$.  

 \begin{Lemma} \label{Lem:GlueBorders}
Suppose that the pairs $(\S_\pm,\phi_\pm)$  have gluable boundaries, and that the $\phi_\pm$-lengths of the  boundaries $\partial \S_\pm$ are equal. Then there is a  Riemann surface  $S$ whose underlying topological surface $\S$ is obtained by gluing the surfaces $\S_\pm$ along their boundaries,     and a meromorphic differential $\phi$ on $S$ which coincides with the differentials $\phi_\pm$ on the interiors of the two subsurfaces $\S_\pm\subset \S$.     \end{Lemma}

\begin{proof}  Parameterize the two boundary components $\partial \S_\pm$ by arc-length in the $\phi_\pm$-metric, and then identify  them. When we do this we have the freedom to choose the  rotation of the two surfaces relative to each other, and we can therefore  ensure that   zeroes of $\phi_\pm$ do not become identified. The fact  that the quadratic differentials $\phi_\pm$ glue together then follows from the equivalence \eqref{meme}.  \end{proof}      
  
\begin{remarks}
\begin{itemize}
\item[(a)]It is clear from the proof of Lemma \ref{Lem:GlueBorders} that the surface $S$ is not uniquely determined by the pairs $(\S_\pm,\phi_\pm)$: we can rotate the subsurfaces $S_\pm$  relative to one another. \smallskip
\item[(b)] The gluable boundary assumption is  necessary: one cannot always glue differentials on surfaces whose boundaries are made up of saddle trajectories. Indeed, otherwise one could take  a degenerate ring domain  whose boundary consists of $i\geq 1$ saddle trajectories, and glue it to itself to obtain a meromorphic differential on a sphere with 2 double poles and $i$ simple zeroes. This   cannot exist by Riemann-Roch.
\end{itemize}
\end{remarks}



\section{Period co-ordinates}
\label{proof}

The aim of this section is  to prove that the period map \eqref{periodmap} on the space of framed  differentials  is a local isomorphism.  For finite area differentials this  is standard, but for the more general meromorphic differentials considered here there does not seem to be a proof in the literature.  The reader prepared to take this result on trust can skip to the next section. We begin by considering geodesics for the metric defined by a GMN differential $\phi$, and the way in which these change as $\phi$ moves in the corresponding space $\Quad(g,m)$.

\subsection{Existence and uniqueness of geodesics}
Let $\phi$ be a meromorphic quadratic differential on a  compact Riemann surface $S$. As in Section \ref{Sec:Flat} we equip the open subsurface $S^\op=S\setminus\Crit_\infty(\phi)$ with the metric space stucture induced by the $\phi$-metric. In this section we  state some well-known global existence and uniqueness properties for geodesics on this surface.  A more detailed treatment can be found in  \cite[\SS \SS  14--18]{Strebel}.

Given  points $p, q\in S^\op$, we denote by $\CC(p,q)$  the set of all rectifiable paths $\gamma\colon [0,1]\to S^\op$ connecting $p$ to $q$. We  equip this set with  the topology of uniform convergence.  Two curves in $\CC(p,q)$ are considered homotopic if they are homotopic relative to their endpoints through paths in $S^\op$. 
  We denote  by $\ell_\phi(\gamma)$ the length of a curve $\gamma\in \CC(p,q)$.
A curve   in $\CC(p,q)$ will be called a \emph{minimal geodesic} if no  homotopic path  has smaller length; any such curve is locally length-minimising, and hence a geodesic. 

The following result is well-known.

\begin{thm}
\label{finini}
\begin{itemize}
\item[(a)] the subset of curves in $\CC(p,q)$ representing a given homotopy class is open and closed,\smallskip
\item[(b)] the function sending a curve in $ \CC(p,q)$ to its length  is lower semi-continuous,\smallskip
\item[(c)]for any $L>0$, the subset of curves in $\CC(p,q)$ of length $\leq L$ which are parameterized proportional to arc-length is compact, \smallskip
\item[(d)]every homotopy class of curves in $\CC(p,q)$ contains at least one  minimal geodesic,\smallskip
\item[(e)]if $\phi$ has no simple poles then  geodesics in $\CC(p,q)$ are homotopic only if they are equal.\end{itemize}
\end{thm}

\begin{pf}
 Since the surface $S$ is assumed compact, the metric space $S^\op$ is  proper, which is to say that all closed, bounded subsets are compact. It is also clear that any two points of $S^\op$ can  be connected by a rectifiable path. The  statements (a) - (d) hold for all metric spaces with these two properties: see for example \cite[Section 1.4]{Pap}. Part (e) is proved by Strebel \cite[Theorem 16.2]{Strebel}.
\end{pf}

 If the differential $\phi$ has no simple poles, Theorem \ref{finini} implies that  all geodesics  are minimal.
If $\phi$ has simple poles the situation is  more complicated: a given homotopy class  may contain more than one geodesic representative, and not all such representatives need be minimal.

\begin{lemma}
\label{fini}
For any  $L > 0$, there are only finitely many geodesics  $\gamma\in \CC(p,q)$ with $\ell_\phi(\gamma)\leq L$.
\end{lemma}

\begin{pf}
First assume that $\phi$ has no simple poles. It follows from Theorem \ref{finini}(c) that the subset of $\CC(p,q)$ consisting of curves of length $\leq L$  has only finitely many connected components. In particular, by part (a), there can only be finitely many homotopy classes of such curves. But, by part (e), a geodesic is determined by its homotopy class, so the result follows. 
In the general case, take a covering $\pi\colon \tilde{S}\to S$ branched at all simple  poles   of $\phi$, and consider the pulled-back differential $\tilde{\phi}=\pi^*(\phi)$.  Any $\phi$-geodesic in $S$ can be lifted to a $\tilde{\phi}$-geodesic in $\tilde{S}$ of the same length. Since  $\tilde{\phi}$   has no simple poles, this reduces us to the previous case. 
\end{pf}


\subsection{Varying the differential}

\label{perm}
Our next step is to study  the way geodesics  of a GMN differential  move  as  the differential  varies in   its moduli space.
 Fix a genus $g\geq 0$ and a collection of $d\geq 1$ positive integers $m=\{m_i\}$.
Recall from the proof of Proposition \ref{rain} that, when it is non-empty, the space $\Quad(g,m)$ is an open subset of a vector bundle 
\[\mathcal{H}(g,m)\lra \Mod(g,d)/\Sym(m).\]
The  fibre of this bundle over a marked curve $(S,(p_i))$  is the space of global sections of the line bundle $\omega_S^{\tensor 2}(\sum_i m_i p_i)$.

 Let us consider a fixed differential $\phi_0\in \Quad(g,m)$, which we view as a base-point, and consider an open ball\footnote{More precisely, if $\phi_0$ is an orbifold point, we should take an {\'e}tale map $Q\to\Quad(g,m)$ from a complex ball, but we suppress this point in what follows. Alternatively one could pull back the bundle \eqref{hs} to Teichm{\"u}ller space and work locally there.} $\phi_0\in Q\subset \Quad(g,m)$. 
By Ehresmann's theorem, the universal curve over $\Mod(g,d)$ pulls back to
a differentiably locally-trivial fibre bundle over ${Q}$. It follows that we can fix an underlying smooth surface $\S$, and view the points of $Q$ as defining pairs consisting of a complex structure on $\S$ together with a meromorphic quadratic differential $\phi$ on the resulting Riemann surface $S$. Composing with a smoothly varying family of diffeomorphisms we can further  assume that  the differentials in $Q$ have poles and zeroes at the same fixed points of $\S$.

\begin{lemma}
\label{Ry}
Fix a constant $R>1$. Then  any point of $Q$ is contained in some neighbourhood $U\subset Q$ such that 
\[(1/R)\cdot \ell_{\phi_1}(\gamma)\leq  \ell_{\phi_2}(\gamma)\leq R\cdot \ell_{\phi_1}(\gamma),\]
for any curve $\gamma$ in $\S$, and any pair of differentials $\phi_i\in U$.
\end{lemma}

\begin{pf}
Fix an arbitrary Riemannian metric $g$ on the smooth surface $\S$, and write $\eta(x,y)$ for the  distance between two points $x,y\in \S$ computed in this metric.  
Away from the poles $p_i$ we can view the meromorphic differential $\phi$ corresponding to a point of $Q$ as defining a smooth section of the bundle $(T^*_{\S})^{\tensor 2}$, the tensor square of the rank 2 bundle of smooth complex-valued 1-forms on $\S$. Near a pole $p_i$ of order $m_i$, the rescaled section $\eta(x,p_i)^{m_i} \cdot \phi(x)$ is smooth  in a neighbourhood of $p_i$, and has non-zero value at $p$.  Similar remarks apply near a zero of $\phi$. 

Given two points $\phi_1,\phi_2\in Q$ it follows that the ratio $|\phi_1|/|\phi_2|$, considered as a smooth function on the set of nonzero tangent vectors to $\S$,  is everywhere defined and varies smoothly with the differentials $\phi_i$. Thus around any point of $Q$ we can find a  neighbourhood  $U\subset Q$ such that $(1/R)\cdot |\phi_2|\leq |\phi_1|\leq R \cdot |\phi_2|$, for all $\phi_1,\phi_2\in U$ and all tangent vectors to $\S$.  Integrating this inequality along a curve   gives the result.
\end{pf}


\subsection{Persistence  of  saddle connections}
In this section we show that if a GMN differential varies continuously in its moduli space then its geodesics also vary continuously. We take notation as in the last section.

\begin{prop}
\label{persist}
Suppose that $\gamma_0\in \CC(p,q)$ is a $\phi_0$-geodesic. Then there is a   family of curves $\gamma(\phi)\in \CC(p,q)$, varying continuously with $\phi\in Q$, such that $\gamma_0=\gamma(\phi_0)$, and such that for all $\phi\in Q$ the curve  $\gamma(\phi)$ is a $\phi$-geodesic. 
\end{prop}

\begin{pf}
Let us first consider the case when $\phi_0$ has no simple poles.
  By Theorem \ref{finini}, for each $\phi\in Q$ there exists a unique $\phi$-geodesic $\gamma(\phi)$ in $\CC(p,q)$ which is homotopic to $\gamma_0$. We must  show that the resulting curves $\gamma(\phi)$ vary continuously with $\phi$.
  Assuming the opposite, let us take $\epsilon>0$ and suppose that there exists a sequence of differentials $\phi_n\in Q$ with $\phi_n\to \phi$, such that for all $n$ the geodesic $\gamma_n=\gamma(\phi_n)$ does not lie within distance $\epsilon$ of $\gamma=\gamma(\phi)$ in the supremum norm. In other words, for each $n$, we can find $t_n\in [0,1]$ such that \[d(\gamma_n(t_n),\gamma(t_n))\geq \epsilon.\]
 Passing to a subseqence we can assume that $t_n\to t\in [0,1]$. 
Lemma \ref{Ry} shows that for any $R>1$  \begin{equation}
\label{cardy}(1/R)\cdot \ell_{\phi}(\gamma_n)\leq  \ell_{\phi_n}(\gamma_n)\leq  \ell_{\phi_{n}}(\gamma)\leq R\cdot \ell_{\phi}(\gamma),\end{equation}
for large enough $n$.
In particular, we can assume that the $\gamma_n$ all satisfy $\ell_\phi(\gamma_n)\leq L$, for some constant $L>0$.
Theorem \ref{finini} implies that, when parameterised proportional to $\phi$-arclength, some subsequence of the $\gamma_n$ converges to a limit curve $\gamma_L$.
This limit curve cannot be equal to $\gamma$, since \[d(\gamma_L(t),\gamma(t))\geq \epsilon.\] On the other hand,  the  inequalities \eqref{cardy} show that $\ell_\phi(\gamma_L)\leq \ell_\phi(\gamma)$.  This contradicts the fact, immediate from Theorem \ref{finini}, that all geodesics are minimal.

For the general case we use the same trick as in Lemma \ref{fini}. Namely, we consider a covering $\pi\colon \tilde{S}\to S$ which is branched at all simple  poles   of $\phi_0$. We can lift $\gamma_0$ to a geodesic $\tilde{\gamma}_0$ on the surface $\tilde{S}$ for the pulled-back differential  $\pi^*(\phi_0)$. This differential   has no simple poles, so we can apply what we proved above to obtain a continuous deformation of $\tilde{\gamma_0}$. Pushing back down to $S$  gives the required deformation of $\gamma_0$.  \end{pf}

\begin{remarks}
\label{outy}
\begin{itemize}
\item[(a)]
If the geodesic $\gamma_0=\gamma(\phi_0)$ of Proposition \ref{persist} is a straight arc (which is to say that it contains no zeroes of $\phi_0$ in its interior) then, by continuity,  the same is true for the geodesics $\gamma(\phi)$ for all differentials $\phi$ in some neighbourhood of $\phi_0$. Thus, in particular, saddle connections   persist under small deformations of the differential. \smallskip

\item[(b)] A minor modification  of the proof  shows that the conclusion of Proposition \ref{persist} also  holds if we allow the endpoints $p,q$ of the path $\gamma(\phi)$ to vary continuously with the differential $\phi$.
\end{itemize}
\end{remarks}


\subsection{Persistence  of  separating trajectories}

We explained in Section \ref{crit}  that an infinite critical point $p$ of a meromorphic quadratic  differential is contained in a trapping neighbourhood $p\in U$ such that all trajectories entering $U$ eventually tend towards the point $p$. In fact we can be more explicit about this neighbourhood.

\begin{lemma}
\label{trap}
Take  a point $p\in \Crit_\infty(\phi)$ which is not a double pole with real residue. Then there is  a disc $p\in D\subset S$ whose boundary consists of saddle connections and such that  any trajectory  intersecting $D$ tends to $p$ in at least one direction.
\end{lemma}

\begin{pf}
Consider the geodesic representative of the closed loop $\delta_p$ around $p$. It consists of a union of straight arcs of varying phase connecting zeroes of $\phi_0$, which together cut out an open disc $p\in D\subset S$ containing no points of $\Crit_\infty(\phi)$.   This disc cannot contain any finite critical points of $\phi$ either: if $z\in D$ were such a point,  the geodesic representative of a loop round $p$ based at $z$ would be homotopic to $\delta_p$, contradicting uniqueness of geodesic representatives.  If a trajectory intersects the boundary of $D$ twice this again contradicts uniqueness of geodesics.  Hence any trajectory in one direction must either be recurrent or tend to the pole. But recurrence is also impossible since the boundary of the resulting spiral domain would involve saddle connections contained in $D$.
\end{pf}

\begin{remarks}
\begin{itemize}
\item[(a)] In the case of a double pole with real residue, the pole is enclosed in a degenerate ring domain whose boundary consists of a union of saddle trajectories. This ring domain is the analogue of the trapping neighbourhood: any trajectory intersecting $D$ is  one of the closed trajectories of $D$. \smallskip
\item[(b)] Proposition \ref{persist}  shows that the region $D=D(\phi)$ of Lemma \ref{trap} varies continuously with $\phi$. In particular, there is an open neighbourhood of the pole $p$ which has the trapping property for all differentials in a neighbourhood of a given base-point $\phi_0$.\end{itemize}
\end{remarks}

In the last section we proved that saddle connections persist to nearby differentials; we shall now prove  a similar result  for separating  trajectories. Note that in contrast to saddle connections (whose phases  vary as they deform) we can always deform separating trajectories in such a way  that they remain horizontal. 

\begin{prop}
\label{stay}
Suppose that $\gamma_0\colon [0,\infty)\to S^\op$ is a   separating trajectory for the differential $\phi_0$, which starts at a point $p\in S^\op$ and   limits to an infinite critical point $r\in S$.   Then there is a neighbourhood $\phi_0\in U\subset Q$, and a  family of curves $\gamma(\phi)\colon [0,\infty)\to S^\op$, varying continuously with $\phi\in Q$, such that $\gamma_0=\gamma(\phi_0)$, and such that for all $\phi\in Q$ the curve  $\gamma(\phi)$ is a  separating trajectory for  $\phi$, starting at $p$ and limiting to  $r$. 
\end{prop}

\begin{pf}
Note that $r$ cannot be a double pole of real residue. Consider the open neighbourhood $r\in D\subset S$ which has the trapping property for  any  $\phi$ lying in some neighbourhood $\phi_0\in U\subset Q$. Take a point $q_0=q(\phi_0)\in D$ on the trajectory $\gamma_0$.
Consider the  holomorphic function near $q_0$ obtained by integrating $\sqrt{\phi}$ along the trajectory $\gamma_0$. This function varies smoothly with $\phi$ so, by the implicit function theorem,  we can continuously vary $q(\phi)\in D$ so that \begin{equation}
\label{bl}\int_{p}^{q(\phi)} \sqrt{\phi} \in \R,\end{equation}
for all $\phi\in U$, 
where the integral is taken along a path homotopic to $\gamma_0$.

By Remark \ref{outy}(b),  there is a continuous family of curves $\gamma(\phi)$ parameterized by $\phi\in U$, with $\gamma(\phi_0)=\gamma_0$, and such that for each $\phi$ the curve  $\gamma(\phi)$ is a $\phi$-geodesic connecting $p$ to $q(\phi)$. Shrinking $U$ if necessary, each of these geodesics is in fact a straight arc, and the relation \eqref{bl} shows that these arcs are all horizontal. By the trapping assumption on $D$, each arc $\gamma(\phi)$   must extend to a separating trajectory $\gamma(\phi)\colon [0,\infty)\to S$  for $\phi$. The fact that these trajectories vary continuously when restricted to any finite interval $[0,t]\subset [0,\infty)$ then follows by another application of the argument of Proposition \ref{persist}.  \end{pf}


\subsection{Horizontal strip decompositions}
\label{periodshorizontal}

Fix again a genus $g\geq 0$ and a  collection of $d\geq 1$ unordered positive integers $m=\{m_i\}$. 
As preparation for proving that the period map \eqref{periodmap} is a local isomorphism, in this section and the next  we will study  the set of all saddle-free GMN differentials whose separating trajectories decompose the underlying smooth surface $\S$ into a given fixed set of horizontal strips and half-planes.  

We say that two saddle-free GMN differentials $(S_i,\phi_i)$ have \emph{the same horizontal strip decomposition} if there is an orientation-preserving diffeomorphism $f\colon S_1\to S_2$ which maps each horizontal strip (respectively half-plane) of $\phi_1$ bijectively onto a horizontal strip (respectively half-plane) of $\phi_2$. 
In particular,  equivalent differentials have the same horizontal strip decomposition.

More concretely, two equivalence-classes of saddle-free differentials  have the same horizontal strip decomposition precsiely if we can find representatives $(S_i,\phi_i)$ which have the same underlying smooth surface $\S$, and the same horizontal strips, half-planes and separating trajectories.

We would like to classify equivalence classes of saddle-free differentials $(S,\phi)$ with a given horizontal strip decomposition in terms of the periods  of the corresponding standard saddle classes $\alpha_h$. However, the existence of   differentials with automorphisms which permute their horizontal strips makes it impossible to  assign a well-defined period point to an arbitrary  saddle-free differential. The solution is to consider framed differentials, as in Section \ref{fr}.

We say that two  framed GMN differentials 
have the same  horizontal strip decomposition if there is an orientation-preserving diffeomorphism $f\colon S_1\to S_2$  preserving the horizontal strip decomposition as before, and also preserving the framings, in the sense  that the distinguished lift $\Hat{f}$ of Section \ref{fr}  makes the diagram \eqref{look} commute.   Again, equivalent framed differentials have the same horizontal strip decomposition.

Note that, by Lemma \ref{bassist}, a  framing  of a saddle-free differential gives rise to a labelling of the horizontal strips  by the  elements of a basis of $\Gamma$, and that conversely, the framing is completely determined by this labelling. Explicitly, if the framing is given by an isomorphism $\theta\colon \Gamma\to \hs$, then the strip $h$ is naturally labelled by  the element $\theta^{-1}(\alpha_h)$. Moreover,  two saddle-free differentials  have the same horizontal strip decomposition precsiely if we can find representatives $(S_i,\phi_i)$ which have the same underlying smooth surface $\S$, and the same horizontal strips as before, and which moreover have the same labellings by elements of $\Gamma$.

The following result will be  the basis for our proof of the existence of period co-ordinates. We defer the proof to the next subsection: by what was said above it amounts to classifying  saddle-free differentials $\phi$ on a smooth surface $\S$ with a  fixed set of horizontal strips and half-planes, and also with a fixed ordering of the horizontal strips.

\begin{prop}
\label{ivans}
Let $U\subset \Quad^\Gamma(g,m)$ be the set of equivalence-classes of framed saddle-free GMN differentials with a given horizontal strip decomposition. Choosing an ordering   of  the horizontal strips, the resulting map
\[\pi_U\colon U\to \C^n , \quad \phi\mapsto Z_\phi(\alpha_{h_i})\]
 is a bijection onto the subset
$ \{(z_1,\dots, z_n)\in \C^n:\Im (z_i)>0\}.$
\end{prop}

The next example shows that it is possible for a a saddle-free differential to have non-trivial automorphisms which preserve each horizontal strip. Such automorphisms preserve the standard arc classes and hence give automorphisms of the corresponding framed differential. 

\begin{Example}
\label{ex}
Consider the case $g=1$ and $m=(2)$: one of the exceptional cases of Lemma \ref{fatman}.
The space $\Quad(g,m)$ parameterizes pairs $(S,\phi)$, where $S$ is a Riemann surface of genus 1, and $\phi$ is a meromorphic differential with one double pole and two simple zeroes. Such differentials can be written explicitly as \[\phi(z)=(a\wp(z)+b)\,dz^{\tensor 2},\] where $\wp(z)$ is the Weierstrass $\wp$-function  corresponding to $S$. These functions are invariant under the inverse map $z\mapsto -z$.

\begin{center}
\begin{figure}[ht]
\includegraphics[scale=0.4]{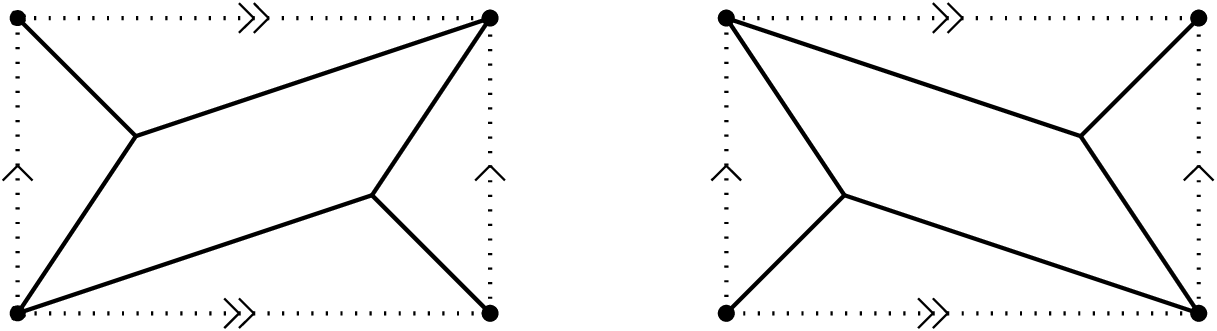}
\caption{Horizontal strip decompositions in case $g=1$, $m=(2)$.\label{Fig:tor}}
\end{figure}
\end{center}
The possible horizontal strip decompositions are shown in Figure \ref{Fig:tor}. Note that the inverse map (which is a rotation by $\pi$ on the diagram) preserves each of these decompositions, and  acts via a non-trivial automorphism of each horizontal strip.
\end{Example}

\subsection{Gluing strips}
\label{glu}

In this section we  prove Proposition \ref{ivans}.

Let $\h\subset \C$ be the upper half-plane, and take $z\in \h$.   We define the standard  complete horizontal strip of period $z$ to be the region 
\[
\barC(z)=\{ 0 \leq \Im(t)\leq \Im(z)\} \subset \bC, 
\]
with  two marked points on its boundary at $\{0,z\}$. We equip the interior $C(z)\subset\barC(z)$ with the  
 quadratic differential $dt^{\otimes 2}$.     Similarly, the standard complete half-plane $\barC(\infty)$ is the region $\{\Im(t) \geq  0\}$, equipped with the differential $dt^{\otimes 2}$ in its interior, and with a single   marked point at $0$.

For any two elements $w,z\in \h$ there is a  diffeomorphism
\[\theta_{w,z}\colon \barC(w)\isom \barC(z),\]
preserving the marked points on the boundary, and
with the further property that in a neighbourhood of each of the two  boundary components of $\barC(w)$ it is given by a translation in $\C$. To be completely definite, we can define
\[\theta_{w,z}(t)=t+\eta(\Im(t)/\Im(w))\cdot (z-w),\]  
where  $\eta\colon [0,1]\to [0,1]$ is some  smooth function satisfying $\eta([0,\frac{1}{4}])=0$  and $\eta([\frac{3}{4},1])=1$.

When $z\in \h$ there is a single non-trivial  automorphism of  $\barC(z)$  preserving the differential and the marked points, namely $t\mapsto z-t$.
We can ensure that the diffeomorphisms $\theta_{w,z}$ we have constructed commute with these non-trivial automorphisms by insising that the function $\eta$ satisfies $\eta(t)+\eta(1-t)=1$.

Let $\phi$ be a saddle-free  GMN differential  on a compact Riemann surface $S$ such that $\Crit_\infty(\phi)$ is non-empty. Thus $\phi$ determines a decomposition of the underlying smooth surface $\S$ into horizontal strips and half-planes.  The restriction of the differential $\phi$ to a horizontal strip $h_i$ is equivalent to the standard differential $dt^{\tensor 2}$ on the  standard cell $C(z_i)$ via an isomorphism $f_i\colon C(z_i)\to h_i$. This extends to a continuous map \[\bar{f_i}\colon \barC(z_i)\to \S,\]   and composing with a translation we can ensure that it takes the marked points $\{0,z_i\}$ to the finite critical points on the boundary of $h_i$. The four boundary half-edges of $\barC(z_i)$ are then taken to the  separating trajectories forming the boundary of $h_i$.

To build a differential $\phi$ on $\S$ with the same horizontal strip decomposition, and arbitrary periods $w_i$, introduce diffeomorphisms \[g_i=f_i\circ \theta_{w_i,z_i}\colon C(w_i)\to h_i.\]  Pushing forward the complex structure and quadratic differential from $C(w_i)$ using $g_i$ defines a new complex structure and differential $\psi$ on the  strips $h_i$, and this trivially extends over the separating trajectories and finite critical points since it agrees with the old one $\phi$ in a neighbourhood of these points. Note that we leave the half-planes completely unchanged.

We must now show that the new complex structure extends over the poles of $\phi$. First note that the function $\theta_{w_i,z_i}$ is invariant under translations in the real direction in $\barC(w_i)$, and hence its derivatives are bounded on $\barC(w_i)$. It follows that there is a  bound
\begin{equation}
\label{spring}(1/R)\cdot |\phi|\leq  |\psi|\leq R\cdot |\phi|,\end{equation}
where we consider both sides as functions on the tangent bundle to $S^\op=S\setminus\Crit_\infty(\phi)$, and $R>1$ is a constant depending only on the periods $w_i$ and $z_i$.

Take a small punctured disc $U\subset \S$ centered at a pole $p$ of $\phi$, and consider the complex structures $U(\phi)$ and $U(\psi)$ induced by the two differentials. Thus $U(\phi)$ is biholomorphic to the standard  punctured disc $D^*$, and we would like to know that this is also the case for $U(\psi)$.  By the Riemann mapping theorem,  $U(\psi)$ is biholomorphic to some annulus \[\{r_1<|z|<r_2\}\subset \C.\] We can compute the modulus $(1/2\pi)\, \log(r_2/r_1)\in [0,\infty]$ using extremal length \cite{FM}, and the inequalities \eqref{spring} show that this gives the same result as for $U(\phi)$. Hence $U(\psi)$ is also biholomorphic to a punctured disc, and  so we can extend the new complex structure over $p$.  Applying the inequalities \eqref{spring} again then  shows that $\psi$ extends to a meromorphic function at $p$ with the same pole order as $\phi$.

The above argument proves that the map $\pi_U$ of Proposition \ref{ivans} is surjective. To prove that it is injective, suppose that two differentials $(S_i,\phi_i)$ have the same  horizontal strip decomposition  and  the same periods $Z_{\phi_i}(\alpha_{h_j})$. Note that the restrictions of $(S_i,\phi_i)$ to the interior of a given horizontal strip $h_j$ are equivalent, via a biholomorphism which extends continuously over the boundary of the strip. Glueing these maps together gives a homeomorphism $f\colon S_1\to S_2$ which  is biholomorphic on the interior of each strip. It follows that $f$ is in fact a biholomorphism, and  since  the meromorphic sections $f^*(\phi_2)$ and $\phi_1$ coincide on an open subset, they must be equal.


\subsection{Period co-ordinates}
\label{per}

We can now prove that (with certain exceptions) the period map \eqref{periodmap} is a local isomorphism.  
Let us fix a genus $g\geq 0$ and a collection of $d\geq 1$ integers $m=\{m_i\}$.
We shall need the following easy corollary of Proposition \ref{stay}.

\begin{lemma}
\label{opener}
Suppose that at least one $m_i\geq 2$. Then the subset $B_0\subset \Quad(g,m)$ of saddle-free differentials is open  and has non-trivial intersection with every $S^1$-orbit.
\end{lemma}

\begin{pf}

 By Lemma \ref{prop:NoClosed}, a GMN differential with an infinite critical point is saddle-free precisely if every trajectory leaving a finite critical point  is separating. Proposition \ref{stay} shows that this condition is stable under small deformations of the differential. Thus $B_0$ is open.

If a GMN differential $\phi$ has a saddle trajectory $\gamma$ then by the definition of the hat-homology class, $Z_\phi(\Hat{\gamma})\in \R_{>0}$.
Consider the subset   $\Theta_\phi\subset S^1$ of phases $\theta$ for which $e^{-i\pi\theta}\cdot \phi$ has a saddle trajectory. 
 Then $\Theta_\phi$ is contained in the set of elements $\arg Z_\phi(\alpha)$ for classes $\alpha\in \hs$ having nonzero period. Thus $\Theta_\phi$ is countable. In particular, the complement of $\Theta_\phi$ is non-empty.
\end{pf}

We remark that the conclusion of Lemma  \ref{opener} is definitely false without the assumption that the differential has an infinite critical point.   For a saddle-free differential on a finite-area surface every trajectory is recurrent;  a fairly simple consequence of that recurrence is that the (still countable) subset $\Theta_{\phi} \subset S^1$ is then dense, see e.g. \cite{Masur} for a detailed proof.

\begin{thm}
\label{perper}
Suppose that the polar type $(g,m)$ is not one of the 6 exceptional cases listed in Lemma \ref{fatman}. Then
the space of framed GMN differentials
$\Quad^\Gamma(g,\m)$  is either empty, or is a complex manifold of dimension $n$, and  the period map
\[\pi\colon \Quad^\Gamma(g,\m) \to \Hom_\Z(\Gamma,\C)\]
 is a local isomorphism of complex manifolds.
\end{thm}

\begin{pf}
We divide into two cases.  In the finite area case when all $m_i=1$ we  appeal to the known result that the period map is a local isomorphism in this setting.\footnote{In fact this appeal can  be avoided: see Remark \ref{forw}.} We do need to be a little bit careful to prove that $\Quad^\Gamma(g,m)$ is a manifold. Consider  the complex manifold ${\mathcal V}(g,m)$ obtained by pulling back the fibration \[\Quad(g,m)\to \Mod(g,d)/\Sym(m)\] to Teichm{\"u}ller space. Pulling back the space of framed differentials gives a local isomorphism $p\colon {\mathcal V}^\Gamma(g,m) \to {\mathcal V}(g,m)$. Taking an open subset $U\subset  {\mathcal V}(g,m)$ and a locally-defined section of $p$ gives local isomorphisms
\[f\colon U\to {\mathcal V^\Gamma}(g,m)\to \Quad^\Gamma(g,m),\]
and taking the composition with the  period map $\pi$ gives a locally-defined period map on ${\mathcal V}(g,m)$. This period map  is known to be a local isomorphism \cite{Veech} and, hence, shrinking $U$ if necessary, we can assume that $f$ is  injective on points.  Our assumption ensures that the space $\Quad^\Gamma(g,m)$ has trivial generic automorphism group so it follows that $f$ is an isomorphism onto its image. Hence $\Quad^\Gamma(g,m)$ is a complex manifold and the period map $\pi$ is a local isomorphism.

Suppose now that some $m_i\geq 2$.  By Lemma \ref{opener} we can use the $S^1$-action and work in a neighbourhood  consisting of saddle-free differentials on a fixed underlying surface $\S$ and with a fixed horizontal strip decomposition. An automorphism of such a differential $\phi$ is a smooth map $f\colon \S\to \S$ satisfying $f^*(\phi)=\phi$. It preserves the framing precisely if it acts trivially on the set of horizontal strips. Assume that $f$ is not the identity.  When pulled-back to a standard strip $C(z_i)$ it must then act by $t\mapsto z_i-t$. But the construction of Section \ref{glu} shows that $f$ then preserves all differentials with the same horizontal strip decomposition as $\phi$. Since $\Quad^\Gamma(g,m)$ is assumed to have trivial generic automorphism group, this is impossible. Hence $\Quad^\Gamma(g,m)$ is a manifold. The result now follows from Propositon \ref{ivans}. The horizontal strip decomposition is locally-constant on $B_0$, so the subset $U$ appearing there is open. The map $\pi_U$ is certainly  holomorphic because its components are periods of the spectral cover, which varies holomorphically with $\phi$.  Since $\pi_U$  is also bijective, it is an isomorphism. 
\end{pf}

%





\section{Stratification by number of separating trajectories}
\label{sadred}

This rather technical section contains some further results concerning the trajectory structure of   GMN differentials.   We focus particularly on  the stratification of the space $\Quad(g,m)$ by differentials with a fixed number of separating trajectories.  Throughout,   we fix a genus $g\geq 0$ and a polar type $m=\{m_i\}$ such that  all  $m_i\geq 2$, and consider differentials in $\Quad(g,m)$. In particular, all differentials are complete and have at least one infinite critical point.

\subsection{Homology classes of saddle trajectories}

Let $\phi$ be a  complete GMN differential on a Riemann surface $S$. Recall from Section \ref{l} that every saddle trajectory $\gamma$  has an associated hat-homology class $\Hat{\gamma}\in \hs$. We say that two saddle trajectories $\gamma_1$ and $\gamma_2$ are \emph{hat-homologous} if $\Hat{\gamma_1}=\Hat{\gamma_2}$. More generally, we say that $\gamma_1$ and $\gamma_2$ are \emph{hat-proportional} if \[\Z_{>0}\cdot \Hat{\gamma_1}=\Z_{>0}\cdot\Hat{\gamma_2}\subset \hs.\]
  Recall from Section \ref{l} that a saddle connection is called closed if its two endpoints coincide. The following result\footnote{To simplify notation, for the purposes of Lemma \ref{alport} and its proof, we will extend the definition of a ``closed" saddle trajectory to include one from a critical point to itself, i.e. we allow the boundary of a ring domain as well as the trajectories in the interior of a ring domain.}
, which is the analogue in our situation of a result of  \cite{MZ}, relies essentially on the assumption that all finite critical points of $\phi$ are  simple zeroes.   

\begin{lemma}
\label{alport}
Suppose that $(S, \phi)$ admits a pair   of distinct hat-proportional saddle trajectories $\gamma_1,\gamma_2$. Then one of the following cases holds: 
\begin{enumerate}
\item[(i)] The $\gamma_i$ are hat-homologous and closed and form the two boundary components of a non-degenerate ring domain.
\smallskip

\item[(ii)] The surface $S$ contains a separating non-degenerate ring domain $A$, bounding on one side an open genus one subsurface $T\subset S$ containing no critical points of $\phi$. The boundary component of $A$ adjoining $T$  is a union of two non-closed saddle trajectories $\nu_1, \nu_2$ and either \smallskip
\begin{itemize}
\item[(iia)] the second boundary of $A$ consists of more than one saddle trajectory; then $\{\gamma_1,\gamma_2\} = \{\nu_1,\nu_2\}$ and the $\gamma_i$ are hat-homologous;
\smallskip
\item[(iib)]   the second boundary of $A$ is a closed saddle trajectory $\mu$; then     $2\Hat{\nu}_1=\Hat{\mu} = 2\Hat{\nu}_2$ and $\{\gamma_1, \gamma_2\}\subset  \{\mu,\nu_1,\nu_2\}$.\end{itemize}\smallskip

\item[(iii)] The surface $S$ is a torus and $\phi$ has a unique pole $p$, which has order $2$ and real residue.  The $\gamma_i$ are hat-homologous, have distinct endpoints, and together form the boundary of the degenerate ring domain enclosing $p$.   \end{enumerate}
\end{lemma}

The three cases are illustrated in Figure \ref{Fig:Hatprop}. In the cases (ii) and (iii) the torus subsurface carries an irrational foliation, i.e. the interior of the torus is a spiral domain. The final paragraph of Section \ref{busp} explains a construction of quadratic differentials illustrating these cases.

\begin{center}
\begin{figure}[ht]
\includegraphics[scale=0.4]{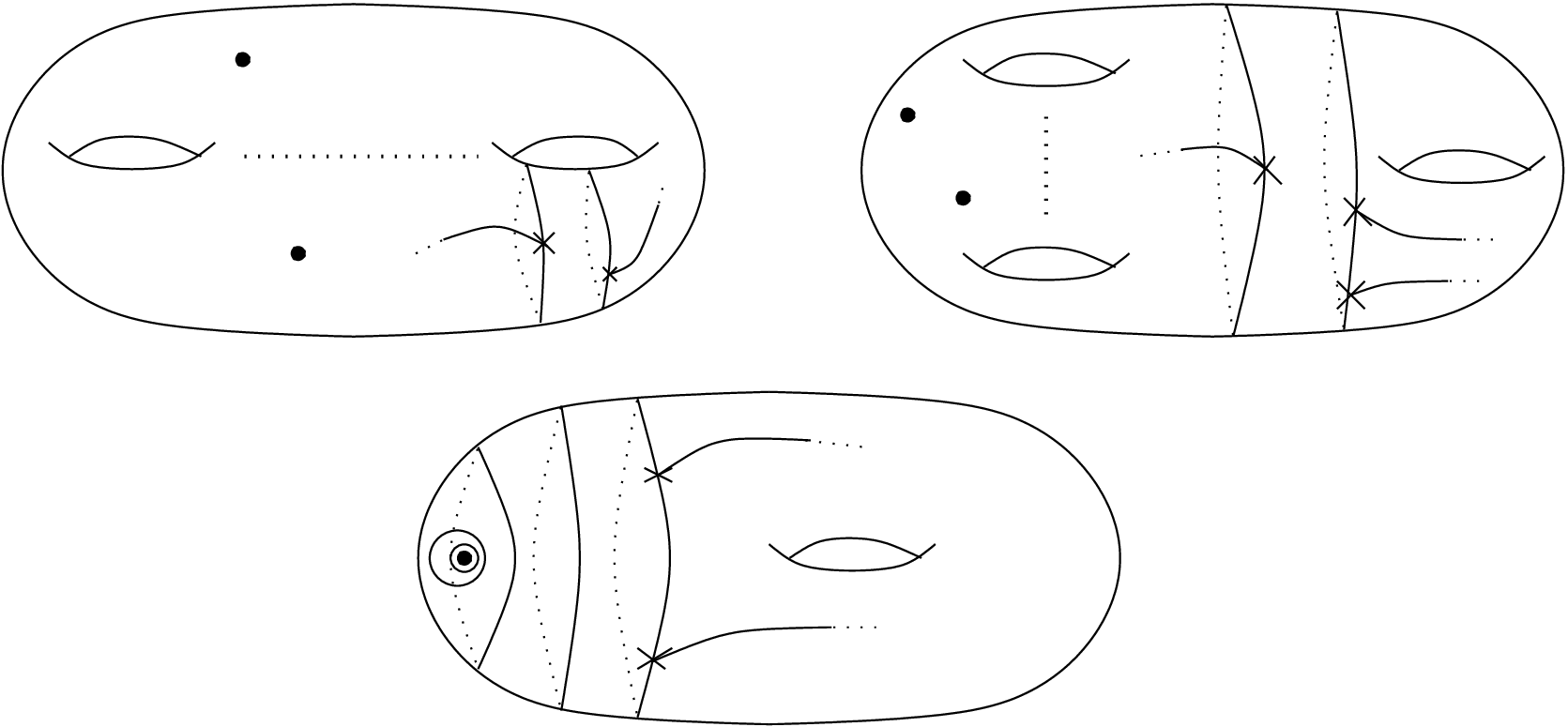}
\caption{Configurations of hat-proportional saddle trajectories.\label{Fig:Hatprop}}
\end{figure}
\end{center}

\begin{pf} 
We divide the proof into three cases, depending on how many of the saddle trajectories $\gamma_i$ are closed.

{\bf Case (1).} First suppose that $\gamma_1$ and $\gamma_2$ are both closed. They must then be disjoint.
Each $\gamma_i$ is one boundary component of a ring domain $A_i$. If both $A_i$  are degenerate,  then $\Hat{\gamma_i}=\beta_{p_i}$ for different double poles $p_i$, and these classes are  linearly independent by Lemma \ref{mov}.  
Thus we can assume that $A_1$ is non-degenerate. Then, as in Figure \ref{ring},
we can write $\Hat{\gamma}_1=\Hat{\alpha}_+ -\Hat{\alpha}_-$, with $\alpha_\pm$ being saddle trajectories contained entirely in the  ring domain $A_1$, and satisfying $\Hat{\alpha}_+\cdot \Hat{\alpha}_-=\pm 2$. But then $\Hat{\alpha}_+\cdot\Hat{\gamma}_1=\pm 2$ also, so if the $\gamma_i$ are hat-proportional, then $\gamma_2$ must meet $\alpha_+$, and hence must be the other boundary component  of $A_1$. It then follows that the $\gamma_i$ are hat-homologous, and this is case (i) of the Lemma. 
\begin{figure}[ht]
\begin{center}
\includegraphics[scale=0.4]{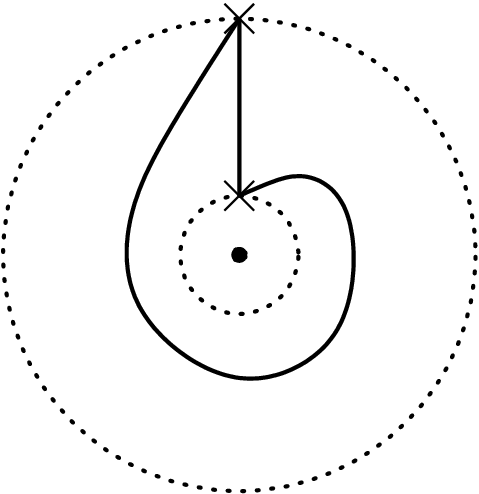}
\end{center}
\caption{A non-degenerate ring domain with a pair of  arcs  $\alpha_\pm$.  \label{ring}}
\end{figure}

{\bf Case (2).} Next suppose that neither $\gamma_1$ nor $\gamma_2$  is closed. If   $\gamma_1\cup \gamma_2$ does not
separate the surface $S$ then we can find a path $\alpha$  in $S$ connecting poles of $\phi$, whose interior lies in $S\setminus\Crit(\phi)$, and which intersects $\gamma_1$ once, and $\gamma_2$ not at all. If we take $\Hat{\alpha}$ to be one of the two inverse images of this path on the spectral cover, then the Lefschetz pairings are $\langle \Hat{\gamma}_1,\Hat{\alpha}\rangle = \pm 1$ and $\langle \Hat{\gamma}_2,\Hat{\alpha}\rangle = 0$. Hence the $\gamma_i$ are not hat-proportional.

Suppose then that $\gamma_1$ and $\gamma_2$ both have the same pair of endpoints $z_1\neq z_2$, and together
form a separating loop $\gamma$.
Consider
the third trajectories coming out of $z_1$ and $z_2$. If these lie on opposite
sides of $\gamma$ then it is easily seen that $\Hat{\gamma_1} \cdot \Hat{\gamma_2} = \pm 2$ so again the $\gamma_i$ are not hat-proportional. 
Thus we conclude that the loop $\gamma$ is one boundary component of a ring
domain $A$.

 If there are poles on both sides of $\gamma$ one can take a path between such poles meeting $\gamma_1$ exactly once and disjoint from $\gamma_2$, showing the $\gamma_i$ are not hat-proportional. Therefore we may assume that $\gamma$ separates $S$, and bounds on one side a subsurface $T$ containing no poles of $\phi$. If $T$ contains a zero $q$ of $\phi$, we may take a path $\alpha$ from a pole to itself which crosses $\gamma_1$, encircles $q$ once, and returns parallel to itself crossing $\gamma_1$ again, with $\alpha$  globally disjoint from $\gamma_2$.  Let $\hat\alpha$ denote one component of the preimage of $\alpha$ on the spectral cover. Then $\langle \hat{\alpha}, \hat{\gamma}_1\rangle = \pm 2$ and $\langle \hat\alpha, \hat{\gamma}_2\rangle  =0$, so the $\gamma_i$ are not hat-proportional.

 Since the other boundary component of the ring domain $A$ contains zeroes or a double pole of $\phi$, it necessarily lies outside $T$.  Thus a  closed trajectory $\sigma$  inside  $A$  bounds a subsurface $R \supset T$ which contains two zeroes (lying on $\gamma$) and no poles.  Doubling $R$ along $\alpha$ then gives a holomorphic quadratic differential on a closed surface with exactly four simple zeroes, which implies that the subsurface $R$ is a torus.

If the ring domain $A$ is non-degenerate we are now in the setting of  case (ii) of the Lemma, and the classes $\hat{\gamma}_i$ coincide since the preimage of $\gamma$ bounds an unpunctured subsurface of $\hat{S}$. 
On the other hand, if $A$ is a degenerate ring domain centered on a double pole $p$ then we are in the setting of case (iii), and the $\gamma_i$ are hat-homologous with hat-homology class equal to half the residue class $\beta_p$.

{\bf Case (3).} Reordering the $\gamma_i$ if necessary, we may now suppose that $\gamma_1$ is closed, and $\gamma_2$ has distinct endpoints. If $\gamma_1$ is not separating, there is a path from a pole to a pole  which meets $\gamma_1$ once and is disjoint from $\gamma_2$, so the $\gamma_i$ are not hat-proportional.  The same argument applies if $\gamma_1$ separates $S$ into subsurfaces each of which contain poles, so we may assume that $\gamma_1$ bounds a subsurface $R$ containing no pole of $\phi$.  

If the interior of $\gamma_2$ lies inside $S \setminus R$ then we may take a path $\alpha$ which goes from a pole, around one end-point of $\gamma_2$, and back to the pole parallel to itself, and which is entirely disjoint from $\gamma_1$. The Lefschetz pairing argument as above then implies that the $\gamma_i$ are not hat-proportional. We may therefore assume that $\gamma_2$ lies inside the subsurface $R$ bound by $\gamma_1$.

If $R$ contains some zero $q$ not lying on either $\gamma_i$, we may pick a closed path $\alpha$, disjoint from $\gamma_2$, which starts at a pole of $\phi$, crosses $\gamma_1$, encircles $q$, and then returns parallel to itself crossing $\gamma_1$. The Lefschetz pairings again imply that the $\gamma_i$ are not hat-proportional.  We may therefore assume that $R$ contains no zeroes other than those lying on $\gamma_1 \cup \gamma_2$. 

The curve $\gamma_1$ is a boundary component of some ring domain $A$.  Suppose $A$ is non-degenerate, and $A$ is contained inside $R$. Then $\gamma_2$ is disjoint from $\gamma_1$.  Let $\sigma$ be a  closed trajectory  in $A$.  Doubling along $\sigma$ yields a surface containing no poles and four simple zeroes, hence of genus 2, which implies that $\gamma_1$ bounds a torus, and $\gamma_2$ connects the two zeroes $z_1,z_2$ lying inside that torus. If both  zeroes $z_i$ lie on the boundary of $A$ then, by considering intersections with closed curves in  the torus, it is easy to see that $\gamma_2$ must be contained in the boundary  of $A$. We are then in case (iib) of the Lemma. Otherwise, the other boundary component of $A$ is comprised of a single saddle $\nu$, and $\gamma_2$ meets $\nu$ in a point, giving a situation as on the left of Figure \ref{Fig:ImpossConfig}.  

\begin{center}
\begin{figure}[ht]
\includegraphics[scale=0.4]{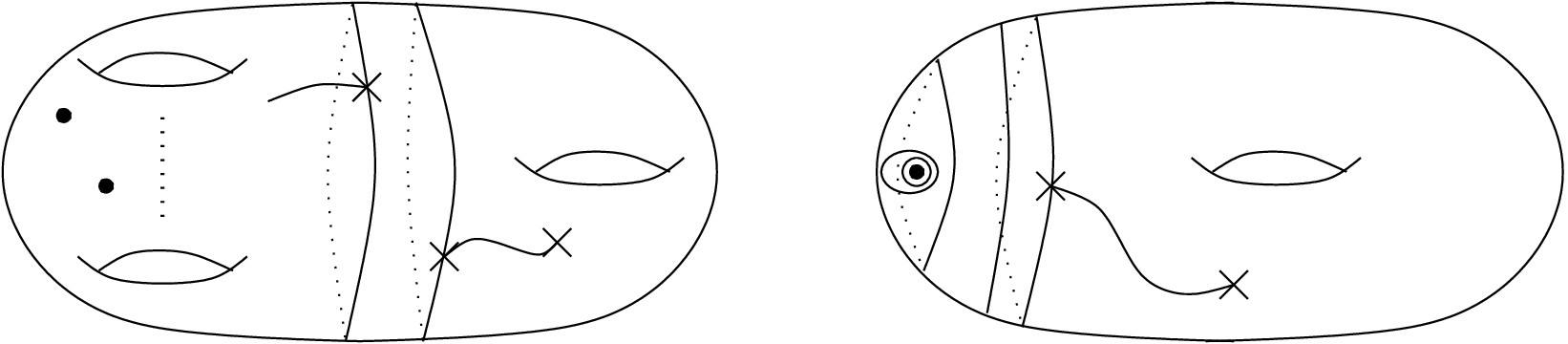}
\caption{Configurations of closed curves which cannot arise as saddle trajectories.\label{Fig:ImpossConfig}}
\end{figure}
\end{center}

We claim this configuration cannot occur. Indeed, if it could, one could replace everything outside $\sigma$ with a degenerate ring domain, yielding the right-hand picture of Figure \ref{Fig:ImpossConfig}. But, as in Example \ref{ex}, any  differential on a torus with a single double pole and  simple zeroes is invariant under  an involution of the torus which permutes the two zeroes, and this rules out the asymmetric trajectory structure shown.\footnote{Alternatively, one could collapse a zero into the double pole using the local surgery from  Section \ref{busp} below.  This would yield a quadratic differential on a torus with a simple pole and a single zero, but as in Example \ref{nosu}, no such exists.}

Suppose next that the ring domain $A$ is non-degenerate, but not contained inside $R$. The third half-edge at the zero on $\gamma_1$ then enters $R$, and we can double along a closed trajectory  $\sigma$ in $A$ to get a surface with at most six zeroes and no poles, hence with at most four zeroes. It follows again that $R$ is a torus, and that $\gamma_2$ must intersect $\gamma_1$ at one point. This gives the same local configuration of saddles as in the previous case, and by gluing in a degenerate ring domain along $\sigma$ one obtains the same contradiction as before.

Finally, if $A$ is degenerate, one arrives directly at the second picture of Figure \ref{Fig:ImpossConfig}, and that again yields a contradiction.  This then completes the proof.
\end{pf}

\subsection{Stratification}
\label{strat}

Let $\phi$ be a GMN differential  on a Riemann surface $S$ defining a point in $\Quad(g,m)$. Note that we are assuming that all $m_i\geq 2$ so $\phi$ has no simple poles, and at least one infinite critical point.
 Since exactly 3 horizontal trajectories emerge from each zero of $\phi$, there is an equality
\[r_\phi + 2 s_\phi  +t_\phi= k, \quad k=3\,|\Zer(\phi)|=3(4g-4+\sum_{i=1}^d m_i),\]
where $r_\phi$ is the number of trajectories that are recurrent in one direction, but tend to a zero in the other, $s_\phi$ is the number of saddle trajectories, and $t_\phi$ is the number of separating trajectories. Define subsets
  \[B_p=\{\phi\in \Quad(g,\m): r_\phi + 2s_\phi\leq p\}\]
Note that  $B_0=B_1$ is precisely the set of saddle-free differentials.  Indeed, by  Lemma \ref{prop:NoClosed}, a differential having no saddle trajectories has no recurrent trajectories either. We call the elements of $B_2$ \emph{tame differentials}; such differentials have at most one saddle trajectory.

\begin{lemma}
\label{open}
The subsets $B_p\subset\Quad(g,\m)$  form an increasing chain of dense open subsets
\[B_0= B_1\subset B_2\subset \dots\subset B_{k}=\Quad(g,\m).\]
\end{lemma}

\begin{pf}
This is very similar to the proof of Lemma \ref{opener}. 
Since $B_p$ is the subset of differentials for which $t_\phi\geq k-p$, the statement that $B_p$ is open is equivalent to the condition that the function $t_\phi$ is lower semi-continuous. This follows from Proposition \ref{stay}. If a differential $\phi$ has a saddle trajectory $\gamma$ then $Z_\phi(\Hat{\gamma})\in \R$ by the definition of the hat-homology class.  In local period co-ordinates the complement of $B_0$ is therefore contained in a countable union of real hyperplanes.\end{pf}

Define  $F_p=B_p\setminus B_{p-1}$ for $p\geq 1$, and set $F_0=B_0$. There is a finite stratification
 \[\Quad(g,\m)=\bigsqcup_{p= 0}^{k}F_p\]
 by the locally-closed subsets $F_i$. The stratum $F_1$ is empty, and  differentials in $F_2$ have exactly one saddle trajectory.

 We call a GMN differential $\phi$ \emph{generic} if the periods of non-proportional  elements of the lattice $\hs$ define distinct rays in $\C$. More precisely, the condition is that   for all $\gamma_1,\gamma_2\in \hs$ there is an implication
\[\R\cdot  Z_\phi(\gamma_1) = \R\cdot  Z_\phi(\gamma_2) \implies \Z \gamma_1 = \Z \gamma_2.\]
It is easy to see that generic differentials are dense in $\Quad(g,m)$: in local period co-ordinates the complement of the set of such differentials is contained in a  countable union of real submanifolds cut out by relations of the form $Z_\phi(\gamma_1)/Z_\phi(\gamma_2)\in \R$.

We say that a differential $\phi$  is \emph{0-generic} if the sublattice
\[\{\gamma\in \hs : Z_\phi(\gamma)\in \R\}\subset \hs\]
has rank $\leq 1$. This implies in particular that all saddle trajectories for $\phi$ are hat-proportional. Clearly,  a differential $\phi$ is generic precisely if all elements of its $S^1$-orbit are 0-generic.


\subsection{Perturbing saddle trajectories}

Let $\phi_0$ be a GMN differential  on a Riemann surface $S$ defining a point \[\phi_0\in F_p\subset \Quad(g,m)\]
for some $p\geq 2$ (recall that $F_1$ is empty). Our aim in this section and the next is to show that  in period co-ordinates in a neighbourhood of $\phi_0$,  the closed subset $F_p\subset B_p$ is contained in  a real hyperplane.  
We begin by considering the case when $\phi_0$ has saddle trajectories lying in the boundary of a horizontal strip or half-plane.
 
\begin{prop}
\label{pe}
Suppose that $\phi_0\in F_p$ has a half-plane or horizontal strip with a boundary component containing precisely $s\geq 1$ saddle trajectories $\gamma_i$.
Let \[\alpha=\sum_{i=1}^s \Hat{\gamma_i}\in \Hat{H}(\phi_0)\] be the  sum of the corresponding hat-homology classes. Then there is an open  neighbourhood $\phi\in U\subset B_p$ such that
\[\phi\in U\cap F_p \implies  Z_\phi(\alpha)\in \R.\]
\end{prop}

\begin{pf}
Considered as a subset of $S$, the half-plane or horizontal strip $h$ is an open disc whose boundary is a closed curve (not necessarily embedded) made  up of saddle trajectories and separating trajectories of $\phi_0$. By Propositions \ref{persist} and \ref{stay}, if $U$ is small enough, these trajectories deform continuously with the differential $\phi\in U$. The resulting deformed curve therefore also cuts out a disc in the surface $S$. 

Integrating $\sqrt{\phi}$ inside this region gives a conformal mapping into $\C$ which is a continuous perturbation of the horizontal strip or half-plane $h$.
\begin{figure}[ht]
\label{saddlechain}
\begin{center}
\includegraphics[scale=0.5]{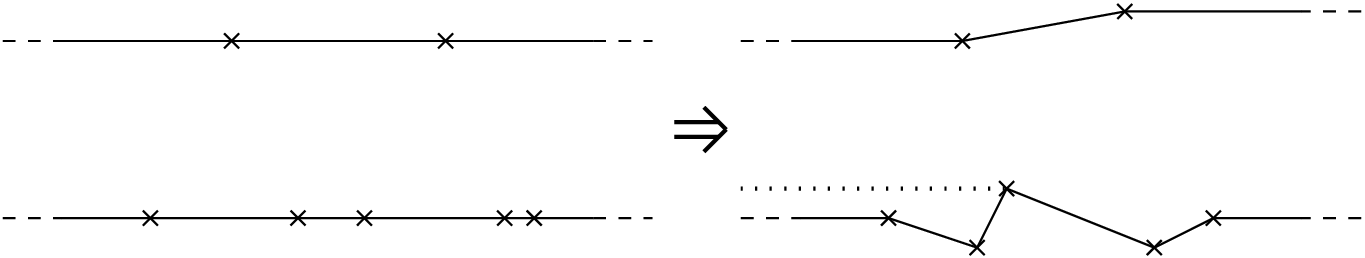}
\end{center}
\caption{Perturbing a horizontal strip\label{last}.}
\end{figure}
 The boundary of the image region in $\C$ consists of straight lines connecting the images of the critical points of the differential. If the region in question is a horizontal strip there are two boundary components; composing with the map $z\mapsto -z$  we may assume that the saddle trajectories $\gamma_i$ occur in the lower one.

Order the saddle trajectories from left to right (i.e. in anti-clockwise order around the boundary) and define real numbers \[y_i=\Im Z_\phi( \Hat{\gamma_1}+\dots +\Hat{\gamma}_i), \quad 1\leq i\leq s.\]
These numbers  give the height of the vertices of the boundary of the perturbed strip, relative to the first vertex.
In particular $y_s=\Im Z_\phi(\alpha)$. Note that  the class $\alpha$ is  definitely non-zero since $ Z_{\phi_0}(\Hat{\gamma}_j)\in \R_{>0}$. 

Suppose that $\phi\in U\cap F_p$.  This implies that if a horizontal arc emerging from a zero  forms part of a non-separating trajectory for $\phi_0$, then the same must be true for the corresponding  arc in $\phi$.  Working from the left, the first vertex with positive height $y_i$ has a ray escaping to the pole on the left, which previously formed a saddle trajectory (see Figure \ref{last}). Thus we must have  $y_i\leq 0$ for all $i$. Given this, if we also have   $y_s<0$ then the last vertex with $y_i=0$  has a ray escaping  to the pole on the right; if none of the vertices has  height $y_i=0$ then the very first vertex has such a ray. We conclude that we must also have $y_s=0$. 
 \end{pf}

Note that we actually proved more, namely that  if $\phi\in U\cap F_p$ then
$y_j\leq 0$ for  $1\leq j\leq s$.


\subsection{Saddle reduction}
\label{s}
As in the last section, let $\phi_0$ be a GMN differential  on a Riemann surface $S$ defining a point \[\phi_0\in F_p\subset \Quad(g,m)\]
for some $p\geq 2$. We shall call a saddle-connection \emph{borderline} if it lies in the boundary of a horizontal strip, half-plane or  degenerate ring-domain. 

The next result is analogous to Proposition \ref{pe} and deals with the case of saddle trajectories lying in the boundary of a degenerate ring-domain.

\begin{lemma}
\label{encircle}
Suppose that $\phi_0\in F_p$ contains a degenerate ring domain $A$ centered on a  double pole $p$. Then there is an open  neighbourhood $\phi_0\in U\subset B_p$ such that
\[\phi\in U\cap F_p \implies  Z_\phi(\beta_p)\in \R.\]
\end{lemma}

\begin{pf}
We can choose $U$ so that we can reach any point by first deforming $\phi_0$ maintaining the condition $Z_\phi(\beta_p)\in \R$, and then applying the $S^1$-action.
When  $Z_\phi(\beta_p)\in \R$  the pole $p$   still lies in  a degenerate ring domain. Thus it is enough to deal with rotations. The boundary of $A$ consists of a union of saddle trajectories.  To understand trajectories for the rotated differential it is equivalent to consider non-horizontal trajectories for $\phi$. It is clear that some of these will fall into the pole $p$.
\end{pf}

Consider the closed subsurface with boundary  $S_+\subset S$ which is the closure of the union of the horizontal strips, half-planes and degenerate ring domains. Consider also  the complementary closed subsurface $S_-\subset S$ which is the closure of the union of the spiral domains and non-degenerate ring domains.
It is easy to see that these two surfaces $S_\pm$ meet along a collection of simple closed curves made up of borderline  saddle-connections.

Note that all infinite critical points of $\phi_0$ are contained in the interior of $S_+$, and since a GMN differential has a non-empty collection of poles, and we are assuming that all $m_i\geq 2$, it follows that $S_+$ is  non-empty. 

\begin{prop}
\label{more}
Take $p\geq 2$ and fix a point $\phi_0\in F_p\subset \Quad(g,m)$. Then there is a neighbourhood $\phi_0\in U\subset B_p$ and   a  nonzero class $\alpha\in \hsinput{0}$ such that
\[\phi\in U\cap F_p \implies  Z_\phi(\alpha)\in \R.\]
\end{prop}

\begin{pf}
 Since $p\geq 2$ there is at least one saddle trajectory for $\phi_0$. It follows that there must be at least one borderline saddle trajectory. Indeed, any saddle trajectory in $S_+$ is borderline, and if $S_-$ is non-empty then $S_-$ and $S_+$ are separated by borderline saddle trajectories. Combining Propositions \ref{pe} and \ref{encircle} therefore gives the result. 
\end{pf}

It follows that, shrinking $U$ if necessary, we can find
 a constant $r>0$  such that
\[e^{i\pi\theta}\cdot \phi \in B_{p-1}\text{ when }0<|\theta|<r\text{ and }\phi\in U\cap F_p.\]
Thus we can always move to a larger stratum  by small rotations of the differential.


\subsection{Ring-shrinking}
\label{rs}
The assumption that a point $\phi\in \Quad(g,m)$ is generic  gives no restriction on which stratum $F_p$ the differential $\phi$ lies in: although all saddle trajectories  are hat-proportional, $\phi$ could  well have a ring domain dividing the surface into two parts, one containing all the poles, and the other consisting of a spiral domain containing some large number of  recurrent trajectories. For this reason it will be important in what follows to use the construction of Section \ref{cp} to eliminate  ring-domains by shrinking them to a closed curve.

Recall that a ring domain is strongly non-degenerate if its boundary consists of two pairwise disjoint, simple, closed curves.   The \emph{width} of a non-degenerate ring domain is the minimal length of a path connecting the two boundary components.  The width is a strictly positive real number; by a ring domain of width zero we mean  any simple closed curve which is a union of saddle trajectories, and which is not a boundary component of a ring domain of strictly positive width. The following result, which will be used in Section \ref{sadred}, shows that any strongly non-degenerate ring domain may be shrunk to width zero. The result is illustrated in Figure \ref{Fig:ShrinkRingDomain}.

\begin{center}
\begin{figure}[ht]
\includegraphics[scale=0.5]{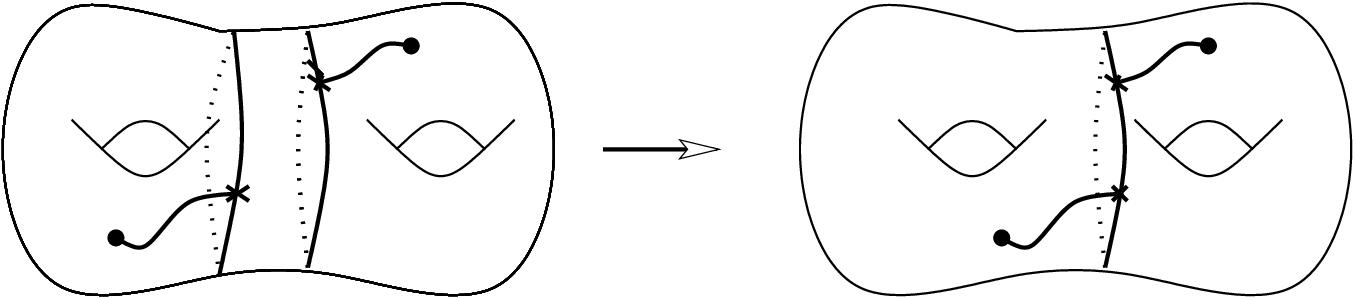}
\caption{Shrinking a ring domain to have width zero.\label{Fig:ShrinkRingDomain}}
\end{figure}
\end{center}

\begin{prop} \label{endofwallmove}
Suppose that a differential $(S_1,\phi_1)\in  \Quad(g,m)$ contains a strongly non-degenerate ring domain $A$ of width $w>0$. Then there is a continuous family  $(S_t,\phi_t)\in \Quad(g,m)$ parameterized by  $t\in [0,1]$, such that each surface  $S_t$ contains a ring domain $A_t$  of width $t.w$, strongly non-degenerate if $t>0$, 
and there are equivalences
\[(S_t\setminus \bar{A_t}, \phi_t|_{S_t\setminus \bar{A_t}})\isom (S \setminus \bar{A},\phi|_{S \setminus \bar{A}}).\]
\end{prop}

\begin{proof}
The non-degenerate ring domain $A$ is equivalent to a region \[\{a<|z|<b\}\subset \C\text{ equipped with }\phi(z)=r\cdot  dz^{\tensor 2}/z^2\] for some $r\in \R_{<0}$. The only invariants are the width, which is $w=\log(b/a)$, and the  length of the two boundary components, which is $2\pi \sqrt{r}$.
For $t\in (0,1)$ we define $A_t$ to be the ring domain with the same length boundary components as $A$, but with width $t\cdot w$. 
We define the surface $S_t$ by glueing $A_t$ into $S\setminus A$ using Lemma \ref{Lem:GlueBorders}.  There is a choice of gluing, since one may rotate one boundary component relative to the other by an angle $\theta$.  The end-point surface $S_0$ is again constructed using Lemma \ref{Lem:GlueBorders}, by directly gluing the two components of $S\setminus A$. To ensure that the resulting differential $\phi_0$ has simple zeroes we may need to take the rotation parameter $\theta$ to be nonzero. 
\end{proof}


\subsection{Walls have ends}
\label{movewalls}

We have shown above that for $p\geq 2$ the stratum $F_p\subset \Quad(g,m)$  is contained in a real hyperplane  in local period co-ordinates. It can therefore be thought of as a wall, potentially dividing two different connected components of the open subset $B_{p-1}$. Now we want to go one step further and show that if $p>2$ then these walls   always have ends: we can move along the stratum $F_p$ to get to a point near which the subset $B_{p-1}$ is locally connected.

\begin{prop}
\label{moremore}
Assume that the polar type is not $m=(2)$ and take $p>2$. Then every connected component of $F_p\subset \Quad(g,m)$ contains a point $\phi$ with a neighbourhood $\phi\in U\subset B_p$, as in Proposition \ref{more}, such that $U\cap B_{p-1}$ is connected.
\end{prop}

\begin{pf} Take $\phi_0\in F_p$ and a neighbourhood $U$ as in Proposition \ref{more}. Consider the inclusion \[U\cap F_p\subset \{\phi\in U: Z_\phi(\alpha)\in \R\}\]
If this inclusion is strict,  the wall has a hole in it, and then  $U\cap B_{p-1}$ is connected and we are done. Otherwise, these two subsets are equal, and so staying in the same connected component of $F_p$ we can replace $\phi_0$ with a very close generic differential $\phi_1$.

Suppose  that $\phi_1$ has only one saddle trajectory $\gamma$.  Then, since $p>2$, there must exist recurrent trajectories. The surfaces $S_\pm$ introduced in Section \ref{s} are thus both non-empty, and must meet along $\gamma$. Then $\gamma$ is closed and forms one boundary component of a ring domain $A$, which  has to be degenerate, since there are no saddle trajectories to form its other boundary. Thus we conclude that $S_+=A$, and since all poles of $\phi_1$ lie in $S_+$ it follows that $\phi_1$ has a single  pole $p$ of order 2, and that $\alpha$ is the corresponding residue class $\beta_p$.

Suppose then  that $\phi_1$ has more than one saddle trajectory. By the genericity assumption these are all hat-proportional, so they are arranged as in one of the cases of Lemma \ref{alport}. It follows that there are  two possibilities,   corresponding to cases (i) and (iia) of Lemma \ref{alport}: case (iii) is ruled out by the  assumption on the polar type, and case (iib) cannot occur for a 0-generic differential, since not all saddle trajectories appearing are hat-proportional. In particular we see that $\phi_1$ has a unique ring domain $A$, which is strongly non-degenerate, and  whose  boundary consists of either 2 or 3 saddle trajectories. 

By Proposition \ref{endofwallmove}, we can move along a path in $\Quad(g,m)$ in which $A$ shrinks so as to have width 0, but the rest of the differential remains unchanged. It is clear that this path remains in the stratum $F_p$ since any separating trajectory lies outside  $A$ and hence is unaffected by the shrinking process.  At the end of  this operation we arrive at a differential $\phi_2$ with no closed trajectories and either 2 or 3 saddle trajectories $\gamma_i$, which together form a simple closed curve $\gamma$.

 We claim that all the saddle trajectories $\gamma_i$ are borderline. Indeed, if the surface $S_-$ of Section \ref{s}  is empty then all saddle trajectories are borderline, and otherwise the two surfaces $S_\pm$ are separated by a simple closed curve made up of saddle trajectories, which must be $\gamma$. Examining the configuration of trajectories near $\gamma$ in the two cases it is easy to see that exactly two of the $\gamma_i$ must lie in the boundary of
  a single horizontal strip or half-plane $h$.  Proposition \ref{pe}, and the remark following it, shows that there is a neighbourhood $\phi_2\in U\subset B_p$ such that, with appropriate ordering of the $\gamma_i$,
\[\phi\in U\cap F_p \implies y=\Im Z_\phi(\Hat{\gamma}_1+\Hat{\gamma}_2)=0\text{ and } z=\pm \Im Z_\phi(\Hat{\gamma}_1)\leq 0.\]
 Lemma \ref{alport} shows that the saddle trajectories $\gamma_1,\gamma_2$ are not hat-proportional, so the variables $y$ and $z$ form part of a co-ordinate system near $\phi_2$. It follows  that  $U\setminus F_p$ is locally connected near $\phi_2$. \end{pf}


\subsection{Homotopies to tame paths}

In the proof of our main Theorems we shall  need the following consequence of Proposition \ref{moremore}.

\begin{prop} \label{TamePath}
Assume that the polar type  is not $m=(2)$. Then  any path $\beta$ in $\Quad(g,\m)$  connecting two points of  $B_2$ is homotopic relative to its end-points to a path in $B_2$. 
\end{prop}

\begin{pf}
Let us inductively assume that $\beta$ has been deformed so as to lie in $B_p$ for some $p>2$. By Propositon \ref{more} we can cover $\beta$ by open subsets in which $F_p\subset B_p$ is contained in a real hyperplane. We can  then wiggle it a little so that it meets $F_p$ at  a finite number of points  $\phi_i$. We now show how to deform $\beta$ so as to reduce the number $k$ of these points. Repeating the argument, we can  deform $\beta$ to lie in $B_{p-1}$. The result then follows by induction.

To eliminate a point $\phi=\phi_i$ we first  use Proposition \ref{moremore} to construct a path $\delta$ in $F_p$  connecting  $\phi$ to a point $\psi$ where $B_{p-1}$ is locally connected.  Consider paths $\delta_\pm$ obtained by small rotations of $\delta$ in  opposite directions. By  Propositon \ref{more} these can be assumed to  lie entirely in $B_{p-1}$. Inserting these paths into $\beta$ we obtain a homotopic path which crosses $F_p$ at the point $\psi$ instead of $\phi$. Since $B_{p-1}$ is locally connected near $\psi$ we can then deform $\beta$ further and so eliminate one of its intersections with $F_p$.
\end{pf}

The   assumption on the polar type in Propostions \ref{moremore} and \ref{TamePath} is necessary, as we explain in the following remark. 

\begin{remark}
\label{rus}Suppose that the polar type is  $m=(2)$ and consider the holomorphic function
\[Z_\phi(\beta_p)^2\colon \Quad(g,m)\to \C^*.\]
We claim  that this maps the subset $B_2$ into the complement of $\R_{>0}\subset \C^*$. Thus  for paths in $B_2$ the function $Z_\phi(\beta_p)^2$ does not wind around the origin. But by rotating a differential, it is easy to construct paths in $\Quad(g,m)$ for which this function does wind around the origin. Thus it follows that Proposition \ref{TamePath}, and hence also Proposition \ref{moremore}, are false in this case.

To prove the claim note that if $\phi\in B_2$ and  $Z_\phi(\beta_p)\in\R$ then the unique pole  $p$ is contained in a degenerate ring domain $A$. The boundary of  $A$ must then be a single closed saddle trajectory, and the third trajectory leaving the zero on the boundary cannot be a saddle trajectory, or a separating trajectory, or recurrent. This  gives a contradiction.
\end{remark}


\subsection{More on ring-shrinking}

We assume in this section that if $g=1$ then the polar type is not $m=(2)$. 
Suppose that  $\phi_+\in \Quad(g,m)$ is a 0-generic differential with more than one saddle trajectory.  As in the proof of Proposition \ref{moremore}, it follows that $\phi_+$  has a unique ring domain $A$, which is moreover strongly non-degenerate, and we can shrink  $A$ to obtain a  differential $\phi$ with a closed curve $\gamma$ formed of a union of either 2 or 3 non-closed saddle trajectories $\gamma_i$. 

Note that the $\gamma_i$ are the only saddle trajectories for $\phi$. 
Let us  write $\alpha_i=\Hat{\gamma}_i\in \hs$.
Examining Figure \ref{Fig:Hatprop}, it is easily seen that we can order the  $\gamma_i$ so that the
complete set of finite-length  trajectories for $\phi_+$ in the two cases  is as follows:
\begin{itemize}
\item[(J1)] a single ring domain $A$ of class $\alpha=\alpha_1+\alpha_2$ whose boundary components are closed saddle trajectories of the same class;
\smallskip

\item[(J2)] a single ring domain $A$ of class $\alpha=\alpha_1+\alpha_2+\alpha_3$, one of whose boundary components is a closed saddle trajectory of the same class, the other being a union of two non-closed saddle trajectories of equal classes $\alpha_1+\alpha_2$ and $\alpha_3$.
\end{itemize} 
The labelling of the $\gamma_i$ is completely determined if we insist that \begin{equation}
\label{plus} \Im Z_{\phi_+}(\alpha_1)/Z_{\phi_+}(\alpha_2)>0,\end{equation}
and we shall always follow this convention. Note that  in the case (J2) there is a relation $\alpha_1+\alpha_2=\alpha_3$.

Proposition \ref{stay} implies that  for any differential sufficiently close to $\phi$ there are saddle connections deforming each of the saddle trajectories $\gamma_i$.
We shall need the following statement later.

\begin{prop}
\label{jugprop}
Given  a class $\beta\in \hs$,  there is a neighbourhood 
$\phi\in U\subset \Quad(g,m)$ with the following property: if $\phi_-\in U$ satisfies \begin{equation}
\label{minus}\Im Z_{\phi_-}(\alpha_1)/Z_{\phi_-}(\alpha_2)<0,\end{equation}
 and $\gamma$ is a saddle trajectory  for $\phi_-$ with hat-homology class $\beta$, then  $\gamma=\gamma_i$,  for some $i$, and hence $\beta=\alpha_i$.
\end{prop}

\begin{proof}
Consider first the abstract situation in which two saddle trajectories $\gamma_1,\gamma_2$ for a differential $\phi$ meet at a  zero $z$. Assume that the $\gamma_i$ have non-proportional hat-homology classes $\alpha_i$,  and consider differentials on either side of the wall \[\Im Z_\phi(\alpha_1)/Z_\phi(\alpha_2)=0.\]
 As above there are saddle connections deforming each  $\gamma_i$. Consider the union $\gamma_1\cup \gamma_2$ near the zero $z$. Local calculations (see Figure \ref{Fig:LocalNearZero}) show that on one side of the wall this path is a geodesic, whereas on the other side it is not, since there is a shorter path which bypasses the zero $z$.

Consider now the differential $\phi$ obtained by shrinking the ring domain in $\phi_+$. The walls \[\Im Z_\phi(\alpha_i)/Z_\phi(\alpha_j)=0\] for $i\neq j$ all coincide. On the side of this wall defined by \eqref{plus},  none of the unions $\gamma_i\cup \gamma_{i+1}$ is a geodesic, since  the shortest paths in these homotopy classes  cross the ring domain. It follows that on the side of the wall defined by \eqref{minus} each of these unions is  a geodesic.

Suppose for a contradiction that we can find a sequence of differentials $\phi_i$ satisfying \eqref{minus},  each with a saddle connection $C_i$ of class $\beta$, and which tend to $\phi$.  The length of the saddle connections $\ell_{\phi_i}(C_i)=|Z_{\phi_i}(\beta)|$ is bounded, so by Theorem \ref{finini}, passing to a subsequence we can assume that the $C_i$ are all homotopic, and converge to a curve $C$.

By continuity, we now have $\ell_\phi(C)= |Z_\phi(\beta)|$. This   implies that $C$  is a union of saddle trajectories for $\phi$, that  is, a  union of the $\gamma_i$.    But as we just argued, the unique $\phi_i$-geodesic representative in this  homotopy class \emph{is}  the corresponding union of $\gamma_i$, and hence can only be a saddle connection if it is one of the $\gamma_i$.
\end{proof}


\subsection{Juggles}
\label{jug}
We conclude this section with a few brief remarks about the relationship between the ring-shrinking operation of the last few sections and the notion of a `juggle' appearing in  Gaiotto-Moore-Neitzke's paper \cite{GMN2}. This material will not be used later and can be safely skipped.

Suppose that $\phi_+\in \Quad(g,m)$ has a non-degenerate ring domain $A$. The closed trajectories of $A$ have  a certain hat-homology class $\alpha\in \hs$.  Let $\delta\in \hs$ be the class of a  saddle connection in $A$ joining  zeroes of $\phi_+$  lying on different boundary components. Considering  lines of suitable rational slope in the universal cover of $A$ shows that for all $k\in \Z$ there are  saddle connections for $\phi_+$ with hat-homology class $\delta + k \alpha$. In particular, the spectrum $\Theta_{\phi_+} \subset S^1$ of phases $\theta$ for which $e^{i\pi\theta}\cdot \phi_+$ has a saddle trajectory has an accumulation point at $\theta=0$.

By taking differentials $e^{i\pi\theta}\cdot \phi_+$ with  $\theta$ varying near $0$ we can define a path  in $\Quad(g,m)$ with saddle-free endpoints  which crosses infinitely many of the real codimension one walls  that are the local connected components of $F_2$.  We refer to such a path as a juggle path. In Section \ref{wkb} we will associate ideal triangulations to saddle-free differentials; the  triangulations associated to the end-points of our path  will then be related by a particular kind of infinite composition of flips, referred to in \cite{GMN2} as a juggle. The  ring-shrinking move of Proposition \ref{moremore} has the effect of removing the accumulation point at $\theta=0$ in the spectrum $\Theta_\phi\subset S^1$. This allows us to replace certain juggle paths by paths which meet only finitely many walls.

\smallskip
\begin{figure}[ht]
\begin{center}
\includegraphics[scale=0.4]{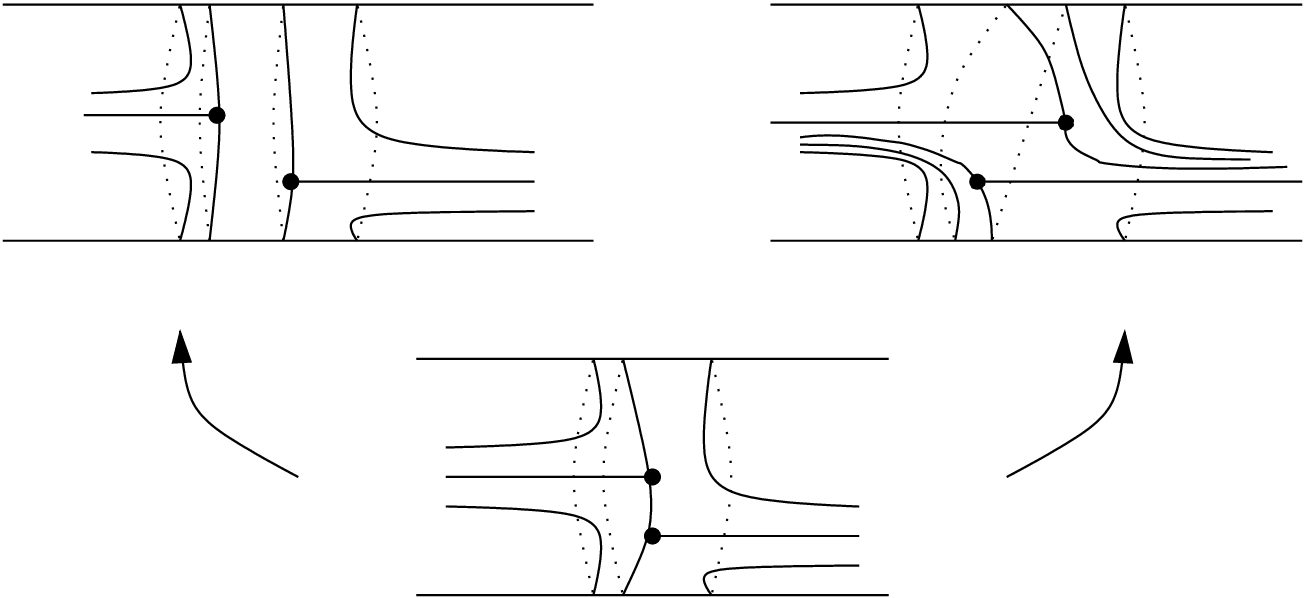}
\end{center}
\caption{Local perturbations of a differential with two saddle trajectories of equal phase; the left perturbation has a ring domain, the right  does not.\label{Fig:JuggleWall}}
\end{figure}

Let us consider  the case when the boundary components of $A$ are both closed saddle trajectories, and the differential $\phi_+$ contains no other finite-length trajectories.
After shrinking we obtain a differential $\phi$ with a closed curve made up of two saddle trajectories $\gamma_1,\gamma_2$ with hat-homology classes $\alpha_1,\alpha_2$ satisfying $\alpha=\alpha_1+\alpha_2$.  The trajectory structure of differentials near $\phi$ satisfying $\Im Z_\phi(\alpha)=0$ is determined by the wall
\[\Im Z_\phi(\gamma_1)/Z_\phi(\gamma_2)=0.\]
Differentials on the $\phi_+$ side of this wall  have a ring domain; differentials on the other side  are saddle-free.   The relevant geometry is illustrated in  Figure \ref{Fig:JuggleWall}.

The representation theory relevant to  juggles  is that of the Kronecker quiver  (see also Example \ref{eg2}).\begin{equation*}
\xymatrix@C=.4em{\bullet^1 \ar^{a_1}@/^/[rrrr] \ar_{a_2}@/_/[rrrr]&&&& \bullet^2   }
\end{equation*}
 Let $\A$ be the category of representations of this quiver, and let $S_1, S_2$ be the vertex simple objects, appropriately ordered.
 Stability conditions on $\A$ satisfying \[\Im Z(S_1)/Z(S_2)>0\] have unique stable objects of dimension vectors  $(n,n+1)$ and $(n+1,n)$  for all $n\geq 0$, and also a moduli space  of stable objects of dimension vector $(1,1)$ which is isomorphic to $\PP^1$. In particular, the set of phases of stable objects has an accumulation point.
 On the other hand, if \[\Im Z(S_1)/Z(S_2)<0\] then the only stable objects are the objects $S_i$ themselves. The operation of ring-shrinking is the analogue of moving from a stability condition with $\Im Z(S_1)/Z(S_2)>0$ to one where $\Im Z(S_1)/Z(S_2)=0$. This has the effect of removing the accumulation point in the spectrum of stable phases.


\section{Colliding zeroes and poles: the spaces $\Quad(\S,\M)$}
The spaces of quadratic differentials appearing in our main Theorems do not have fixed polar type; rather the zeroes are allowed to collide with the double poles. This means that we are dealing with spaces which are  unions of strata of the form $\Quad(g,\m)$. It is convenient to  label these spaces by diffeomorphism classes of marked bordered surfaces. For definitions concerning such surfaces see the Introduction or Section \ref{tri} below.

\subsection{Union of strata}
\label{union}

A GMN differential $\phi$ on a compact Riemann surface $S$ determines a marked bordered surface $(\S,\M)$ by the following construction.   To define the surface $\S$ we take the  underlying smooth surface of $S$ and perform  an oriented real blow-up  at each pole of $\phi$ of order $ >2$.  The marked points $\M$ are then the poles of $\phi$ of order $\leq 2$, considered as points of the interior of $\S$, together with the points on the boundary of $\S$ corresponding to the distinguished tangent directions of Section \ref{crit}.

By a quadratic differential on a marked bordered surface $(\S,\M)$ we mean a pair $(S,\phi)$, consisting of a compact Riemann surface $S$ and a GMN differential $\phi$,  whose associated marked bordered surface is diffeomorphic to $(\S,\M)$. We let $\Quad(\S,\M)$ denote the space of equivalence classes of such pairs. 

A marked bordered surface $(\S,\M)$ is determined up to diffeomorphism by the genus of $\S$, the number of punctures, and an unordered  collection of positive integers  encoding the number of marked points on each boundary component. In more concrete terms then, we have
\begin{equation}
\label{dec}
\Quad(\S,\M)=\bigcup_{(g,m)}\Quad(g,m),\end{equation}
where the union is over pairs $(g,m)$, where $g=g(\S)$ is the genus of $\S$,  and  there is one $m_i\in\{1,2\}$ for each puncture $p\in \bP$, and one $m_i=k_i+2$ for each boundary component containing $k_i$ marked points.

Let $(g,m)$ be the unique pair appearing in the decomposition \eqref{dec} for which all $m_i\geq 2$. In the proof of Proposition \ref{rain} we considered a vector bundle 
\begin{equation}
\label{paul}\mathcal{H}(g,m)\lra \Mod(g,d)/\Sym(m),\end{equation}
whose  fibre  over a marked curve $(S,(p_i))$  is the space of global sections of the line bundle $\omega_S^{\tensor 2}(\sum_i m_i p_i)$. The space $\Quad(\S,\M)$ is the open subset of $\mathcal{H}(g,m)$  consisting of  sections  with simple zeroes which are disjoint from the points $p_i$ for which $m_i> 2$. As in the proof of Proposition \ref{rain} it is therefore either empty, or a complex orbifold of dimension $n$.

Recall that a GMN differential is called \emph{complete} if it has no simple poles. 

\begin{lemma}
\label{new}
The subset of complete differentials is an open subset
\begin{equation}
\label{stratum}\Quad(\S,\M)_0=\Quad(g,m)\subset \Quad(\S,\M)\end{equation}
whose complement is a normal crossings divisor.
\end{lemma}

\begin{pf}
Locally on the universal curve over $\Mod(g,d)$ we can trivialise the line bundle $\omega_S^{\tensor 2}(\sum_i m_i p_i)$. Working locally on $\mathcal{H}(g,m)$ we can 
therefore associate to each point  an unordered collection of  complex numbers $\{r_p:p\in \PP\}$ obtained by evaluating the defining section $\phi$ at the marked points $p_i$ for which $m_i\leq 2$. The resulting locally-defined functions $r_p$ are holomorphic on $\mathcal{H}(g,m)$, and the complement of the open stratum  \eqref{stratum} is precisely the vanishing locus of the product of these functions.

Suppose that a point $\phi\in \Quad(\S,\M)$ has $s\geq 1$ simple poles. Then the locally-defined map \[r\colon \Quad(\S,\M)\to \C^s\]
  given by the  functions $r_p$ corresponding to the simple poles of $\phi$ is a submersion at $\phi$. Indeed, using Riemann-Roch,  for each simple pole $p$ of $\phi$ we can find sections of $\omega_S^{\tensor 2}(\sum_i m_i p_i)$ which vanish at all the other simple poles of $\phi$ but not at $p$. Adding linear combinations of such sections to $\phi$ shows that $r$ has a locally-defined section. It follows from this that the complement of the open stratum \eqref{stratum} is a normal crossings divisor.
\end{pf}


\subsection{Signed differentials}
\label{tofin}

Fix a marked bordered surface $(\S,\M)$. 
Although the hat-homology groups $\hs$  form a  local system over the orbifold $\Quad(\S,\M)_0,$ this is not true over the larger orbifold $\Quad(\S,\M)$, since by Lemma \ref{train}, at differentials where a  zero has collided with a double pole the rank of the hat-homology group  drops by one.

A stronger statement is that the local system of hat-homology groups   over $\Quad(\S,\M)_0$ cannot  be extended to a local system on $\Quad(\S,\M)$. The reason is that parallel transport around a differential with a simple pole at a point $p$ changes the sign of  the residue class $\beta_p$ (see the proof of Lemma \ref{se} below).

By a \emph{signed quadratic differential} on $(\S,\M)$ we mean a differential \[(S,\phi)\in \Quad(\S,\M)\] together with a choice of sign of the residue $\res_p(\phi)$ at each puncture $p\in\bP$.  Note that by \eqref{gotcha} this is equivalent to choosing a square-root of the function  $r_p$ of the last paragraph. The set of such signed differentials therefore forms  a smooth complex orbifold 
equipped with a finite map
\begin{equation*}\label{noel}\Quad^{\pm}(\S,\M)\to \Quad(\S,\M)\end{equation*} branched precisely over the complement of the incomplete locus.
We write $\Quad^\pm(\S,\M)_0$ for the open subset of $\Quad^\pm(\S,\M)$ consisting of signed differentials whose underlying differential is complete.

\begin{center}
\begin{figure}[ht]
\includegraphics[scale=0.3]{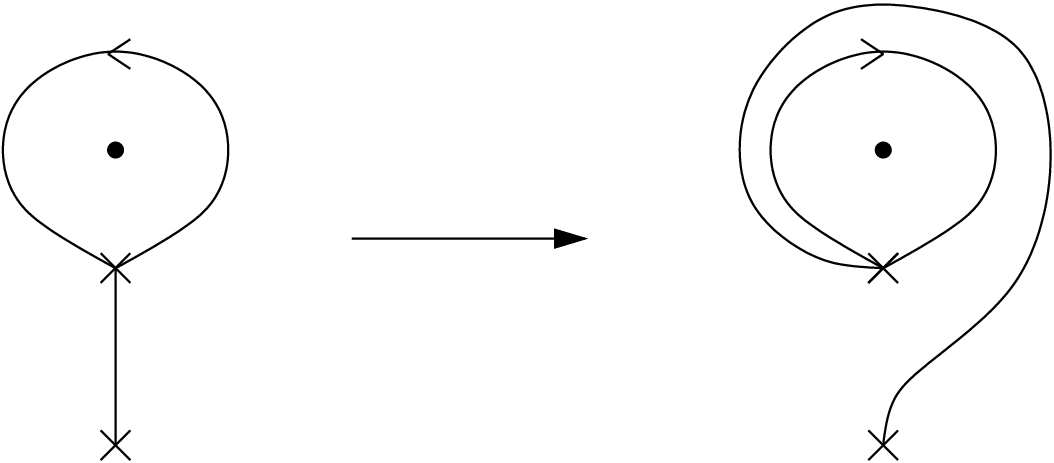}
\caption{The local monodromy as a zero encircles a double pole.\label{Fig:MonodromyOrder2}}
\end{figure}
\end{center}

\begin{lemma}
\label{se}
The local system of hat-homology groups  $\hs$ pulled back to the  {\'e}tale cover  $\Quad^\pm(\S,\M)_0\to \Quad(\S,\M)_0$ extends to a local system on $\Quad^\pm(\S,\M)$.
\end{lemma}

\begin{pf}
We must compute the monodromy of the hat-homology local system  around each component of the boundary divisor consisting of non-complete differentials. 
Consider a differential  $\phi_0$ lying on this divisor, having a single simple pole $p_0$. Nearby complete differentials $\phi$ will have a corresponding double pole $p$ and a simple zero $q$ which have collided to produce $p_0$.

The hat-homology group of $\phi$ is spanned by the hat-homology classes of saddle connections. Saddle connections of $\phi$ not ending at  $q$   correspond canonically to saddle connections of $\phi_0$ not ending at $p_0$,  and their hat-homology classes are therefore unaffected by the local monodromy around $\phi_0$.
Consider  the class  $\alpha_q$ of a  saddle connection ending at $q$, and let $\beta_p$ be  the residue class at $p$. 
The local monodromy of the Gauss-Manin connection\footnote{To check the sign change for the residue class $\beta_p$, it may be helpful to consider the family of differentials $(z-a) dz^{\otimes 2} / z^2$ of residue $4\pi i \sqrt{a}$, as $a$ encircles the origin.} acts on the classes $(\alpha_q,\beta_p)$ by the transformation  \begin{equation}
\label{finito}\alpha_q\mapsto \alpha_q+ \beta_p, \quad \beta_p\mapsto -\beta_p,\end{equation}
see Figure \ref{Fig:MonodromyOrder2}. 
This transformation  has order 2 and hence becomes trivial when pulled-back to the double cover determined by a choice of sign of $Z_\phi(\beta_p)$. \end{pf}

\subsection{Extended hat-homology group}
 Let us consider the quotient orbifold
\begin{equation}
\label{jim}\Quadorb(\S,\M)=\Quad^\pm(\S,\M)/\,\Z_2^{\oPP},\end{equation}
where $\Z_2^{\oPP}$  acts in the obvious way on the signings. Note that this quotient is to be understood in the category of spaces over $\Quad(\S,\M)$, since the  punctures $\PP$ form a non-trivial local system over this space. Practically speaking,  we can locally trivialise this local system $\PP$ on $\Quad(\S,\M)$, define local quotients by the group $\Z_2^{\oPP}$, and then glue these  together to form the global quotient \eqref{jim}.
Note that there is an open inclusion \[\Quad(\S,\M)_0\subset \Quadorb(\S,\M).\]
The only difference between the spaces $\Quad(\S,\M)$ and $\Quadorb(\S,\M)$ is some extra orbifolding along the incomplete locus.

The local system of Lemma \ref{se} decends to the orbifold $\Quadorb(\S,\M)$. The \emph{extended hat-homology group} $\hse$  of a GMN differential $\phi$ is defined to be the fibre of this  local system  at $\phi$.  This group   coincides with the usual hat-homology group $\hs$ precisely if $\phi$ is complete. In general $\hse$ comes equipped with a skew-symmetric pairing and canonically defined residue classes $\beta_p$, one for each simple or even order pole of $\phi$. This data is obtained by parallel transport from a nearby complete  differential.

\begin{lemma}
\label{rich}
For any GMN differential $\phi_0$ there is a canonical group homomorphism
\[q\colon \Hat{H}^e(\phi_0)\to \Hat{H}(\phi_0)\]
whose kernel is spanned over $\mathbb{Q}$ by the residue classes $\beta_p$ corresponding to the simple poles of $\phi_0$.
\end{lemma}

\begin{pf}
Consider  the family of spectral covers $\hS\to S$ defined by differentials $\phi$ in    some small neighbourhood $\phi_0\in U \subset \Quadorb(\S,\M)$. These covers vary holomorphically because the divisor $E$ of formula \eqref{g} varies holomorphically, and $\varphi$ is a holomorphically varying  section.
Note however that the open surface $\hS^\op$ changes discontinuously in general, as it must, since the rank of the hat-homology group drops at differentials with simple poles.

More precisely, when a zero of the differential $\phi$ collides with a double pole $p$, the  infinite critical point $p$ becomes a finite critical point of $\phi_0$  which is moreover a branch-point of the corresponding spectral cover. Thus the two punctures in the surface $\hS$ lying over $p$ which are removed when defining $\hS^\op$   collide and get filled in.

Define a subsurface $\hS'\subset \hS$ by removing from $\hS$ the inverse images of those infinite critical points of $\phi$ which remain infinite critical points for the differential $\phi_0$. 
 The homology groups $H_1(\hS';\Z)^-$ form a local system over $U$ whose fibre at $\phi_0$ coincides with $H_1(\hS^\op;\Z)^-$.
 Over the complete locus, the inclusion $\hS^\op\subset \hS'$ defines a map of local systems
\[q\colon H_1(\hS^\op;\Z)^-\to H_1(\hS';\Z)^-.\]
The same analysis we used to prove  Lemma \ref{mov} shows that the kernel of $q$ is spanned over $\QQ$ by the residue classes $\beta_p$ corresponding to the simple poles of $\phi$. Specialising the map $q$ to the fibres at $\phi_0$ then gives the result.
\end{pf}


 \subsection{Blowing up simple poles}
\label{busp}

In this section we explain a surgery which, whilst not required in the proofs of the main theorems of the paper, helps explain the geometry as zeroes collide with double poles and one passes between different strata in $\Quad(\S,\M)$. 

The surgery involves `blowing up' a simple pole and inserting a metric cylinder (i.e. a disk with differential $r\,dz^{\otimes 2}/z^2$ for some $r\in \bR_{<0}$).
Although as topological surfaces the complement of the inserted cylinder differs from the original surface by a real blow-up, metrically the surfaces are related by slitting a finite length of trajectory and opening up the slit into a boundary component, as indicated in Figure \ref{Fig:LocalModelBlowUp}.   \begin{center}
\begin{figure}[ht]
\includegraphics[scale=0.5]{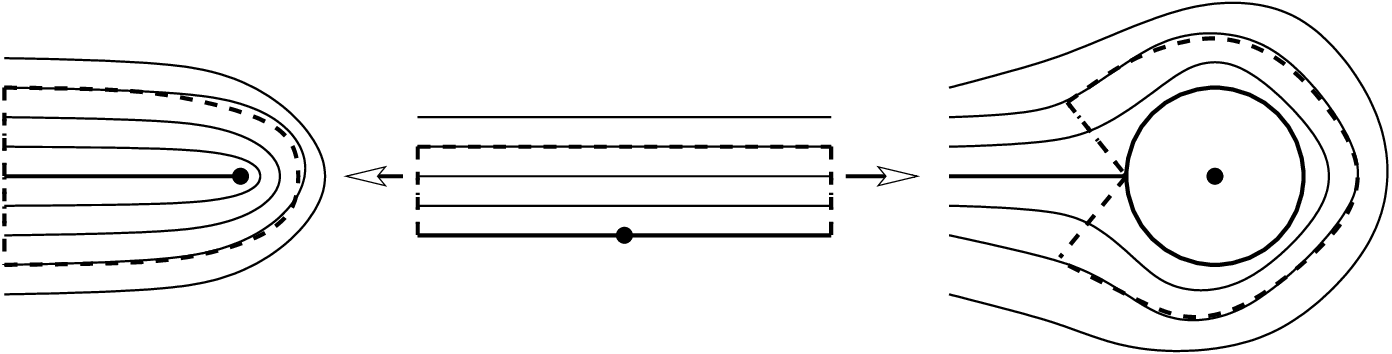}
\caption{Replacing a simple pole with a degenerate ring domain.\label{Fig:LocalModelBlowUp}}
\end{figure}
\end{center}

\begin{prop}
\label{iviv}
Let $\phi$ be a GMN differential on $S$ with $s$ simple poles $p_i$. Let $r_i\in \bR_{>0} $ be  sufficiently small.  Then there is a uniquely-defined complete GMN differential with double poles at the $p_i$, centred on degenerate ring domains with parameters $-r_{i}$, and equivalent on the complement of the closures of those ring domains  to $(S\backslash \cup_i \gamma_i, \phi)$, where $\gamma_i$ is the unique horizontal trajectory of $\phi$ of length $r_{i}$ with one end-point at $p_i$. \end{prop}

\begin{proof}
If the $r_{i}$ are sufficiently small, the trajectory $\gamma_i$ is embedded in the surface and does not contain any critical points other than $p_i$. The existence part of the statement is then depicted in Figure \ref{Fig:LocalModelBlowUp}. The dashed rectangle has boundaries on the horizontal and vertical foliations for the relevant differentials; it lies in a co-ordinate chart in the central picture, and is mapped conformally in its interior to the two outer pictures, which define  surfaces with quadratic differentials which are equivalent in the component exterior to the rectangle's boundary arc. The existence of the differential on the surface on the right, obtained by gluing a cylinder and a surface with geodesic boundary containing a simple zero, is an application of Lemma \ref{Lem:GlueBorders}. 
\end{proof}

If the original differential has finite area, there are closed geodesics for a dense set of phases \cite[Theorem 25.2]{Strebel}; however, the only such geodesics which survive to the surgered surface are those which are disjoint from the length $r_i$ segments of the  trajectories emanating from the simple poles.   This is compatible with the fact that the spectrum of phases of  closed geodesics after surgery is closed in $S^1$. 

\begin{Remark}
\label{forw}
Let us use the notation $\Quad(g,(1^d))$ for the space of GMN differentials $(S,\phi)$ with $S$ of genus $g$, and $\phi$ having $d$ simple poles. Similarly let $\Quad_r(g,(2^d))$ denote the space of GMN differentials $(S,\phi)$ with $S$ of genus $g$ and $\phi$ having $d$ double poles, each of residue $\pm r$. The construction of Proposition \ref{iviv} gives an injective map
\[B\colon \Quad_r(g,(2^d)) \to \Quad(g,(1^d))\]
which moreover commutes with the locally-defined period maps on both sides. It is not hard to convince oneself that $B$ is in fact a local homeomorphism. This reduces the question of whether the period map is a local isomorphism to the case of differentials with at least one infinite critical point.
\end{Remark}

A closely related model is obtained by opening up a length $l$ segment of an irrational foliation on a torus $S = T^2$ to obtain a recurrent surface with one boundary component, the boundary made of two equal length saddle trajectories; one can then glue in a degenerate ring domain centred on a double pole as above to obtain another foliation on a closed torus.   The closed geodesics in the new (infinite area) surface correspond to the $(p,q)$-curves on $T^2$ which are disjoint from the original straight arc of length $l$. It is easy to check that only finitely many $(p,q)$-curves have this property, hence the surgery collapses the spectrum of closed trajectories from a dense subset of the circle to a finite subset.


\subsection{Extended period map}
\label{epm}

Let $(\S,\M)$ be a marked bordered surface,  and let \[\Quad(g,m)=\Quad(\S,\M)_0\subset \Quad(\S,\M)\]
 be the corresponding open stratum of complete differentials. Fix a free abelian group $\Gamma$ of rank $n$ given by \eqref{n}.
By an \emph{extended  framing} of a point $\phi\in \Quadorb(\S,\M)$ we mean an isomorphism of groups 
\[\theta\colon \Gamma\to  \hse.\]
Defining the space $\Quad^\Gamma(\S,\M)$ of extended framed differentials in the obvious way,  we obtain an unbranched  cover  \[\Quad^\Gamma(\S,\M)\to \Quadorb(\S,\M).\]
Over the locus $\Quad(\S,\M)_0$ of complete differentials, the resulting space coincides with the space $\Quad^\Gamma(g,m)$ considered before.

Lemma \ref{fatman} shows that the generic automorphism group of the orbifold $\Quad(\S,\M)$  is trivial except when $(\S,\M)$ is one of  
\begin{itemize}
\item[(i)] an unpunctured disc with 3 or 4 points on its boundary;\smallskip
\item[(ii)]  an annulus with one marked point on each boundary component;\smallskip
\item[(iii)] a closed torus with a single puncture;
\end{itemize}
corresponding to polar types $(5)$, $(6)$ and $(3,3)$ in genus $g=0$, and polar type $(2)$ in genus $g=1$. As explained before, in all these cases the orbifold $\Quad(\S,\M)$ also has a non-trivial generic automorphism group.

\begin{prop}
\label{reff}
Assume that $(\S,\M)$ is not one of the 4 exceptional surfaces listed above. Then the space $\Quad^\Gamma(\S,\M)$ is a complex manifold. The period map extends to a local isomorphism of complex manifolds
\[\pi\colon \Quad^\Gamma(\S,\M)\to \Hom_\Z(\Gamma,\C).\]
\end{prop}

\begin{pf}
Assume first that $(\S,\M)$ is not a sphere with 3 or 4 punctures.
Suppose that a point of $\Quad^\Gamma(\S,\M)$ has a non-trivial automorphism. This means that the underlying differential $\phi$ has a non-trivial automorphism which acts trivially on the extended hat-homology group $\hse$. It follows from Lemma \ref{rich} that this automorphism also acts trivially on the hat-homology group $\hs$. But we proved in Theorem \ref{perper} that  no such automorphisms exist. Thus $\Quad^\Gamma(\S,\M)$ is a manifold.

The extended period map $\pi$ is defined in the obvious way: the period of a differential $\phi$ defines a map $Z_\phi\colon \hs\to \C$ which  induces a group homomorphism $Z_\phi\colon \hse\to \C$ by composing with the map $q$ of Lemma \ref{rich}. To show that $\pi$ is a local isomorphism, suppose that a nonzero tangent vector $v$ to $\Quad^\Gamma(\S,\M)$ at some point $\phi$ lies in the kernel of the derivative of $\pi$. Then, since the strata of $\Quad(\S,\M)$ are determined by the vanishing of the periods $Z_\phi(\beta_p)$,  it follows that $v$ is tangent to the stratum containing $\phi$. But the period map is a local isomorphism on each stratum by Theorem \ref{perper}, so this gives a contradiction.

In the case when $(\S,\M)$ is a 3 or 4 punctured sphere, the above proof is incomplete, because  the  two polar types $(1,1,2)$ and $(1,1,1,1)$ in genus $g=0$ were excluded from Theorem \ref{perper}, since the corresponding spaces $\Quad(g,m)$ have a non-trivial generic automorphism group. In both of these cases  the generic automorphisms identified in Example \ref{spec} act non-trivially on the extended hat-homology group, since they permute the simple poles, and hence the corresponding residue classes $\beta_p$. The fact that the period map is a local isomorphism on the corresponding strata of  $\Quad(\S,\M)$ can be  proved exactly as in Theorem \ref{perper}, or just checked directly. 
\end{pf}



\subsection{Degenerations}
\label{deg}

We finish this first part of the paper with two technical results which will be used later in the proofs of our main Theorems. The first one 
 will allow us to extend our correspondence between differentials and stability conditions over the incomplete locus in $\Quad(\S,\M)$.

\begin{prop}
\label{leaving}
Take a framed  differential  $\phi_0\in  \Quad^\Gamma(\S,\M)$. Then for any $\epsilon>0$ there is a neighbourhood $\phi_0\in U\subset \Quad^\Gamma(\S,\M)$  such that for any differential $\phi\in U$, and any  class $\gamma\in \Gamma$ represented by a non-closed saddle connection in $\phi$, there is an inequality
\[|Z_{\phi}(\gamma)-Z_{\phi_0}(\gamma)|<\epsilon\,| Z_{\phi_0}(\gamma)|.\]
\end{prop}

\begin{pf}
We can assume that all differentials $\phi\in U$ are on a fixed underlying smooth surface $\S$, with finite critical points at fixed points $x_i\in \S$. However we must allow the double poles of $\phi$ to move, so that they can collide with the zeroes. We can assume that if $\phi$ has a simple pole at $x_i$ then so does $\phi_0$. 

Consider the subset of $ U\times \S$ consisting of pairs $(\phi,y)$ with $y\in \S$ lying on a non-closed saddle connection for $\phi$, and let $F$ be its closure. Then $F$ contains no points of the form $(\phi,p_j)$ with $p_j$ an infinite critical point  of $\phi$, because any such point is contained in a trapping neighbourhood containing no non-closed  saddle connections. Thus, shrinking $U$ if necessary, we have a bound
\[|\sqrt{\phi}-\sqrt{\phi_0}\,| <\epsilon\, |{\sqrt{\phi_0}}|\]
for all points of $F$.
Integrating this along a non-closed saddle connection  for $\phi$ gives the result.
\end{pf}

The next result is  a kind of completeness result for the space $\Quad(\S,\M)$. It will be used later to prove that the image of the   map we construct from differentials to stability conditions is closed.
We say that a saddle connection $\gamma$  for a GMN differential $\phi$ is  \emph{degenerate}  if it is closed, and is moreover freely homotopic in $S^\op$ to a small loop around a double pole of $\phi$.\footnote{ Note that a saddle connection $\gamma$ of phase $\theta$ is nothing but a saddle trajectory for the rotated differential $e^{-i\pi\theta}\cdot \phi$, and that $\gamma$ is degenerate precisely if this saddle trajectory is the boundary of a degenerate ring domain.}

 \begin{prop}
\label{mum}
Consider a sequence of framed, complete  differentials \[\phi_n\in \Quad^\Gamma(\S,\M)_0, \quad n\geq 1,\]   whose periods $Z_{\phi_n}\colon \Gamma \to \C$ converge. Suppose moreover that there is a universal constant $L>0$ such that any  non-degenerate saddle connection  for $\phi_n$ has length $\geq L$.
 Then some subsequence of the points $\phi_n$ converges to a limit in $\Quad^\Gamma(\S,\M)$.
\end{prop}

\begin{pf}
 
 Using the fact that  stable curves are  Gorenstein it is easy to see  that the vector bundle  \eqref{paul}  extends to a bundle
  \[\bar{\mathcal{H}}(g,m)\lra \bar{\Mod}(g,d)/\Sym(m)\]  
 over the Deligne-Mumford compactification.
 The projectivisation of this bundle   is a compact space, and so, passing to a subsequence, and potentially rescaling the $\phi_n$, we can assume that the differentials $\phi_n$  have a limit $(S,\phi)\in \bar{\mathcal{H}}(g,m)$. The hypothesis that the  periods $Z_{\phi_n}$ converge then implies that the rescaling must have been unneccesary.

Note that our assumption implies that if $\gamma$ is a path in the surface $S_n$ which  either connects two finite critical points of $\phi_n$,  or is closed and not homotopic to a small loop around a double pole, then the length of $\gamma$ in the $\phi_n$-metric is at least $L$. Indeed, by Theorem \ref{finini}, the curve $\gamma$ has a minimal geodesic  in its homotopy class and this is a union of saddle connections.

Suppose that the limit curve $S$ has a node $p$, and that the limit section $\phi$ is non-vanishing at $p$. Note that  the induced quadratic differential on the normalization has a double pole at the inverse image of $p$. Consider a curve  connecting two zeroes on $S_n$ lying on opposite sides of the neck which shrinks to the node $p$. Then  as $n\to \infty$ the period of the corresponding hat-homology class diverges to infinity, which contradicts the fact that the periods $Z_{\phi_n}$ converge.

Suppose instead that $S$ is a  stable curve with a node $p$, and that the section $\phi$ vanishes at $p$. Then consider a closed curve $\gamma$ on $S_n$ encircling the neck, homotopic to the vanishing cycle. Either $\gamma$ is non-separating, or there is more than one marked point on each side of $\gamma$, so   $\gamma$ cannot be homotopic to a small loop around a double pole. Then as $n\to \infty$ the length in the $\phi_n$-metric of $\gamma$ tends to zero, which again gives a contradiction. 

Thus we conclude that the limit curve $S$ is non-singular. 
 Suppose that the limit differential $\phi$ is defined by a section of $\omega_S^{\otimes 2} (\sum_i m_ip_i)$ which has a zero of some order $k\geq 1$ at a point $p\in S$. This means that $k$ simple zeroes of the $\phi_n$ have collided in the limit. If $p=p_i$ is a marked point then we set $m=m_i$, and otherwise we set $m=0$. 
 
Suppose first that $k-m\leq-2$ so that $p$ is  an infinite critical point of $\phi$. Take a path $\gamma_n$ which connects two  zeroes $z_1,z_2$  of $\phi_n$, both of which tend to $p$. Then the period of the corresponding hat-homology class tends to infinity, contradicting the assumptions. If $k=1$ so that there is only one zero $z$ which tends to $p$,  then we must have $m\geq 3$, and we can take  $\gamma_n$ to be a small loop around $p$ based at $z$.

On the other hand, if $k-m>-2$ then $p$ is a finite critical point of $\phi$. If also $k\geq 2$, then  take a path $\gamma_n$ connecting two zeroes of $\phi_n$ which both tend to $p$. The length of this path will tend to zero as $n\to \infty$ which gives a contradiction as before.
 We conclude that  $\phi$ has simple zeroes distinct from the points $p_i$ of order $m_i>2$. This is precisely the condition that $\phi\in \Quad(\S,\M)$.
\end{pf}






\section{Quivers and stability conditions}
\label{quivstab}

This section consists of fairly well-known background material on t-structures and tilting, quivers with potential and their mutations, and  stability conditions. 

\subsection{Introduction}
Let $\D$ be a $k$-linear triangulated category of finite type. We denote the shift functor by $[1]$ and use the notation
\[\Hom^i_\D(A,B):= \Hom_\D(A,B[i]).\]
The finite type condition is the statement that  for all objects $A,B\in \D$ 
\[\dim_k \bigoplus_{i\in\Z} \Hom_\D^i(A,B) <\infty.\]
The Grothendieck group $K(\D)$ then carries 
the Euler bilinear form 
\[\chi(\blank,\blank)\colon K(\D)\times K(\D)\to \Z\]
 defined by the formula
\[\chi(E,F)=\sum_{i\in\Z} (-1)^i \dim_k \Hom_\D^i(E,F).\]
Beginning in Section \ref{catsurf} we shall focus on the particular properties of the categories $\D=\D(\S,\M)$ appearing in our main Theorems, but 
the present section consists of general theory, and the only properties of $\D$  that will be important are
\begin{itemize}
\item[(i)]$\D$ admits a bounded t-structure whose heart $\A\subset \D$ is of finite length and has a finite number $n$ of simple objects up to isomorphism; \smallskip

\item[(ii)] $\D$ is a \CY category, meaning that there are functorial isomorphisms
\[\Hom^i_\D(A,B)\isom \Hom_\D^{3-i}(B,A)^* \text{ for all objects } A,B\in \D.\]
\end{itemize}

Note that (i) implies that $K(\D)\isom\Z^{\oplus n}$ is free of finite rank, and (ii) implies that the Euler form is skew-symmetric.

The main point of this section is to expand on the following two statements:

\begin{itemize}
\item[(a)] Associated to any  triangulated category $\D$   there is a complex manifold $\Stab(\D)$ of dimension $n$ parameterizing certain structures on $\D$ known as stability conditions. When $\D$ satisfies condition (i), a large open  subset of $\Stab(\D)$ can be described as a union of cells, one for each  bounded t-structure with finite-length heart. The  way these cells are glued together along their boundaries  is controlled by an abstract operation called tilting.
\\
\item[(b)]
A large class of triangulated categories $\D$ satisfying both  conditions (i) and (ii) can be defined using   quivers with potential  via the Ginzburg algebra construction. The abstract tilting operation referred to in (a) can then be described concretely in terms of  mutations of quivers with potential.
\end{itemize}


\subsection{Hearts and tilting}
\label{out}

Let $\D$ be a triangulated category. We shall be concerned with bounded t-structures on $\D$. 
Any such t-structure  is determined by its heart $\A\subset \D$, which is a full abelian subcategory. We  use the term  \emph{heart} to mean the heart of a bounded t-structure.
A heart  will be called \emph{finite-length} if it is   artinian and noetherian as an abelian category.

We say that a pair of  hearts $(\A_1,\A_2)$ in $\D$ is a \emph{tilting pair}  if the equivalent conditions
\[\A_2\subset \langle \A_1,\A_1[-1] \rangle, \quad \A_1\subset \langle\A_2[1],\A_2\rangle\]
are satisfied.\footnote{The angular brackets here signify the \emph{extension-closure} operation: given full subcategories $\A,\B\subset \D$, the extension-closure $\CC=\langle \A,\B\rangle \subset \D$ is the smallest full subcategory of $\D$ containing both $\A$ and $\B$, and such that if $X\to Y\to Z\to X[1]$ is a triangle in $\D$ with $X,Z\in \CC$ then $Y\in \CC$ also.} We also say that $\A_1$ is a \emph{left tilt} of $\A_2$, and that $\A_2$ is a \emph{right tilt} of $\A_1$. Note that $(\A_1,\A_2)$ is  a tilting pair precisely if  so is $(\A_2[1],\A_1)$.

If $(\A_1,\A_2)$ is a tilting pair in $\D$, then  the subcategories \[\TT=\A_1\cap \A_2[1], \quad  \FF=\A_1\cap \A_2\] form a torsion pair $(\TT,\FF)\subset \A_1$. Conversely, if $(\TT,\FF)\subset \A_1$ is a torsion pair, then the subcategory $\A_2=\langle \FF, \TT[-1]\rangle$ is a heart, and the pair $(\A_1,\A_2)$ is a tilting pair.

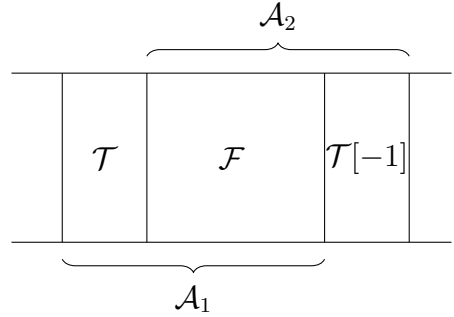
\begin{figure}
\begin{center}
\begin{tikzpicture}[scale=0.45]
\draw (0,4)--(13,4);
\draw (0,9)--(13,9);
\draw (1.5,4)--(1.5,9);
\draw (4,4)--(4,9);
\draw (9.25,4)--(9.25,9);
\draw (11.75,4)--(11.75,9);
\draw (2.75,6.5) node {$\TT$};
\draw (6.5,6.5) node {$\FF$};
\draw (10.5,6.5) node {$\TT[-1]$};
\draw [decorate, decoration={brace,amplitude=5pt}] (4,9.5)--(11.75,9.5)
node [midway, above=6pt] {$\A_2$};
\draw [decorate, decoration={brace,amplitude=5pt}] (9.25,3.5)--(1.5, 3.5)
node [midway, below=6pt] {$\A_1$};
\end{tikzpicture}
\caption{A tilting pair.}
\end{center}
\end{figure}

A special case of the tilting construction will be particularly important.
Suppose that $\A$ is a finite-length heart and $S\in\A$ is a simple object. 
 Let
$\langle S \rangle\subset \A$ be the full subcategory consisting of objects $E\in\A$ all
of whose simple factors are isomorphic to $S$. Define full subcategories
\[ S^\perp=\{E\in \A:\Hom_{\A}(S,E)=0\}, \quad \smash{^\perp}{S}=\{E\in\A:\Hom_{\A}(E,S)=0\}.\]
One can either view $\langle S \rangle$ as
the torsion part of a torsion pair on $\A$, in which case the torsion-free
part is $S^\perp$,
or as the torsion-free part, in which case the torsion part is $\smash{^\perp}{S}$.
We can then define tilted
 hearts
\[ \mu^-_S (\A)=\langle S[1],\smash{^\perp}{S}\rangle,\qquad \mu^+_S (\A) = \langle S^\perp, S[-1]\rangle,\]
which we refer to as the left and right tilts of the heart $\A$ at the simple $S$.
They fit into tilting pairs $(\mu^-_S (\A),\A)$ and $(\A,\mu^+_S (\A))$. Note the relation
\begin{equation}
\label{jen}\mu^+_{S[1]}\circ \mu^-_S(\A)=\A.\end{equation}

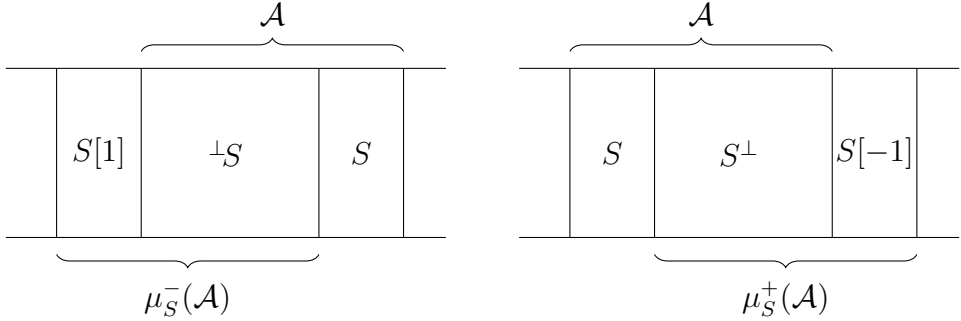
\begin{figure}
\begin{center}
\begin{tikzpicture}[scale=0.45]
\draw (0,4)--(13,4);
\draw (0,9)--(13,9);
\draw (1.5,4)--(1.5,9);
\draw (4,4)--(4,9);
\draw (9.25,4)--(9.25,9);
\draw (11.75,4)--(11.75,9);
\draw (2.75,6.5) node {$S[1]$};
\draw (6.5,6.5) node {$^\perp\!S$};
\draw (10.5,6.5) node {$S$};
\draw [decorate, decoration={brace,amplitude=5pt}] (4,9.5)--(11.75,9.5)
node [midway, above=6pt] {$\A$};
\draw [decorate, decoration={brace,amplitude=5pt}] (9.25,3.5)--(1.5, 3.5)
node [midway, below=6pt] {$\mu_S^-(\A)$};
\end{tikzpicture}
\quad\quad
\begin{tikzpicture}[scale=0.45]
\draw (0,4)--(13,4);
\draw (0,9)--(13,9);
\draw (1.5,4)--(1.5,9);
\draw (4,4)--(4,9);
\draw (9.25,4)--(9.25,9);
\draw (11.75,4)--(11.75,9);
\draw (2.75,6.5) node {$S$};
\draw (6.5,6.5) node {$S^\perp$};
\draw (10.5,6.5) node {$S[-1]$};
\draw [decorate, decoration={brace,amplitude=5pt}] (1.5,9.5)--(9.25,9.5)
node [midway, above=6pt] {$\A$};
\draw [decorate, decoration={brace,amplitude=5pt}] (11.75,3.5)--(4, 3.5)
node [midway, below=6pt] {$\mu_{S}^+(\A)$};
\end{tikzpicture}
\caption{Left and right tilts of a heart.}\label{fig:mut-hearts}
\end{center}
\end{figure}

The \emph{tilting graph} of $\D$ is the graph $\Tilt(\D)$ whose vertices are finite-length hearts, and  in which two vertices are joined by an edge if the corresponding hearts are related by a tilt in a simple object. There is a natural action of the group of triangulated autoequivalences $\Aut(\D)$ on this graph.

If $\A\subset \D$ is a finite-length heart we denote by $\Tilt_\A(\D)\subset \Tilt(\D)$ the connected component containing $\A$. We say that the hearts in $\Tilt_\A(\D)$ are  \emph{reachable} from  $\A$. We say that an autoequivalence $\Phi\in \Aut(\D)$ is \emph{reachable} from $\A$ if its action on $\Tilt(\D)$  preserves the connected component $\Tilt_\A(\D)$. These autoequivalences form a subgroup $\Aut_\A(\D)\subset \Aut(\D)$.

We say that  a finite-length heart $\A\subset \D$ is  \emph{infinitely tiltable} if the graph $\Tilt_\A(\D)$ is  $2n$-regular,  where $n$ is the rank of $K(\D)$. This means that the tilting process can  be continued indefinitely at all simple objects, and in both directions, without leaving the class of finite-length hearts.


\subsection{Tilting in the \CY case}
\label{tiltcy}

Suppose now that $\D$ is a  triangulated category with the \CY property.  To ensure the existence of the twist functors appearing below we should also assume that $\D$ is algebraic in the sense of Keller \cite[Section 3.6]{dg}.

Associated to a finite-length heart $\A\subset \D$ there is a quiver  $Q(\A)$, 
 whose vertices are indexed by the isomorphism classes of simple objects $S_i\in \A$ and which has
\[n_{ij}=\dim_k \Ext^1_\A(S_i,S_j)\]
 arrows connecting vertex $i$ to vertex $j$.
We call  a finite-length  heart $\A\subset \D$  \emph{non-degenerate} if it is infinitely-tiltable and  if, for every heart $\B\subset \D$ reachable from $\A$, the quiver $Q(\B)$ has no loops or oriented 2-cycles.

The quiver $Q(\A)$ associated to a finite-length heart $\A\subset \D$ has no loops precisely if 
the simple objects of $\A$ are all spherical  in the sense of \cite{ST}. Any spherical object  $S\in \D$ defines an  autoequivalence  $\Tw_{S}\in \Aut(\D)$ called a spherical twist. It has the property that  for any  object $E\in \D$ there is a triangle
\[\Hom^\blob_\D(S,E)\tensor S\lra E\lra \Tw_S(E).\]

Suppose now that $\A\subset \D$ is a non-degenerate finite-length heart with $n$ simple objects up to isomorphism.  
 Taken together, the spherical twists in these simple objects   generate a subgroup
\[\Sph_\A(\D)=\langle \Tw_{S_1},\dots ,\Tw_{S_n}\rangle\subset \Aut(\D).\]
The following result is well-known, but for the reader's convenience we include a sketch proof. A more careful treatment  can be found for example in \cite{yuqiu}.
\begin{prop}
\label{rpt}

\begin{itemize}
\item[(a)] for every simple object $S\in \A$ there is a relation \[\Tw_{S} (\mu^-_{S}(\A))=\mu^+_{S}(\A)\subset \D;\]
\item[(b)] if $\B$ is reachable from $\A$, then $\Sph_\A(\D)=\Sph_\B(\D)$.
\end{itemize}
\end{prop}

\begin{pf}Take a simple object $S\in \A$.  Since $Q(\A)$ has no oriented 2-cycles, we can order the simple objects of $\A$ so that $S=S_i$, and  $\Ext^1(S_j,S_i)=0$ for $j<i$ and $\Ext^1(S_i,S_j)=0$ for $j>i$. Then for all  $j<i$, the object $\Tw_{S_i}(S_j)$ is the universal extension
\[0\lra S_j\lra \Tw_{S_i}(S_j)\lra \Ext^1_\A(S_i,S_j)\tensor S_i\lra 0.\]
This clearly lies in  $S_i^{\perp}\subset \mu^+_{S_i}(\A)$, and is easily checked to be simple. In this way one sees that the simple objects of $\mu^+_{S_i}(\A)$ are
\[(\Tw_{S_i}(S_1),\dots, \Tw_{S_i}(S_{i-1}),S_i[-1], S_{i+1}, \dots ,S_n).\]
By a similar argument, or using $\mu^-_{S_i[-1]} \,\mu_{S_i}^+(\A)=\A$, it follows that the simple objects in $\mu^-_{S_i}(\A)$ are
\[(S_1,\dots, S_{i-1},S_i[1], \Tw_{S_i}^{-1}(S_{i+1}),\dots, \Tw_{S_i}^{-1}(S_{n})).\]
Property (a) is  then clear, using the identity $\Tw_{S_i}(S_i)=S_i[-2]$, and the fact that a finite-length heart is determined by its simple objects.
Property (b) follows from the  identity
\[\Tw_{\Tw_{S_i}(S_j)}=\Tw_{S_i}\circ \Tw_{S_j}\circ \Tw^{-1}_{S_i},\]
and the fact that $\Tw_{S_i[1]}=\Tw_{S_i}$.
\end{pf}

%

\subsection{Quivers with potential}
\label{ginz}

Suppose  that $\D$ is a triangulated category with the \CY property.
In Section \ref{tiltcy} we associated a quiver $Q(\A)$  to a finite-length heart $\A\subset \D$ encoding the dimensions of the extension spaces between the simple objects.  The next result shows that after   making the choice of  a potential on $Q(\A)$ one can reverse this process.

 For all notions regarding quivers with potential we refer to \cite[\SS 2--5]{DWZ} and \cite[\SS 2]{KY}. In particular, we recall that a \emph{potential} on a quiver $Q$ is a formal linear combination of oriented cycles in $Q$, and that a potential is called \emph{reduced} if it is a sum of cycles of length $\geq 3$.

\begin{thm}
\label{eyes}
Associated to a quiver with reduced potential $(Q,W)$ there is a \CY triangulated category $\D(Q,W)$ of finite type over $k$, with a  bounded  t-structure whose heart  \[\A=\A(Q,W)\subset \D(Q,W)\] is of finite-length and has associated quiver $Q(\A)$ isomorphic to $Q$.
\end{thm}

\begin{pf}
Define the category $\D(Q,W)$ to be the subcategory of the derived category of the complete Ginzburg algebra $\Pi(Q,W)$ consisting of objects with finite-dimensional cohomology. It has the \CY property by \cite[Lemma 7.16, Theorem 7.17]{KY}.

The category $\D(Q,W)$ has a standard bounded  t-structure \cite[Lemma 5.2]{KY} 
whose heart $\A(Q,W)$ is equivalent to the category of finite-dimensional modules for the complete Jacobi algebra $J(Q,W)=H^0(\Pi(Q,W))$. The algebra $J(Q,W)$  is the quotient of the complete path algebra of $Q$ by the relations obtained by cyclically differentiating the potential $W$. In particular, its simple modules are naturally in bijection with the vertices of $Q$, and the spaces of extensions between them are based by the arrows in $Q$.
\end{pf}

The combinatorial incarnation of the process of tilting at a simple module is called \emph{mutation}. It acts on right-equivalence classes of quivers with potential. Roughly speaking, two potentials on a quiver $Q$ are said to be \emph{right-equivalent} if they differ by  an automorphism of the completed path algebra which fixes the zero length paths; for the full definition see \cite[\SS 4]{DWZ} or \cite[\SS 2.1]{KY}.   Right-equivalent potentials give rise to isomorphic complete Ginzburg algebras \cite[Lemma 2.9]{KY}, and hence equivalent  categories $\D(Q,W)$.

  Suppose that $(Q,W)$ is a reduced quiver with potential, and fix a vertex $i$ of $Q$. The mutation $(Q',W')=\mu_i(Q,W)$ is another reduced quiver with potential, well-defined up to right-equivalence, and depending only on the right-equivalence class of $(Q,W)$. The vertex sets of $Q$ and $Q'$ are naturally identified, and the operation $\mu_i$ is an involution. We refer the reader to \cite[\SS 5]{DWZ} or \cite[\SS 2.4]{KY} for the relevant definitions.

For our purposes, the importance of mutations of quivers with potential is the following result of Keller and Yang \cite[Thm. 3.2, Cor. 5.5]{KY}. 

\begin{thm}
\label{keller_yang}
Let $(Q,W)$ be a quiver with reduced potential, such that $Q$ has no loops or oriented 2-cycles. Let $i$ be a vertex of $Q$ and  set \[(Q',W')=\mu_i(Q,W).\]
Then there is a canonical pair of $k$-linear triangulated equivalences 
\[\Phi_{\pm}\colon \D(Q',W')\lra \D(Q,W)\]
which induce tilts in the simple object $S_i\in \A(Q,W)$ in the sense that  \[\Phi_\pm(\A(Q',W'))=\mu^{\pm}_{S_i}(\A(Q,W))\subset \D(Q,W),\]
and which moreover induce the natural bijections on simple objects.
\end{thm}

For the last part of this statement, recall that there is a natural bijection between the simple objects of $\A$ and those of $\mu^\pm_{S_i}(\A)$ which was made explicit in the proof of Proposition \ref{rpt}. There is also a natural bijection between the vertices of the quivers $(Q,W)$ and $\mu_i(Q,W)$. The claim is that these bijections are compatible with the canonical bijections between the vertices of the quivers $Q,Q'$  and the simple objects in the corresponding standard hearts.

A quiver with potential $(Q,W)$ is called \emph{non-degenerate} \cite[\SS 7]{DWZ} if any sequence of mutations of $(Q,W)$ results in a quiver with potential having no loops or oriented 2-cycles. Theorem \ref{keller_yang} shows that this condition  is equivalent to the statement that the standard heart $\A(Q,W)\subset \D(Q,W)$ is non-degenerate in the sense of Section \ref{tiltcy}.


\subsection{Stability conditions}
\label{stabsumm}
Here we summarize the basic properties of spaces of stability conditions. We refer the reader to \cite{Bridgeland,Bridgeland2} for more details on this material. Let us fix a  triangulated category $\D$, and assume for simplicity  that the Grothendieck group $K(\D)\isom \Z^{\oplus n}$ is free of  finite rank.

 A \emph{stability condition} $\sigma=(Z,\P)$ on   $\D$
consists of
a group homomorphism
$Z\colon K(\D)\to\C$ called the \emph{central charge},
and full additive
subcategories $\P(\phi)\subset\D$ for each $\phi\in\R$,
which together satisfy the following axioms:
\begin{itemize}
\item[(a)] if $E\in \P(\phi)$ then $Z(E)\in \R_{>0}\cdot e^{i\pi\phi}\subset \C,$ \smallskip
\item[(b)] for all $\phi\in\R$, $\P(\phi+1)=\P(\phi)[1]$,\smallskip
\item[(c)] if $\phi_1>\phi_2$ and $A_j\in\P(\phi_j)$ then $\Hom_{\D}(A_1,A_2)=0$,\smallskip
\item[(d)] for each nonzero object $E\in\D$ there is a finite sequence of real
numbers
\[\phi_1>\phi_2> \dots >\phi_k\]
and a collection of triangles
\[
\xymatrix@C=.5em{
0_{\ } \ar@{=}[r] & E_0 \ar[rrrr] &&&& E_1 \ar[rrrr] \ar[dll] &&&& E_2
\ar[rr] \ar[dll] && \ldots \ar[rr] && E_{k-1}
\ar[rrrr] &&&& E_k \ar[dll] \ar@{=}[r] &  E_{\ } \\
&&& A_1 \ar@{-->}[ull] &&&& A_2 \ar@{-->}[ull] &&&&&&&& A_k \ar@{-->}[ull] 
}
\]
with $A_j\in\P(\phi_j)$ for all $j$.
\end{itemize}

The semistable objects $A_j$ appearing in the filtration of axiom (d) are unique up to isomorphism, and are called the \emph{semistable factors} of $E$. We set
\[\phi^+(E)=\phi_1, \quad \phi^-(E)=\phi_k, \quad m(E)=\sum_{i} |Z(A_i)|\in \R_{>0}.\]
The real number $m(E)$ is called the \emph{mass} of $E$.
It follows from the definition that the subcategories $\P(\phi)\subset \D$ are  abelian categories; the objects of $\P(\phi)$ are said to be \emph{semistable of phase $\phi$}, and the simple objects of $\P(\phi)$ are said to be \emph{stable of phase $\phi$}.  For any interval $I\subset \R$ there is a full subcategory $\P(I)\subset \D$ consisting of objects whose semistable factors have phases in $I$. 

We shall always assume that our stability conditions $\sigma=(Z,\P)$ satisfy the \emph{support property} of \cite{KS}, namely that for some norm $\|\cdot\|$ on $K(\D)\tensor\R$ there is a constant $C>0$ such that
\begin{equation}
\label{support}\|\gamma\|< C\cdot |Z(\gamma)|\end{equation}
for all classes $\gamma\in K(\D)$  represented by $\sigma$-semistable objects in $\D$.
  As explained in \cite[Prop. B.4]{BM}, this is equivalent to assuming that they are \emph{full} and \emph{locally-finite} in the terminology of \cite{Bridgeland, Bridgeland3}.   We let  $\Stab(\D)$ denote the set of all such stability conditions on $\D$.

There is a natural topology on $\Stab(\D)$ induced by the metric
 \begin{equation}
 \label{dist}
d(\sigma_1,\sigma_2) = \sup_{0\neq E\in\D}\left\{
  |\phi_{\sigma_2}^- (E) - \phi_{\sigma_1}^-(E)|, \  |\phi_{\sigma_2}^+
  (E) - \phi_{\sigma_1}^+(E)|, \ \left|\log
    \frac{m_{\sigma_2}(E)}{m_{\sigma_1}(E)}\right|  
\right\}\in[0,\infty].
\end{equation}
The following result is proved in \cite{Bridgeland}.

\begin{thm}
\label{basic}
The space $\Stab(\D)$ has the structure of a complex manifold, such that the forgetful map
\[\pi\colon \Stab(\D)\lra \Hom_{\Z}(K(\D),\C)\]
taking a stability condition to its central charge,
is a local isomorphism.
\end{thm}

There are two commuting group actions on $\Stab(\D)$ that will be important later. The group of triangulated autoequivalences $\Aut(\D)$ acts on $\Stab(\D)$ in a rather obvious way: an autoequivalence $\Phi\in\Aut(\D)$ acts by
\[\Phi\colon (Z,\P)\mapsto (Z',\P'), \quad Z'(E)=Z(\Phi^{-1}(E)), \quad \P'=\Phi(\P).\] 
There is also an action of  the universal cover of the group  $\operatorname{GL^+}(2,\R)$ of orientation-preserving linear automorphisms of $\R^2$. This  action does not change the subcategory $\P$, but acts by post-composition on the central charge, viewed as a map to $\C=\R^2$, with a corresponding adjustment of the grading on $\P$. This action is not free in general, but there is a  subgroup isomorphic to $\C$ which does act freely: an element $t\in \C$ acts by
\[t\colon (Z,P)\mapsto (Z',\P'), \quad Z'(E)=e^{-i\pi t}\cdot Z(E), \quad \P'(\phi)=\P(\phi+\operatorname{Re}(t)).\]
Note that for any integer $n\in \Z$, the action of the multiple shift functor $[n]$ coincides with the action of $n\in \C$.

Later we will need the following more precise version of Theorem \ref{basic}.

\begin{prop}
\label{lastnew}
 Fix a real number $0<\epsilon\ll 1$. Given a stability condition  $\sigma=(Z,\P)\in \Stab(\D)$, and a group homomorphism $W\colon K(\D)\to \C$ satisfying
\[|W(E)-Z(E)|< \epsilon \cdot |Z(E)|\]
for all $\sigma$-stable objects $E\in \D$,  there is a unique stability condition $\sigma'\in \Stab(\D)$ with central charge $W$ such that $d(\sigma,\sigma')<\frac{1}{2}$.
\end{prop}

\begin{pf}
 In fact we can take any $0<\epsilon<\frac{1}{8}$. The support property implies that for any interval $I\subset \R$ of length $<1$, the quasi-abelian categories $\P(I)$ have finite-length. The existence part then follows  from the results of \cite[Section 7]{Bridgeland}. The uniqueness is a consequence of \cite[Lemma 6.4]{Bridgeland}.
\end{pf}


\subsection{Walls and chambers}
\label{wallchamber}
In this section we give some basic results on the wall-and-chamber decomposition of the space of stability conditions. These are well-known, but the proofs in the general setting are not available in the literature. As in the last section, we fix a triangulated category $\D$ and assume that $K(\D)\isom \Z^{\oplus n}$ is free of finite rank. 

\begin{prop}
\label{fe}
Fix an object $E\in \D$. Then
\begin{itemize}
\item[(a)] the set of $\sigma\in \Stab(\D)$ for which $E$ is $\sigma$-stable is open. 
\item[(b)]the set of $\sigma\in \Stab(\D)$ for which $E$ is $\sigma$-semistable is closed.
\end{itemize}
\end{prop}

\begin{pf}

For part (a), take a stability condition $\sigma=(Z,\P)$ and an object $E\in \P(\phi)$ which is $\sigma$-stable. Choose $0<r\ll 1$ and consider  the open ball $B_r(\sigma)$  of radius $r$ centered at $\sigma$, with respect to the metric  \eqref{dist}.  By definition of this metric, for $\sigma'=(Z',\P')\in B_r(\sigma)$ there are inclusions
\[\P(\phi)\subset \P'(\phi-r,\phi+r)\subset \P(\phi-2r,\phi+2r).\]
Thus $E$ fails to be stable in $\sigma'$ precisely if there is a triangle $A\to E\to B$ whose objects all lie in $\P(\phi-2r,\phi+2r)$ and for which $\phi'(A)\geq \phi'(E)$.

The support property implies that  for any $M>0$ there are only finitely  many classes $\alpha\in K(\D)$ satisfying $|Z(\alpha)|<M$ for which there exist  $\sigma$-semistable objects of class $\alpha$. It follows that the set of classes $\alpha\in K(\D)$ of  objects $A$ as above is finite. Since $E$ is stable in $\sigma$ we must have $\phi(A)<\phi(E)$  for each such subobject, and so, reducing $r$ if necessary, we can assume that these phase inequalities continue to hold for all $\sigma'\in B_r(\sigma)$.  It follows that $E$ is stable for all stability conditions in $ B_r(\sigma)$.

Part (b) is immediate:
the object $E$ is semistable in $\sigma$ precisely if   $\phi_\sigma^+(E)=\phi_\sigma^-(E)$. By the definition of the metric \eqref{dist} this is  a closed condition. 
  \end{pf}

Let us now fix a class $\gamma\in K(\D)$ and consider stability for objects of this class. Let $\alpha\in  K(\D)$ be another class  which is not proportional to $\gamma$.  We   define \[W_\gamma(\alpha)\subset \Stab(\D)\] to be the subset of stability conditions $\sigma=(Z,\P)$ satisfying the  following condition:
for some  $\phi\in \R$ there is an inclusion $A\subset E$ in the category $\P(\phi)$ such that $A$ and $E$ have classes $\alpha$ and $\gamma$  respectively.
Locally, the subset $W_\gamma(\alpha)$ is contained in the  real codimension one submanifold of $\Stab(\D)$ defined by the condition $ Z(\alpha)/Z(\gamma)\in \R_{>0}$.

\begin{lemma}
If $B\subset \Stab(\D)$ is a compact subset then the set of classes $\alpha$ for which the subset $W_\gamma(\alpha)$ intersects $B$ is finite.
\end{lemma}

\begin{pf}
 The support property for a fixed stability condition $\sigma$ implies that for any given $M>0$ there are only finitely many classes $\alpha\in K(\D)$  represented by objects of mass $<M$ in $\sigma$. On the other hand, the definition  of the metric \eqref{dist} shows that the masses of objects of $\D$ vary by a uniformly bounded amount in $B$, so the same is true if we allow  $\sigma$ to vary in $ B$. Using compactness again we can assume that $M$ is large enough that $|Z(\gamma)|<M$ for all points in $B$. But if $\sigma\in W_\gamma(\alpha)\cap B$ then there is an inclusion $A\subset E$ in some $\P(\phi)$, and it follows that $A$  has mass $<M$, and hence has one of  finitely-many classes.
 \end{pf} 
 Consider the  complement of the closures \[\CC_\gamma= \Stab(\D)\setminus \bigcup_{\alpha\not\sim \gamma} \bar{W}_\gamma(\alpha)\]
where the union is over classes $\alpha$ which are not proportional to $\gamma$.
This   is the complement of a locally-finite union of closed subsets, hence is open.

We refer to the subsets $W_\gamma(\alpha)$ as \emph{walls} for the class $\gamma$, and the connected components of $\CC_\gamma$  will be called  \emph{chambers}. The following result shows that the question of whether a given object $E\in \D$ of class $\gamma$ is stable or semistable has a constant answer for stability conditions in a fixed chamber.

\begin{prop}
\label{chamberprop}
Let $U\subset \CC_\gamma$ be a chamber. If an object $E$ of class $\gamma$ is (semi)stable for some stability condition $\sigma\in U$ then the same is true for all $\sigma\in U$.
\end{prop}

\begin{pf}
We say that an object $E\in \D$ is pseudostable in a  stability condition $\sigma$ if it is semistable and the classes in $K(\D)$ of its stable factors are all proportional. The set of points $\sigma\in \Stab(\D)$ for which a given object $E\in \D$ is pseudostable is open: indeed, by Lemma \ref{fe}(a), if $E$ is pseudostable for $\sigma$ then the stable factors of $E$ remain stable in some open neighbourhood of $\sigma$, and their phases remain equal since they have proportional classes. The set of points where $E$  is unstable is also open, by Lemma \ref{fe}(b), since it is the complement of the points where $E$ is semistable.

Suppose now that $E\in \D$ has class $\gamma$.  If $E$ is semistable but not pseudostable then $\sigma$ must lie on a wall $W_\gamma(\alpha)$.  Thus, the subset of points of the chamber $U$ for which $E$ is semistable is both open and closed. Since $U$ is connected, this subset must be either empty or the whole of $U$.

Assume now that $E$ is semistable for all $\sigma\in U$. As above this implies that $E$ is pseudostable at each $\sigma\in U$. The set of $\sigma\in U$ for which $E$ is stable is then open, by Lemma \ref{fe}(a), and its complement, the set of points for which  $E$ is strictly pseudostable is also open, by the argument given above. Hence, if $E$ is stable for some stability condition in $U$, then it is stable for all of them.  \end{pf}


\subsection{Stability conditions from t-structures}
\label{chambers}

Let $\D$ be a triangulated category. Any stability condition $\sigma=(Z,\P)$ on $\D$ has an associated  heart \[\A=\P((0,1])\subset \D.\] It is the  extension closure of the subcategories $\P(\phi)$ for $0<\phi\leq 1$.  All nonzero objects of $\A$ are mapped by $Z$ into the semi-closed upper half plane
\[\H=\{r\exp(i\pi\phi):r\in\R_{>0}\text{ and }0<\phi\leq 1\}\subset\C.\]
Conversely, given a heart $\A\subset \D$, and a group homomorphism $Z\colon K(\A)\to\C$ with this property, then providing some finiteness conditions are satisfied, there is a unique stability condition on $\D$ with  heart $\A$ and central charge $Z$.

In particular, if $\A\subset \D$ is a finite-length heart  with $n$ simple objects $S_i$ up to isomorphism, the
subset $\Stab(\A)\subset \Stab(\D)$ consisting of stability conditions with heart $\A$ is mapped bijectively by $\pi$ onto the subset
\[\{Z\in \Hom_{\Z}(K(\D),\C): Z(S_i)\in \H\},\]
 and is therefore homeomorphic to $\H^n$.

The following  result shows that tilting controls the way the subsets $\Stab(\A)$, for different hearts $\A\subset \D$, are glued together in $\Stab(\D)$.

\begin{lemma}
\label{wolf}
Let $\A\subset \D$ be a finite-length heart, and suppose that \[\sigma=(Z,\P)\in\Stab(\D)\] lies on a unique boundary component of the region $\Stab(\A)\subset \Stab(\D)$, so that  $\Im Z(S_i)=0$ for a unique simple object $S_i$. Assume that the tilted hearts $\mu^\pm_{S_i}(\A)$ are also finite-length.
Then there is a neighbourhood $\sigma\in U\subset \Stab(\D)$ such that one of the following holds
\begin{itemize}
\item[(i)] $Z(S_i)\in\R_{<0}$,  and $U\subset \Stab(\A)\sqcup \Stab(\mu^+_{S_i}(\A))$,

 \item[(ii)] $Z(S_i)\in\R_{>0}$,  and $U\subset \Stab(\A)\sqcup \Stab(\mu^-_{S_i}(\A))$.
\end{itemize}
\end{lemma}

\begin{pf}
This is stated without proof in \cite[Lemma 5.5]{Bridgeland2}. A special case is proved in \cite[Proposition 2.4]{Br3}, and the general case is proved in exactly the same way.
\end{pf}

The subsets $\Stab(\A)\subset \Stab(\D)$ form a different system of walls and chambers in $\Stab(\D)$. The walls consist of points where the subcategory $\P(0)$ contains nonzero objects. To distinguish them from the walls considered in the last subsection they are often referred to as walls of type II. Note that they do not depend on a choice of class $\alpha\in K(\D)$. 

Suppose that $\A\subset \D$ is a finite-length heart. It follows from Lemma \ref{wolf} that there is a single connected component $\Stab_\A(\D)\subset \Stab(\D)$ containing all stability conditions whose hearts lie in the connected component $\Tilt_\A(\D)\subset \Tilt(\D)$. Note however that it is not usually the case  that $\Stab_\A(\D)$ is the union of the chambers $\Stab(\B)$ for hearts $\B$ reachable from $\A$ (see \cite{woolf} for a detailed discussion of this point).

An autoequivalence of $\D$ lying in the subgroup $\Aut_\A(\D)\subset \Aut(\D)$ of autoequivalences  reachable from $\A$ necessarily preserves the connected component $\Stab_\A(\D)$. The converse is false:  the existence of the $\C$-action shows that the shift functor $[1]$  fixes all connected components of $\Stab(\D)$, but it is not generally true that $[1]$ is reachable.

It is easy to see that a triangulated autoequivalence $\Phi$  acts trivially on  $\Stab_\A(\D)$ precisely if it fixes a heart $\A\subset \D$ and furthermore fixes pointwise the isomorphism classes of its simple objects. This is equivalent to the condition that $\Phi$ acts trivially on the connected component $\Tilt_\A(\D)$. We say that such autoequivalences are \emph{negligible} with respect to the heart $\A$.





\section{Surfaces and triangulations}
\label{surf}
The particular examples of \CY categories considered in this paper  will be defined using quivers with potential associated to triangulations of marked bordered surfaces.  Unfortunately, the  non-degenerate ideal triangulations  appearing  in Section \ref{fir} will not be sufficient for our purposes. Indeed,  to understand the space of stability conditions on our categories, we will need to understand all  hearts that are reachable from the standard heart; whilst some of these hearts correspond to  non-degenerate ideal triangulations,  others correspond to more exotic objects introduced in \cite{FST} called tagged triangulations. 


\subsection{Ideal triangulations}

\label{tri}

Here we give a brief summary of the relevant definitions concerning  triangulations of marked bordered surfaces. The reader can find a more careful treatment in  the paper  of Fomin, Shapiro and Thurston \cite{FST}.

A \emph{marked bordered surface} is defined to be a pair $(\S,\M)$  consisting of a  compact, connected oriented surface with boundary, and a finite non-empty set $\M\subset \S$ of marked points such that each boundary component of $\S$ contains at least one marked point. Marked points in the interior of $\S$ are called \emph{punctures}; the set of punctures is denoted $\bP\subset \M$.

An \emph{arc}  in   $(\S,\M)$ is a smooth path $\gamma$ in  $\S$ connecting  points of $\M$, whose interior lies in the open subsurface $\S\setminus (\M\cup \partial \S)$, and which has no self-intersections in its interior.  We moreover insist that $\gamma$ should not be  homotopic, relative to its endpoints, to a single point, or to a path in $\partial \S$ whose interior contains no points of $\M$. Two arcs are considered to be equivalent if they are related by a homotopy through such arcs.

An \emph{ideal triangulation} of  $(\S,\M)$ is defined to be a maximal collection of equivalence classes of  arcs for which it is possible to find representatives whose interiors are pairwise disjoint. We refer to the arcs as the \emph{edges} of the triangulation.
An example of an ideal triangulation of a disc with 5 marked points on its boundary is depicted in Figure \ref{below}; note that it has just two edges.  To get something more closely approximating the intuitive notion of a triangulation of the surface one should add arcs in $\partial \S$ connecting the points of $\M$.

\begin{figure}[ht]
\begin{center}
\includegraphics[scale=0.25]{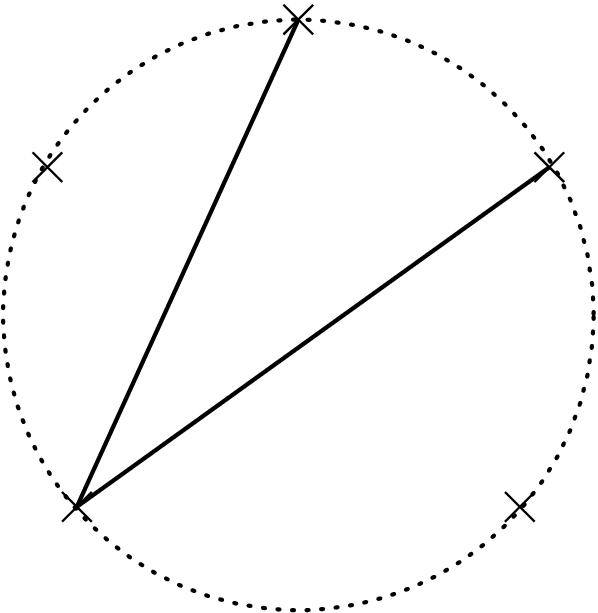}
\end{center}
\caption{A triangulation of a disc with 5 marked points.}\label{below}
\end{figure}

A \emph{face} or \emph{triangle} of an ideal triangulation $T$ is the closure in $\S$ of a connected component of the complement of all arcs of $T$.   A triangle  is called \emph{interior} if its intersection with $\partial \S$ is contained in $\M$.  Each  interior triangle  is topologically a disc,  containing either two or three distinct edges of the triangulation.

 \begin{figure}[ht] \begin{center}
\includegraphics[scale=0.3]{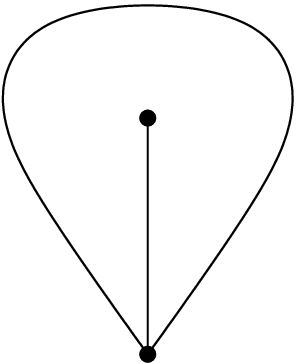}
\caption{A self-folded triangle.}
\end{center} \end{figure}
 
An interior triangle with just two distinct edges is called a \emph{self-folded triangle}; such a triangle  has a \emph{self-folded edge} and an \emph{encircling edge}.  The \emph{valency} of a  puncture $p\in \bP$ with respect to a  triangulation $T$ is the number of half-edges of $T$ that are incident with it; a puncture has valency 1 precisely if it is contained in a self-folded triangle.

   For various  technical reasons, when dealing with triangulations of marked bordered surfaces $(\S,\M)$, we shall always  make the following

\begin{assumption}
\label{asstwo}
We assume that $(\S,\M)$ is not one of the following surfaces
\begin{itemize}
\item[(a)] a sphere with $\leq 5$ marked points;
\item[(b)] an unpunctured disc with  $\leq 3$ marked points on the boundary;
\item[(c)] a disc with a single puncture and one marked point on its boundary.
\end{itemize}
\end{assumption}

In the cases of a sphere with $\leq 2$ punctures, or an unpunctured disc with $\leq 2$ marked points, there are no ideal triangulations, and so the theory described below is vacuous. In the cases of an unpunctured disc with 3 marked points, and the surface of case (c), there is a unique ideal triangulation and the theory is trivial and rather degenerate; see Examples \ref{beggar} and \ref{begg}.

The case of a three-punctured sphere is special in that there is an ideal triangulation consisting of two self-folded triangles meeting along a common edge; this plays havoc with the definition of a tagged triangulation below and for this reason it is better to deal with this case directly:  see Section \ref{threepunctures}. 
Finally, the cases of  spheres with 4 or 5 punctures are definitely interesting, but we  have to exclude them because the crucial  results of Section \ref{sunny} have not been established for these surfaces.

A marked bordered   surface   $(\S,\M)$ is determined up to diffeomorphism by its genus $g$, the number of punctures $p$, and a collection of integers $k_i\geq 1$ encoding the number of marked points on each boundary component.
Any ideal triangulation of such a surface has the same number of edges, namely
\[n = 6g-6+3p+\sum_i (k_i+3).\]


\subsection{Flips and pops}
\label{flops}

 Let $(\S,\M)$ be a marked bordered surface satisfying Assumption \ref{asstwo}.
 A \emph{signed triangulation} of $(\S,\M)$ is a pair $(T,\epsilon)$ consisting of an ideal triangulation $T$, and a function \[\epsilon\colon \bP\to \{\pm 1\}.\]
 By a \emph{pop} of a signed triangulation $(T,\epsilon)$ we mean the operation of changing the sign $\epsilon(p)$ associated to a puncture $p\in \PP$ of valency one. Note that any such puncture lies in the centre of a self-folded triangle of $T$.
 
 The popping operation generates an equivalence relation on signed triangulations, in which two signed triangulations $(T_i,\epsilon_i)$ are  equivalent precisely if the underlying triangulations $T_i$ are the same, and the signings $\epsilon_i$  differ only at punctures  $p\in \PP$ of valency one. 
  
It turns out that the  equivalence classes  for this relation can be explicitly represented by a combinatorial gadget called a \emph{tagged triangulation}. We will explain this in Section \ref{tagg} below, but for now we simply  define a tagged triangulation to be an equivalence class of signed triangulations.

Let us introduce notation
$\Tri(\S,\M),  \Tri_\pm(\S,\M),  \Tri_{\bowtie}(\S,\M)$
 for the sets of ideal, signed and tagged  triangulations of $(\S,\M)$ respectively. There is a diagram of maps
\begin{equation}\label{diagr}\xymatrix{ 
&\Tri_{\pm}(\S,\M)\ar[d]_{q} \\
\Tri(\S,\M) \ar[r]^{i} \ar[ur]^{j}&\Tri_{\bowtie}(\S,\M)}\end{equation}
where $q$ is the obvious quotient map, and the  arrows $i$ and $j$ are embeddings obtained by  considering an ideal triangulation as a signed, and hence a tagged triangulation,  using the signing $\epsilon\equiv +1$.

 Two ideal triangulations $\T_1$ and $\T_2$ are  related by a \emph{flip} if  they are distinct, and there are edges $e_i\in \T_i$  such that $\T_1\setminus \{e_1\}=\T_2\setminus \{e_2\}$.
Note that the edges $e_1$ and $e_2$ are  necessarily non-self-folded.
Conversely, if $e$ is a non-self-folded edge of an ideal triangulation $T$, it is contained in exactly two triangles of $T$, and  there is a unique ideal triangulation which is the flip of $T$ along $e$. \begin{figure}[ht]
\begin{center}
\includegraphics[scale=0.4]{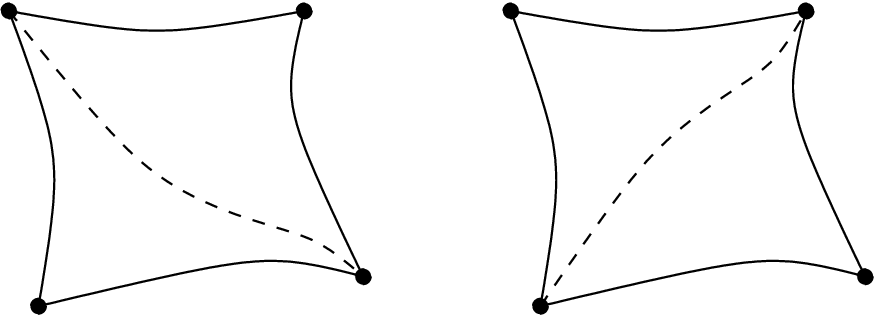}
\end{center}
\caption{Flip of a triangulation.\label{flippy}}
\end{figure}
The flipping  operation extends to signed triangulations in the obvious way: we flip the underlying triangulation, keeping the signs constant.
We  say that two tagged triangulations are related by a flip if they  can be represented by signed triangulations which differ by a flip.

The sets  appearing in the diagram \eqref{diagr} can be considered as graphs, with two (ideal, signed, tagged) triangulations being connected by an edge if they differ by a flip.  The maps in the diagram then become maps of graphs. 
The important point is that, unlike the graph $\Tri(\S,\M)$ of ideal triangulations, the graph $\Tri_{\bowtie}(\S,\M)$ of tagged triangulations  is $n$-regular.

The basic explanation for this regularity is as follows. When a triangulation $T_1$ contains a self-folded triangle $\Delta$, we cannot flip the self-folded edge $f$  of $\Delta$, so the number of flips that can be performed on $T_1$ is less than  the total number of edges $n$. On the other hand,  if we choose a signing $\epsilon\colon \PP\to \{\pm 1\}$, and consider  the signed triangulation $(T_1,\epsilon_1)$ up to the above equivalence relation, then when we flip the encircling edge $e$ of $\Delta$,  the puncture $p$ contained in $\Delta$ has valency 2 in the new triangulation $T_2$, and so there  are  two inequivalent possible choices for the sign $\epsilon_2(p)$. 

It is well-known that any two ideal triangulations of $(\S,\M)$ are related by a finite chain of flips; thus the graph $\Tri(\S,\M)$ is always connected \cite[Prop. 3.8]{FST}.
The graph $\Tri_{\bowtie}(\S,\M)$ is also connected, except for the case when $(\S,\M)$ is a closed surface with a single puncture $p\in \PP$: in that case   $\Tri_{\bowtie}(\S,\M)$ has two connected components corresponding to the two possible choice of signs $\epsilon(p)$ \cite[Prop. 7.10]{FST}.


\subsection{Tagged triangulations}
\label{tagg}
We now explain why the set $\Tri_{\bowtie}(\S,\M)$ of tagged triangulations we defined above coincides with the standard version as defined by Fomin, Shapiro and Thurston  \cite{FST}.  This material will not be used in the rest of the paper, and is only logically necessary to justify the above assertions  that the graph $\Tri_{\bowtie}(\S,\M)$ is connected and $n$-regular.

Let $(\S,\M)$ be a marked, bordered surface $(\S,\M)$ satisfying Assumption \ref{asstwo}. 
A \emph{tagged arc} in $(\S,\M)$ is an arc as defined above, each end of which has been labelled by one of two labels:  plain or tagged.
Fix a function $\epsilon\colon \bP\to \{\pm 1\}$. Given an ordinary arc $e$, there is a  corresponding tagged arc $t_\epsilon(e)$ defined by the following rule:
 \begin{itemize}
\item[(a)]
 If $e$ is not a loop enclosing a once-punctured disc, the underlying arc of $t_{\epsilon}(e)$ is just $e$, and an end of $e$ is labelled tagged precisely if it lies at  a puncture $p\in \bP$ with $\epsilon(p)=-1$. \\

\item[(b)]If $e$ is a loop based at $m\in \M$, enclosing a disc which contains a single puncture $p\in \bP$, then the underlying arc of $t_\epsilon(e)$ is the arc connnecting $p$ to $m$ inside the disc.  We label the edge adjacent to $m$ tagged precisely if $m$  is a puncture with $\epsilon(m)=-1$, and the edge adjacent to $p$ tagged precisely if $\epsilon(p)=+1$.
\end{itemize}

By \cite[Lemma 9.3]{FST},
 a  tagged triangulation in the standard sense considered there is precisely  a set of tagged arcs of the form $t_\epsilon(T)$ for some signed triangulation $(T,\epsilon)$.  The following result shows  that these tagged triangulations are in bijection with the equivalence classes of signed triangulations considered in the last section.

\begin{lemma}
 \label{willy}
Suppose that $(T_1,\epsilon_1)$ and $(T_2,\epsilon_2)$ are signed triangulations. Then  $t_{\epsilon_1}(T_1)=t_{\epsilon_2}(T_2)$ if and only if $T_1=T_2$, and the signings $\epsilon_1, \epsilon_2$ differ only at punctures of valency one.
\end{lemma}

\begin{pf}
Any vertex of valency one  lies in the interior of a self-folded triangle, and it is clear from the definition of $t_\epsilon(T)$ that the resulting collection of tagged arcs does not distinguish between the two choices of sign at such a vertex (the enclosing and folded edge of the self-folded triangle get mapped to two taggings of the same arc; changing the sign just exchanges these two). 

The converse follows easily from the following observations. Suppose that $\eta\in t_\epsilon(T)$ is a tagged arc, with underlying arc $f$. Consider tagged arcs $\zeta\in t_\epsilon(T)$ which have the same underlying arc $f$. If there is no such  $\zeta$ then we must have $\eta=t_\epsilon(f)$, and the arc $f$  is not contained in a self-folded triangle of $T$. If there is such a $\zeta$, then there is a self-folded triangle $\Delta$ in $T$ with self-folded edge $f$, and encircling edge $e$, such that $\{\eta,\zeta\}=\{t_\epsilon(e),t_\epsilon(f)\}$.  Moreover, the encircling edge $e$ is completely determined, because the puncture inside $\Delta$ is the one at which $\eta$ and $\zeta$ have different markings. 
\end{pf}

We should also check that our definition of when two tagged triangulations differ by a flip coincides with the standard one.
 Namely, in \cite{FST}, two tagged triangulations $\tau_1$ and $\tau_2$ are said to be related by a flip if they are distinct, and there are tagged arcs $\eta_i\in \tau_i$ such that $\tau_1\setminus \{\eta_1\}=\tau_2\setminus \{\eta_2\}$.

\begin{lemma}
 Two tagged triangulations differ by a flip in the above sense, precisely if they can be represented by signed triangulations differing by a flip.
\end{lemma}

\begin{pf} 
One implication is clear, since if two signed triangulations differ by a flip then by Lemma \ref{willy}, so do the associated tagged triangulations. For the converse, suppose that $\tau_1\setminus \{\eta_1\}=\tau_2\setminus \{\eta_2\}$. Then we can write $\tau_1=t_\epsilon(T_1)$ for some signed triangulation $(T_1,\epsilon)$ in such a way that $\eta_1=t_\epsilon(e_1)$ for some non-self-folded edge $e_1$. Flipping this edge gives a different signed triangulation $T_2$ satisfying $T_1\setminus\{e_1\}=T_2\setminus\{e_2\}$. It follows that $\tau_1\setminus\{\eta_1\}=t_\epsilon(T_2)\setminus t_\epsilon(e_2)$. But the flip of a  tagged triangulation in a  tagged arc is unique \cite[Theorem 7.9]{FST}.
Thus we have $\tau_2=t_\epsilon(T_2)$ and $\eta_2=t_\epsilon(e_2)$.
\end{pf}


\subsection{Edge lattice and quiver}
\label{edgelattice}

Let $(\S,\M)$ be a marked bordered surface satisfying Assumption \ref{asstwo}. The \emph{edge lattice}  of an ideal triangulation $T$ of $(\S,\M)$ is defined to be the free abelian group $\Gamma(T)$  on the edges of $T$. We denote by $[e]$ the basis element corresponding to the edge $e\in T$; thus
\[\Gamma(T)=\bigoplus_{e\in T}\Z\cdot  [e].\]
For distinct edges $e,f\in T$, we define $c(e,f)$ 
to be the number of  triangles of $T$ in which $e$ and $f$ appear as adjacent edges in clockwise order.
There is  a  skew-symmetric form \[\langle-,-\rangle \colon \Gamma(T)\times \Gamma(T)\to \Z\]
 given by the formula\[\langle [e],[f]\rangle=c(f,e)-c(e,f).\]  Note that if $e$ and $f$ are the encircling and self-folded edges of a self-folded triangle then
\begin{equation}
\label{deg}c(e,f)=1=c(f,e), \quad \langle [e],[f]\rangle =0.\end{equation}
It will be convenient for  later purposes to define $c(e,e)=-2$ for all edges $e$, although of course this has no effect on the form $\langle-,-\rangle$.

 We will also need a modified basis $\{e\}$ for the group $\Gamma(T)$ defined as follows: 
\begin{itemize}
\item[(a)]
if $e\in T$ is not an edge of a self-folded triangle then  $\{e\}=[e]$; \smallskip
\item[(b)]
 if $e$ and $f$ are respectively the encircling and self-folded edges of a self-folded triangle,  then
  $\{e\}=[e]$ and $\{f\}=[e]+[f]$.
 \end{itemize}
 We will  see some more intrinsic interpretations of the edge lattice  $\Gamma(T)$ later:  as a Grothendieck group  with its Euler form (Lemma \ref{hope}), and as a homology group with an intersection form (Lemma \ref{move}).   We will also give some   explanation for the strange-looking definition of the basis $\{e\}$ (Section \ref{sing}).
 
Define a map $\kappa\colon T\to T$  by setting $\kappa(f)=f$ unless $f$ is a self-folded edge of a self-folded triangle, in which case $\kappa(f)=e$ is the encircling edge of the same triangle. For distinct edges $e$ and $f$ define
\[n(e,f)=\max\big(0,\langle [\kappa(f)],[\kappa(e)]\rangle\big)\geq 0.\]
Note that there is a relation
\begin{equation}\label{grass}\langle \{e\},\{f\}\rangle=\langle [\kappa(e)],[\kappa(f)]\rangle =n(f,e)-n(e,f)\end{equation}for all $e,f\in T$.
This is easily checked by noting that when $f$ is a self-folded edge, the basis element $[f]$ lies in  the kernel of the form $\langle -,-\rangle$.

To any ideal triangulation $T$ we can now associate  a quiver $Q(T)$ whose vertices are the edges of $T$, and with $n(e,f)$ arrows from vertex $e$ to vertex $f$. 
By its definition it has no loops or 2-cycles. In the case of a non-degenerate ideal triangulation of a closed surface it reduces to the quiver considered in Section \ref{fir}.

\begin{remarks}
\label{remy}
\itemize
\item[(a)]
 If $T$ has a self-folded triangle with edges $e$ and $f$, then since $\kappa(e)=\kappa(f)$,  there is an involution of the quiver $Q(T)$ exchanging the vertices corresponding to these two edges. \smallskip

\item[(b)]  If $e$ and $f$ are distinct  non-self-folded edges,  then  it is easily checked that $c(e,f)$ and $c(f,e)$ are both nonzero precisely if $e$ and $f$ meet at a puncture of valency 2. Thus if a pair of edges $e$, $f$ are such that $\kappa(e)$ are  $\kappa(f)$ are distinct, and do not meet at a vertex of valency 2, then  $n(e,f)=c(\kappa(e),\kappa(f))$.
\end{remarks}


\subsection{Ordered versions}
\label{ordvers}

Sometimes in what follows it will be clearer to work with ordered versions of our basic combinatorial objects: triangulations, quivers, t-structures etc.  In this section we gather the necessary definitions; these mostly  proceed along the obvious lines. 

An ordered ideal triangulation is an ideal triangulation equipped with an ordering of its edges. Similarly, one can consider ordered signed triangulations. By a   pop of an ordered signed triangulation  we mean the operation which changes the sign $\epsilon(p)$ associated to a puncture $p\in \PP$ of valency 1,  and which also changes the ordering of the  triangulation by transposing  the two edges of the self-folded triangle containing $p$. Two ordered signed triangulations are considered equivalent if they differ by a finite sequence of  such pops. By an ordered tagged triangulation we mean an equivalence class of ordered signed tagged triangulations.  The map $t_\epsilon$ of Section \ref{tagg} respects  this equivalence relation, and it follows that we can realise ordered tagged triangulations as  ordered collections of tagged arcs.

 Two ordered  triangulations are related by a flip if  the underlying  triangulations are related by a flip, and if the orderings of their edges are compatible with the obvious bijection between the edges of the two triangulations. Similarly, one can consider flips of ordered signed triangulations.
Two ordered tagged triangulations are related by a flip if they can be represented by ordered signed triangulations that are related by a flip.

An ordered  quiver is a quiver equipped  with a fixed ordering of its vertices. An ordered triangulation $T$ has an associated ordered quiver $Q(T)$.  Remark \ref{remy}(a) shows that the ordered quiver associated to an ordered triangulation is invariant under transposing the order of the two edges of  a self-folded triangle, and it follows that every ordered tagged triangulation  also has an associated ordered quiver. 

Finally, suppose that $\D$ is  a \CY triangulated category. By an ordering of a finite-length heart $\A\subset \D$ we mean an ordering of the  simple objects of $\A$. The associated quiver $Q(\A)$ is then also  ordered in the obvious way. As we explained in the proof of Proposition \ref{rpt},  for any simple object $S\in \A$, there is a canonical bijection between the simple objects of the heart $\A$ and those of the tilted heart  $\mu^\pm_{S}(\A)$. We say that two ordered hearts $\A,\B\subset \D$ are related by a tilt in a simple object,  if the hearts $\A,\B$ are related by such a tilt, and if the orderings on $\A,\B$ are compatible with this canonical bijection.

 We denote the  graphs of ordered ideal, signed, tagged triangulations by \[\Triord(\S,\M),\quad \Triord_{\pm}(\S,\M), \quad \Triord_{\bowtie}(\S,\M)\] respectively. The maps in \eqref{diagr} induce maps of the ordered versions in the obvious way.
Similarly, given a non-degenerate heart $\A\subset \D$ we use the notation
\[\Tilt_\A^\ord(\D), \quad \Ex_\A^\ord(\D)\]
for the graphs of ordered reachable finite-length hearts, and the quotient of this graph by the group $\Sph_\A(\D)$. We note that these graphs will not be connected in general.


\subsection{Mapping class group}

 By a diffeomorphism of a marked bordered surface $(\S,\M)$ we mean a diffeomorphism of $\S$ which fixes the subset $\M$, although possibly permuting its elements.
The \emph{mapping class group}  $\MCG(\S,\M)$ is the group of all orientation-preserving diffeomorphisms of $(\S,\M)$ modulo  those which are homotopic to the identity  through diffeomorphisms of $(\S,\M)$.

The mapping class group  clearly acts on the graphs of (ideal, signed, tagged) triangulations of the surface $(\S,\M)$, since the edges of such triangulations consist of homotopy classes of arcs.

\begin{prop}
\label{weiwen}
Suppose that $(\S,\M)$ is a marked bordered surface which satisfies Assumption \ref{asstwo} and which is not one of the following 3 surfaces
\begin{itemize}
\item[(a)] a once-punctured disc with 2 or 4 marked points on the boundary;\smallskip
\item[(b)]  a 
 twice-punctured disc with 2 marked points on the boundary.
 \end{itemize}
Then two ideal triangulations  of $(\S,\M)$ differ by an element of $\MCG(\S,\M)$ precisely if the associated quivers   are isomorphic.  
\end{prop}

\begin{pf}
One implication is clear: if two triangulations differ by an orientation-preserving  diffeomorphism then they have the same combinatorics and hence the same associated quivers.

For the converse, 
suppose that two surfaces $(\S_i,\M_i)$ have ideal triangulations $T_i$. 
In \cite[Section 13]{FST} it is explained how to decompose the quivers $Q(T_i)$  into certain blocks. In the proof of \cite[Proposition 14.1]{FST} it is shown that if each quiver $Q(T_i)$ has a unique block decomposition, then the quivers $Q(T_i)$  are isomorphic precisely if there is an orientation-preserving  diffeomorphism between the  surfaces $(\S_i,\M_i)$ taking one triangulation $T_i$ to the other.  
Weiwen Gu \cite{Gu} has classified all quivers which have more than one block decomposition. The only examples corresponding to combinatorially-distinct triangulations of the same surface occur when $(\S,\M)$ is one of the three cases listed in the statement of the Proposition, or a sphere with $3$ or $4$ punctures. These last two cases are already excluded by Assumption \ref{asstwo}.
\end{pf}


The mapping class group of $(\S,\M)$  acts on the set $\PP\subset \M$ of punctures in the obvious way, and we define the \emph{signed mapping class group} to be the corresponding semi-direct product
\begin{equation}
\label{semi}
\MCG^\pm(\S,\M)=\MCG(\S,\M)\ltimes \Z_2^{\PP}.\end{equation}
This group acts on the set of  signed triangulations, with the $\Z_2^{\PP}$ part acting by changing the signs $\epsilon(p)\in \{\pm 1\}$ associated to the punctures.  This action clearly descends to an action on tagged triangulations.
It is an immediate consequence of Proposition \ref{weiwen} that the quivers associated to two signed or tagged triangulations of $(\S,\M)$ are isomorphic  precisely if they differ by an element of the signed mapping class group.


\subsection{Free action on ordered triangulations}

One reason to introduce  ordered triangulations is the following  result.

\begin{prop}
\label{freely}
Suppose that $(\S,\M)$ satisfies Assumption \ref{asstwo} and is not one of the following 3 surfaces:
\begin{itemize}
\item[(i)] an unpunctured disc with 4 points on its boundary;\smallskip
\item[(ii)]  an annulus with one marked point on each boundary component;\smallskip
\item[(iii)] a closed torus with a single puncture.
\end{itemize}
Then the action of the mapping class group $\MCG(\S,\M)$  on the set $\Triord(\S,\M)$ of ordered ideal triangulations   is free. Similarly, the actions of the signed mapping class group  $\MCG^\pm(\S,\M)$ on the sets $\Triord_\pm(\S,\M)$ and $\Triord_{\bowtie}(\S,\M)$  are free.
\end{prop}

\begin{pf}
For the case of ideal triangulations, we must show  that an orientation-preserving diffeomorphism $g$ of $(\S,\M)$ which fixes the edges of  such a  
triangulation $T$ is homotopic to the identity, through diffeomorphisms of $(\S,\M)$. 

Suppose first that  $g$ fixes every  triangle of $T$.  Since $g$ then induces an orientation-preserving diffeomorphism of  each  triangle, it follows that $g$ preserves the orientation of each edge of $T$. Moreover, since every triangle contains at least one edge,  we see that $g$ preserves each connected component of $\partial \S\setminus \M$. Now $\Diff(I, \partial I)$ is  
contractible, and the edges  are disjoint in their  
interiors,  so we can isotope $g$ so that it fixes all edges of $T$, and all  components of $\partial \S$  pointwise.  The result then follows from the fact  \cite[Theorem B]{Smale}  that the group  $\Diff(D^2, \partial D^2)$ is contractible. 

Suppose instead that there is some   triangle $\Delta$ such that $g(\Delta)\neq \Delta$.   Then the edges of $\Delta$ coincide with those of $g^i(\Delta)$ for all $i\in \Z$.  Since   any edge occurs in the boundary of at most 2 triangles, it follows that $g^2(\Delta)=\Delta$. Moreover,   the surface $\S$ is completely covered by the   two triangles $\Delta$ and $g(\Delta)$, since passing through an edge of $\Delta$ takes us  into $g(\Delta)$ and vice versa.
We then obtain the three possibilities listed, according to whether the closures of the triangles $\Delta$ and $g(\Delta)$ meet in 1,2 or 3 edges.

The extension to signed triangulations is obvious. For the case of tagged triangulations, suppose that an element of $g\in \MCG^\pm(\S,\M)$ fixes an ordered tagged triangulation $\tau$, which we view as an equivalence-class of signed triangulations $(T,\epsilon)$. Note that the action of $g$ on $\Tri^\ord_\pm(\S,\M)$  commutes with the flipping operation. Thus we can reduce to the case when $T$ has no self-folded triangles. Then $g$ must fix the signed triangulation $(T,\epsilon)$, since this is the only signed triangulation in the equivalence-class $\tau$. Hence $g$ is the identity. \end{pf}

We note that the excluded cases in Proposition \ref{freely} are essentially the same as those in Proposition \ref{reff} (the case of an unpunctured disc with 3 marked points would appear above were it not already excluded by Assumption \ref{asstwo}). This is not a coincidence: generic automorphisms of the space $\Quad(\S,\M)$ correspond to automorphisms of $(\S,\M)$ which preserve a horizontal strip decomposition together with an ordering of the horizontal strips. Any such automorphism will also preserve an ordered version of the WKB triangulation of Section \ref{wkb}.

\section{The category associated to a surface}
\label{catsurf}
In this section we introduce the particular examples of \CY categories  that appear in our main Theorems. They are indexed by diffeomorphism classes of marked bordered surfaces $(\S,\M)$. We also provide the combinatorial underpinning of our main theorems, by giving a precise correspondence between tagged triangulations of $(\S,\M)$  and   t-structures on the corresponding category $\D(\S,\M)$. Throughout this section we rely heavily on the work of D. Labardini-Fragoso.


 \subsection{Some results of Labardini-Fragoso}
 \label{sunny}

Let $(\S,\M)$ be a marked bordered surface satisfying Assumption \ref{asstwo}. Let $T$ be  an ideal triangulation of $(\S,\M)$.
Labardini-Fragoso \cite{LF1} defined a reduced potential $W(\T)$ on the quiver $Q(T)$  introduced in Section \ref{edgelattice}, depending also on some nonzero scalar constants $x_p\in k\setminus\{0\}$, one for each puncture $p\in \bP$. We shall always take these scalars to be defined by a signing; thus we take a signed triangulation $(T,\epsilon)$ and consider the  quiver with potential $(Q(T),W(T,\epsilon))$ obtained by setting $x_p=\epsilon(p)$.

In the case when $T$ is non-degenerate, the resulting potential reads
\begin{equation}
\label{bonn}W(\T,\epsilon)=\sum_{f} T(f) - \sum_{p} \epsilon(p) C(p),\end{equation}
where $T(f)$ and $C(p)$ are the  cycles in $Q(T)$ defined in Section \ref{fir}.
In the presence of  punctures of valency $\leq 2$ the recipe becomes more complicated.
The explicit form of the  potential  will not be important for what follows, and we refer to \cite[Section 3]{LF1} for details. What will be important are the invariance properties under flips and pops which we now describe.

 The case of flips is dealt with by the following result of Labardini-Fragoso \cite[Theorem 30]{LF1}. To be absolutely clear we state it for ordered triangulations.

\begin{thm}
\label{lab}
Let $(\S,\M)$ be a marked bordered surface satisfying Assumption \ref{asstwo}.
Suppose that two ordered signed triangulations $(\T_i,\epsilon_i)$ of $(\S,\M)$  are
  related by a flip in a non-self-folded edge $e$. Then, up to right-equivalence, the ordered  quivers with potential $(Q(T_i),W(T_i,\epsilon_i))$ are related by a mutation at the corresponding vertex. 
\end{thm}

We now move on to the case of pops.  As we remarked in Section \ref{ordvers},   the popping symmetry of Remark \ref{remy}(a) implies that an ordered  tagged triangulation  has an associated ordered quiver.  The following result \cite[Theorem 6.1]{LF4} extends this statement to quivers with potential.

\begin{thm}
\label{blah}
Let $(\S,\M)$ be a marked bordered surface satisfying Assumption \ref{asstwo}. Suppose that two ordered signed triangulations $(T,\epsilon_i)$ of $(\S,\M)$ are related by a pop. Then the associated ordered quivers with potential $(Q(T),W(T,\epsilon_i)$ are right-equivalent.
\end{thm}

We  conclude that every ordered tagged triangulation  has an associated ordered quiver with potential, well-defined up to right-equivalence. If we think of a tagged triangulation  $\tau$ as a collection of tagged arcs as in Section \ref{tagg}, then we can say that  there is an associated quiver with potential $(Q(\tau),W(\tau))$, well-defined up to right-equivalence, whose   vertices are in natural bijection with these tagged arcs.

To avoid all technical difficulties  we shall mostly work with the following class of surfaces. We return to some of the exceptional cases in Section \ref{therest}.

\begin{defn}
\label{amenable}
We say that a marked bordered surface $(\S,\M)$ is \emph{amenable} if 
\begin{itemize}
\item[(a)] $(\S,\M)$ satisfies Assumption \ref{asstwo};\smallskip
\item[(b)] $(\S,\M)$ is not one of the 3 surfaces listed in Proposition \ref{weiwen};\smallskip
\item[(c)] $(\S,\M)$ is not one of the 3 surfaces listed in Proposition \ref{freely};\smallskip
\item[(d)] $(\S,\M)$ is not a closed surface with a single puncture.
\end{itemize}
\end{defn}

Note for example that $(\S,\M)$ is amenable if $g(\S)>0$ and $|\M|>1$.

Recall from Section \ref{ginz} the definition of a non-degenerate quiver with potential.

\begin{thm}
 Suppose  that the marked bordered surface $(\S,\M)$ is amenable. Then the quiver with potential $(Q(T),W(T,\epsilon))$ associated to  any signed triangulation of $(\S,\M)$ is non-degenerate.
\end{thm}

\begin{pf}Up to right-equivalence the quiver with potential  $(Q(T),W(T,\epsilon))$ associated to a signed triangulation depends only on the corresponding tagged triangulation.  Since tagged triangulations can be flipped in any edge, and such flips can be expressed as  flips of signed triangulations to which Theorem \ref{lab} applies, we conclude that every mutation of $(Q(T),W(T,\epsilon))$ is of the same type. The result follows, since by definition, none of the quivers $Q(T)$ has  loops or oriented 2-cycles.
\end{pf}


\subsection{Definition of the category}
Let $(\S,\M)$ be an amenable marked bordered surface. In this section we introduce the associated \CY triangulated category $\D(\S,\M)$, well-defined up to $k$-linear triangulated equivalence.

Given a  signed triangulation  $(T,\epsilon)$ of  $(\S,\M)$, let us write 
  \[\D(T,\epsilon)=\D(Q(T),W(T,\epsilon))\]for  the  \CY category defined by the quiver with potential considered  in the last subsection. 
This category  comes equipped with a standard  heart  \[\A(T,\epsilon)\subset \D(T,\epsilon),\] whose simple objects  $S_e$ are  indexed by the edges $e$ of the  triangulation $T$.

Combining  Theorem \ref{keller_yang} and Theorem \ref{lab}  immediately gives

\begin{thm}
\label{fl}
Suppose that two signed triangulations $(T_i,\epsilon_i)$ of $(\S,\M)$ differ by a flip in an edge $e$. Then  there is a canonical pair of $k$-linear  triangulated equivalences 
\[\Phi_\pm\colon \D(T_1,\epsilon_1)\to \D(T_2,\epsilon_2),\]
satisfying $\Phi_\pm(\A(T_1,\epsilon_1))=\mu_{S_e}^\pm(\A(T_2,\epsilon_2))$ and inducing the natural bijection on simple objects. 
\end{thm}

Similarly, Conjecture \ref{conj}, which holds by the assumption that $(\S,\M)$ is amenable, implies

\begin{thm}
\label{pop}
Suppose that two signed triangulations $(\T,\epsilon_i)$ of $(\S,\M)$  are
  related by a pop at a puncture $p\in \PP$. Then there is a $k$-linear triangulated  equivalence \[\Psi\colon \D(T,\epsilon_1)\to \D(T,\epsilon_2)\]
which identifies the standard hearts, and exchanges the two simple objects $S_e$ and $S_f$ corresponding to the 
two edges of the self-folded triangle containing $p$.
\end{thm}

Since the graph $\Tri_{\bowtie}(\S,\M)$ is connected, these two results show that, up to $k$-linear triangulated equivalence, the category $\D(T,\epsilon)$ depends only on the surface $(\S,\M)$ and not on the chosen signed triangulation.
 Thus we can associate to the surface $(\S,\M)$ a \CY triangulated category
  \[\D=\D(\S,\M).\]
 In fact it will be important to make a slightly stronger statement, as we now explain.
 Each category $\D(T,\epsilon)$ comes with a distinguished connected component of its tilting graph, namely the one containing the standard heart. Moreover,  if $(T_i,\epsilon_i)$ are two signed triangulations of  $(\S,\M)$ then by composing the equivalences of Theorems \ref{fl} and \ref{pop} we obtain   equivalences $\D(T_1,\epsilon_1)\isom \D(T_2,\epsilon_2)$ which identify these connected components. Thus the category $\D$ comes  equipped with a distinguished connected component \[\Tilt_\triangle(\D)\subset \Tilt(\D).\]
 Adapting the general notation from Section \ref{quivstab}, we write \[\Reach(\D)\subset \Aut(\D)\] for the group of
autoequivalences of $\D$ which preserve this connected component; such autoequivalences will be called \emph{reachable}.  We write \[\Nil(\D)\subset \Reach(\D)\] for the autoequivalences which act trivially on $\Tilt_\triangle(\D)$; we call these
autoequivalences \emph{negligible}.

Negligible autoequivalences fix the simple objects of all the hearts  in our distinguished component $\Tilt_\triangle(\D)$. They will also act trivially on the corresponding distinguished connected component $\Stab_\triangle(\D)$. For this reason, it is useful  to  consider the quotient group
\[ \uReach(\D)=\Reach(\D)/\Nil(\D)\]
which acts effectively on these spaces.

\begin{remark}
\label{completed}
We could instead consider defining a category $\D^\dagger(\S,\M)$ by using uncompleted Ginzburg algebras, rather than the complete ones we are using here. It seems likely that if $(\S,\M)$ is  amenable  the resulting category would be equivalent to $\D(\S,\M)$. 
As evidence for this, note that in the case when $\S$ has non-empty boundary the natural map $J^\dagger(Q,W)\to J(Q,W)$ from the uncompleted Jacobi algebra to the complete version is an isomorphism \cite[Theorem 5.7]{LF3}.  
When the surface is not amenable this  statement can definitely fail. For example, when $(\S,\M)$ is a closed torus with a single puncture the algebra $J^\dagger(Q,W)$ is infinite-dimensional, whereas $J(Q,W)$  is  finite-dimensional \cite[Example 8.2]{LF11}. See also Example \ref{threepunctures} below for the case of the three-punctured sphere.
\end{remark}


\subsection{T-structures and autoequivalences}
\label{au}

 Let $(\S,\M)$ be an amenable  marked bordered surface, and let   $\D=\D(\S,\M)$ be the associated \CY triangulated category. In this section we study the distinguished connected component $\Tilt_\triangle(\D)$ of the tilting graph of $\D$, and the corresponding group $\Reach(\D)$ of  reachable autoequivalences. 

Recall from Section \ref{tiltcy} that there is a subgroup   \[\Sph_\triangle(\D)\subset \Reach(\D),\]
 generated by the twist functors $\Tw_{S_i}$ in the simple objects of any heart $\A\in \Tilt_\triangle(\D)$. We write
 \[\uTwist(\D)\subset \uReach(\D),\]
 for the corresponding subgroup of $\uReach(\D)$.
  The group $\Sph_\triangle(\D)$  acts on the tilting graph $\Tilt(\D)$, and  Proposition \ref{rpt} implies that this action preserves the connected component $\Tilt_\triangle(\D)$. We call the quotient graph \[\Ex_\triangle(\D)=\Tilt_\triangle(\D)/\Sph_\triangle(\D)\]
the \emph{heart exchange graph} of $(\S,\M)$. There is also an ordered version $\Ex^\ord_\triangle(\D)$ defined in the obvious way.

The following result gives the basic link between  triangulations of the surface $(\S,\M)$ and t-structures in the corresponding  category $\D(\S,\M)$. 

\begin{thm}
\label{kell}
There are isomorphisms of graphs
\[ \Tri_{\bowtie}(\S,\M) \isom  \Ex_\triangle(\D), \quad \Tri_{\bowtie}^\ord(\S,\M) \isom  \Ex^\ord_\triangle(\D). \]
\end{thm}

\begin{pf}
 Corollary 3.6 of \cite{CKLP} and Theorem 5.6 of \cite{K} together imply that the heart exchange graph $\Ex_\triangle(\D)$ is isomorphic to the cluster exchange graph. On the other hand,  the cluster exchange graph 
was shown to be isomorphic to the  tagged triangulation graph $\Tri_{\bowtie}(\S,\M)$ in \cite[Theorem 7.11]{FST} and \cite{FT}. Both of these isomorphisms are constructed in such a way that they lift to maps of the corresponding ordered graphs, and it follows that these maps are also isomorphisms.
\end{pf}

We note that the isomorphisms  of Theorem \ref{kell} have the property that if a tagged triangulation is represented by a signed triangulation $(T,\epsilon)$, then the corresponding heart $\A\subset \D$ is the image of the standard heart $\A(T,\epsilon)$ under an equivalence
\[\D(T,\epsilon)\isom \D.\]
The obvious generalization to the ordered versions also holds.
We can use this result to prove\footnote{The method of proof was explained to us by Alastair King.} a result on the structure of the   group $\uReach(\D)$.

\begin{thm}
\label{clareshort}
Assume that $(\S,\M)$ is amenable. Then
there is a short exact sequence
\[1\lra {\uTwist}(\D)\lra {\uAll}(\D)\lra \MCG^\pm(\S,\M)\lra 1.\]
\end{thm}

\begin{pf}
Consider the $n$-fold free product group
\[F_n=\Z_2 * \cdots * \Z_2 = \langle \mu_1,\cdots,\mu_n: \mu_i^2=1\rangle.\]
There is an obvious action of the symmetric group $\Sym_n$ permuting the generators $\mu_i$, and we  also consider  the semi-direct product
\[G_n=\Sym_n\ltimes F_n.\]
The  set  $\Tri_{\bowtie}^\ord(\S,\M)$  of ordered, tagged triangulations has a natural   action of the group $G_n$:
the  generator $\mu_i$ acts by flipping the $i$th edge of an ordered  triangulation, and the  group $\Sym_n$ acts by permuting the ordering of the edges. This action is transitive, because, with our assumptions on $\S$,  the graph of unordered tagged triangulations $\Tri_{\bowtie}(\S,\M)$ is connected.

In a similar way, the set $\Ex^\ord_\triangle(\D)$ of ordered reachable hearts has a transitive action of the group $G_n$. This time the generator $\mu_i$ acts by  tilting an ordered heart at the $i$th simple object. Note that the left and right tilts co-incide on the exchange graph by Proposition \ref{rpt}(a), and the relation  $\mu_i^2=1$  follows from the relation \eqref{jen}. 

Consider the set $\mathcal{Q}_n$  of right-equivalence classes of  ordered quivers with $n$ vertices. It carries an action of the group $G_n$, where the  generator $\mu_i$ acts by mutation at the $i$th vertex. We now have a commutative diagram of  $G_n$-equivariant maps of sets  \begin{equation}\label{al}\xymatrix@C=1.5em{ \Tri_{\bowtie}^\ord(\S,\M) \ar[dr]_{p} \ar[rr]^{\theta} && \Ex^\ord_\triangle(\D)\ar[dl]^{q} \\
&\mathcal{Q}_n}\end{equation}
where  $p$ sends an ordered tagged triangulation  $(T,\epsilon)$ to the associated ordered quiver  $Q(T)$,  the map $q$ sends an ordered heart $\A\subset \D$ to the associated  quiver $Q(\A)$, and $\theta$ is the bijection of Theorem \ref{kell}.

 The signed mapping class group $\MCG^\pm(\S,\M)$ acts freely on   $\Tri_{\bowtie}^\ord(\S,\M)$ by Proposition \ref{freely} and obviously commutes with the $G_n$-action. Proposition \ref{weiwen} shows that the orbits for this action are precisely the fibres of $p$.  It then follows from the transitivity of the $G_n$-action that $\MCG^\pm(\S,\M)$  can be identified with the group of automorphisms of the set $\Tri_{\bowtie}^\ord(\S,\M)$ which commute with the $G_n$-action and preserve the map $p$. 
 
 By the definition of a negligible autoequivalence, the group $\uReach(\D)$ acts freely on the graph $\Tilt^\ord_\triangle(\D)$ of ordered reachable hearts. Quotienting $\uReach(\D)$ by the normal subgroup $\uTwist(\D)$, we therefore obtain a free action of the group \begin{equation}
 \label{stan}\uReach(\D)/\uTwist(\D)\end{equation} on  the set $\Ex^\ord_\triangle(\D)$,  commuting with the $G_n$-action.
 To complete the proof we must show that the orbits of $\uReach(\D)$ are precisely the fibres of the map $q$.
  
  Suppose that two ordered hearts $\A_i\subset \D$ lie in the same fibre of $q$. Under the bijection $\theta$, these hearts correspond to ordered tagged triangulations $(T_i,\epsilon_i)$ lying in the same fibre of $p$. It follows that they differ by an element  $g\in \MCG^\pm(\S,\M)$. We claim  that the ordered  quivers with potential $(Q(T_i),W(T_i,\epsilon_i))$ are right-equivalent up to scale, meaning that there is a nonzero scalar $\lambda\in k^*$ and a right-equivalence
\[(Q(T_1,),W(T_1,\epsilon_1)) \sim (Q(T_2),\lambda W(T_2,\epsilon_2)).\]
In particular it follows that there are equivalences
\[\D(T_1,\epsilon_1)\isom \D(T_2,\epsilon_2)\]
preserving the standard hearts. The remark following Theorem \ref{kell} then shows that the hearts $\A_i$ differ by an autoequivalence of $\D$,  and the result then follows.

When $g\in \MCG(\S,\M)$ the claim  is obvious since the quiver with potential associated to a signed triangulation depends only on the combinatorial structure of the triangulation. For general $g\in \MCG^\pm(\S,\M)$ the claim follows from  the statement that the quiver with potential $(Q(T),W(T,\epsilon))$ is independent of the signing $\epsilon$, up to scaling and right-equivalence.  A proof of this  statement will  appear in \cite{forth} (see \cite[Prop. 10.4]{LF3}). When $\S$ has non-empty boundary, no scaling is necessary, and the claim follows from \cite[Prop. 10.2]{LF3}.
\end{pf}



\subsection{Grothendieck group}
Let  $(\S,\M)$ be an amenable marked bordered surface. Recall from Section \ref{edgelattice}  the definition of the edge lattice  $\Gamma(T)$ associated to an ideal triangulation  of $(\S,\M)$.

\begin{lemma}
\label{hope}
Let $(T,\epsilon)$ be a signed triangulation of $(\S,\M)$. Then there is an isomorphism of abelian groups
\[\lambda\colon \Gamma(T)\to K(\D(T,\epsilon)),\]
such that for each  edge $e$, the basis element $\{e\}$ is mapped to the class  of the corresponding simple object $S_e$. This map takes  the  form $\langle-,-\rangle$ on $\Gamma(T)$ to the Euler form on $K(\D(T,\epsilon))$.  
\end{lemma}

\begin{pf}
For any reduced quiver with potential $(Q,W)$, the Grothendieck group of the category $\D(Q,W)$ is identified with that of its canonical heart $\A=\A(Q,W)$, and is therefore the free abelian group on the vertices of $Q$. The \CY condition implies that the Euler form  is given by skew-symmetrising the adjacency matrix of the quiver. The result is then immediate from \eqref{grass}. 
\end{pf}

Suppose that $(T_i,\epsilon_i)$ are signed triangulations of $(\S,\M)$ differing by a flip in an edge $e$.  We then use the natural bijection between the  edges of $T_1$ and $T_2$ to identify these two sets.  Theorem \ref{fl} gives equivalences \[\Phi_{\pm}\colon \D(T_1,\epsilon_1)\isom \D(T_2,\epsilon_2),\] which induce isomorphisms $\phi_{\pm}$ on the Grothendieck groups. We have the following explicit formulae for these maps.

\begin{lemma}
\label{lo}
Define maps $F_\pm$ by the  commutative diagram
 \begin{equation*}\xymatrix@C=1.5em{ \Gamma(T_1) \ar[d]_{\lambda_1} \ar[rr]^{F_\pm} && \Gamma(T_2)\ar[d]^{\lambda_2} \\
K(\D(T_1,\epsilon_1)) \ar[rr]^{\operatorname{\phi_\pm}} &&K(\D(T_2,\epsilon_2)) }\end{equation*}
where the $\lambda_i$ are the maps of Lemma \ref{hope}. Then for all  edges $f$  we have
 \begin{equation}
\label{cc}
F_+(\{f\})=\{f\}+n(e,f) \, \{e\},\quad F_-(\{f\})=\{f\}+n(f,e) \, \{e\},\end{equation}
 where  $n(-,-)$ is computed in  the  triangulation $T_2$, and we set $n(e,e)=-2$.
\end{lemma}

\begin{pf}
According to Theorem \ref{fl}, the simple objects of the canonical heart $\A(T_1,\epsilon_1)$ are mapped by the functor $\Phi_{\pm}$ to the simple objects of the tilted heart $\mu_{S_e}^{\pm}(\A(T_2,\epsilon_2))$. The simple  objects of  the mutated heart $\mu_{S_e}^{+}(\A(T_2,\epsilon_2))$ were listed in the proof of Proposition \ref{rpt}. The first formula of \eqref{cc} then follows because  the extension groups between the simple objects in $\A(T_2,\epsilon_2)$ are based by the arrows in the quiver $Q(T_2)$.

To prove the second formula in \eqref{cc} note  that by \eqref{grass} it  differs from the first by a reflection in the element $\{e\}$ with respect to the form $\langle -,-\rangle$. By Lemma \ref{hope} this corresponds under the isometry $\lambda_2$ to the action of the twist functor $\Tw_{S_e}$ on $K(\D(T_2,\epsilon_2))$. The result therefore follows from Proposition \ref{rpt}(a).
\end{pf}

We have a similar result for pops. Suppose that $(T,\epsilon_i)$ are signed triangulations of $(\S,\M)$ differing by a pop at a puncture $p\in \PP$.   Theorem  \ref{pop} gives an equivalence \[\Psi\colon \D(T,\epsilon_1)\isom \D(T,\epsilon_2),\] which induces an isomorphism $\psi$ on the Grothendieck groups. 
\begin{lemma}
\label{popk}
Define a map $F$ by the  commutative diagram
 \begin{equation*}\xymatrix@C=1.5em{ \Gamma(T) \ar[d]_{\lambda_1} \ar[rr]^{F} && \Gamma(T)\ar[d]^{\lambda_2} \\
K(\D(T,\epsilon_1)) \ar[rr]^{\operatorname{\psi}} &&K(\D(T,\epsilon_2)) }\end{equation*}
where  the $\lambda_i$  are the maps of Lemma \ref{hope}.
Then the map $F$ exchanges the two elements $\{e\}$ and $\{f\}$  corresponding to the edges of the self-folded triangle containing $p$, and fixes all other elements of this basis. 
\end{lemma}

\begin{pf}
Immediate from Theorem \ref{pop}.
\end{pf}


\subsection{An unpleasant Lemma}

Let $(\S,\M)$ be an amenable marked bordered surface. Suppose that
 $(T_i,\epsilon)$ are two signed triangulations related by  a flip in an edge $e$.  The purpose of this section is to write the maps $F_\pm$  of Lemma \ref{lo} in terms of the original basis elements $[e]$ of the lattices $\Gamma(T_i)$. As before, we use the natural bijection between the edges of $T_1$ and $T_2$ to identify these two sets. 
 
\begin{lemma}
\label{plop}
The maps $F_\pm$  of Lemma \ref{lo} satisfy
\begin{equation}
\label{aa}
F_+([f])=[f]+ c(e,f) \, [e],\quad F_-([f])=[f]+ c(f,e) \, [e],\end{equation}  where $c(-,-)$ is computed in the triangulation $T_2$, and we set $c(e,e)=-2$. \end{lemma}

\begin{pf}
As in the proof of Lemma \ref{lo}, it is enough to check the result for $F_+$ since the two formulae differ by a reflection in the element $[e]=\{e\}$.
  Recall  from Section \ref{edgelattice} the map $\kappa$, which sends an edge to itself, unless it is self-folded, in which case it sends it to the corresponding encircling edge. We have the relations
\begin{equation}
\label{bb}\{f\}= [f] + \rho_i(f) \,[\kappa_i(f)], \quad i=1,2, \end{equation}
where $\rho_i(f)=0$ unless edge $f$ is self-folded in the triangulation $T_i$, in which case it is equal to 1. 
 By definition, the edge $e$ is not self-folded in either triangulation, so $\rho_i(e)=0$ and $\kappa_i(e)=e$.

We note two basic facts which we will use in the proof. Firstly, if $f$ is self-folded in $T_1$ then $f$ fails to be self-folded in
$T_2$ if and only if $e=\kappa_1(f)$. Secondly, if $f$ is not self-folded in $T_1$ then it  is self-folded in $T_2$ precisely if $e$ and $f$ meet in $T_1$ at a puncture of valency 2. 

The formula of the statement certainly defines some isomorphism; we must just check that it agrees with the formula of Lemma \ref{lo}.
 Substituting \eqref{bb} into \eqref{aa} gives
\[F_+(\{f\})=[f] + \rho_1(f) [\kappa_1(f)] + (c(e,f)+\rho_1(f) c(e,\kappa_1(f)))\,[e].\]
\[=\{f\}-\rho_2(f) [\kappa_2(f)] + \rho_1(f) [\kappa_1(f)] + (c(e,f)+\rho_1(f) c(e,\kappa_1(f)))\,\{e\}.\]
We now claim that
\[-\rho_2(f) [\kappa_2(f)] + \rho_1(f) [\kappa_1(f)] =(\rho_1(f) \delta_{e,\kappa_1(f)}-\rho_2(f) \delta_{e,\kappa_2(f)} )\,\{e\}.\]
To see this, note that both sides are zero unless $f$ is self-folded in one of the $T_i$. If $f$ is self-folded in both, then $e$ is not equal to $\kappa_1(f)$ or $\kappa_2(f)$, and since $\kappa_1(f)=\kappa_2(f)$,  both sides are still zero. If $f$ is self-folded in $T_1$ but not in $T_2$ then necessarily $e=\kappa_1(f)$, and both sides return  $\{e\}$. Similarly, if $f$ is self-folded in $T_2$ but not in $T_1$ then $e=\kappa_2(f)$ and  both sides return  $-\{e\}$.

Thus it remains to show that\[n(e,f) = c(e,f)+\rho_1(f) c(e,\kappa_1(f))+ \rho_1(f) \delta_{e,\kappa_1(f)}-\rho_2(f) \delta_{e,\kappa_2(f)},\]
where $c(-,-)$ and $n(-,-)$ are always computed in the triangulation $T_2$.
Write $m(e,f)$ for the expression on the right. 
If $e=f$ then since $e$ is not self-folded in either triangulation we have $m(e,e)=c(e,e) =n(e,e)=-2$. Thus we assume that $e\neq f$. We proceed by a case-by-case analysis according to whether $f$ is self-folded in each of the two triangulations $T_i$.

Case (a). $\rho_1(f)=0$, $\rho_2(f)=0$. In this case we have $n(e,f)=c(e,f)=m(e,f)$.

Case (b).  $\rho_1(f)=0$, $\rho_2(f)=1$. Then  $e=\kappa_2(f)$ so $n(e,f)=0$ and $m(e,f)=1-1=0$.

Case (c). {$\rho_1(f)=1$, $\rho_2(f)=0$}. Then $e$ and $f$ must meet at a vertex of valency $2$ in $T_2$, so $c(e,f)=c(f,e)=1$ and  $n(e,f)=0$. But then  $\kappa_1(f)=e$, so $m(e,f)=1-2+1=0$. 

Case (d). {$\rho_1(f)=1$, $\rho_2(f)=1$}.  Then $\kappa_1(f)=\kappa_2(f)$ and $e\neq \kappa_2(f)$, so using Remark \ref{remy}(b) we have  $m(e,f)=c(e,\kappa_2(f))=n(e,f)$.
\end{pf}





\section{From differentials to stability conditions}
In this last part of the paper we shall prove our main theorems, by combining the geometry of Sections 2--6 with the algebra and combinatorics of Sections 7--9. In this first section we  explain the basic link between quadratic differentials and stability conditions, following the ideas of  Gaiotto, Moore and Neitzke \cite[Section 6]{GMN2}.

 Our starting point is the observation that a complete and saddle-free differential $\phi\in \Quad(\S,\M)$ determines an ideal triangulation $T(\phi)$ of  the surface $(\S,\M)$ up to the action of the mapping class group $\MCG(\S,\M)$.  We then go on to study how this triangulation changes as we cross between different connected components of the open subset of saddle-free differentials.

\subsection{WKB triangulation}
\label{wkb}

Let $(\S,\M)$ be a marked bordered surface.
Take a  complete and  saddle-free  GMN differential $\phi$ on a Riemann surface $S$ which defines a point of the space $\Quad(\S,\M)_0$. The basic link with the combinatorics of ideal triangulations is the following. 

\begin{lemma}
Taking one generic trajectory  from each horizontal strip of $\phi$ defines an ideal triangulation $T(\phi)$ of the surface $(\S,\M)$, well-defined up to the action of $\MCG(\S,\M)$. \end{lemma}
 
\begin{pf}
Let us identify $(\S,\M)$ with  the marked bordered surface associated to $(S,\phi)$. This identification   is unique up to the action of the group of orientation-preserving diffeomorphisms of $(\S,\M)$. In each horizontal strip $h$ for $\phi$ choose  a corresponding generic trajectory $g_h$.
Note that if $g_h$   tends to a pole $p$ of order $m+2$, then it approaches $p$ along one of the $m$ distinguished tangent vectors at $p$. It therefore defines a path $\delta_h$ in the surface $\S$ connecting two (not necessarily distinct)  points of $\M$, which we  denote $\delta_h$.  

  The  different $\delta_h$  are clearly non-intersecting in their interiors, and by Lemma \ref{bassist} there are the correct number $n$ of them. The fact that they are arcs corresponds to the statement that the original separating trajectories $g_h$ are not contractible relative to their endpoints through paths with interiors in $S\setminus \Crit_\infty(\phi)$. This follows from the fact that they are minimal geodesics. \end{pf}

The triangulation $T(\phi)$ is called the WKB triangulation in \cite{GMN2}. By definition there is a bijection $e\mapsto h_e$ between the edges of  $T(\phi)$ and the horizontal strips of $\phi$.

\begin{lemma}
\label{iv}
Under the bijection $e\mapsto h_e$  an edge $e$ of $T(\phi)$ is self-folded precisely if the corresponding horizontal strip $h_e$ is degenerate.
\end{lemma}

\begin{pf}This is clear from Figure \ref{Fig:Degenerate}: the two zeroes in the boundary of a non-degenerate strip $h_e$ are distinct, so there are four neighbouring strips, whose corresponding edges form the two triangles containing $e$; in the case of a degenerate strip $h_f$ there is a unique neighbouring strip, necessarily non-degenerate, corresponding to the unique encirling edge $e$ of the self-folded triangle containing the self-folded edge $f$. 
\end{pf}

 Note that for any puncture $p\in \PP$ the residue $\res_p(\phi)$  is not real, since a double pole with a real residue is contained in a degenerate ring domain, whose boundary consists of  saddle connections. Suppose that we fix a signing of $\phi$ as in Section \ref{tofin}; this consists of a choice of sign for the residue $\res_p(\phi)$ at each puncture $p\in \bP$. The WKB triangulation $T=T(\phi)$    then also has a naturally defined signing $\epsilon=\epsilon(\phi)$: for a puncture $p\in \bP$ we define  $\epsilon(p)\in\{\pm 1\}$ by the condition \begin{equation}
\label{condn}
\epsilon(p) \cdot  \res_p(\phi)\in \h,\end{equation}
where $\h\subset \C$ is the upper half-plane. We  refer to $(T(\phi),\epsilon(\phi))$ as the \emph{signed WKB triangulation} of the signed differential $\phi$.


\subsection{Hat-homology and the edge lattice}
\label{sing}

 Let $(\S,\M)$ be a marked bordered surface, and take a complete and saddle-free differential $\phi\in \Quad(\S,\M)$.
 If $e$ is an edge of the WKB triangulation $T=T(\phi)$ we denote by $\alpha_e=\alpha_{h_e}$ the standard saddle class of the corresponding horizontal strip $h_e$. The edge lattice $\Gamma(T)$  introduced in Section \ref{edgelattice} then has the following geometric interpretation.

\begin{lemma}
\label{move}
 There is an isomorphism of abelian groups\[\mu\colon \Gamma(T)\to \hs,\]
such that for each edge $e\in T$, 
  the basis element $[e]$ is mapped   to the standard saddle class $\alpha_e$ of the corresponding  horizontal strip $h_e$. This map takes the form $\langle-,-\rangle$ to the intersection form.
\end{lemma}

\begin{pf}
The map $\mu$ is an isomorphism by Lemma \ref{bassist}, since it takes a basis to a basis.
We must just check the relation \[\alpha_e\cdot  \alpha_f =c(f,e)-c(e,f)\]
for all edges $e$ and $f$ of the triangulation $T$.
Examining Figure \ref{dualgraph2} it is clear that  $\alpha_e$ meets $\alpha_f$ precisely if $e$ and $f$ are two sides of the same face of $T$. The intersection then occurs at the unique point of the spectral cover lying over the zero of the differential contained in this face.

\begin{figure}[ht]
\begin{center}
\includegraphics[scale=0.6]{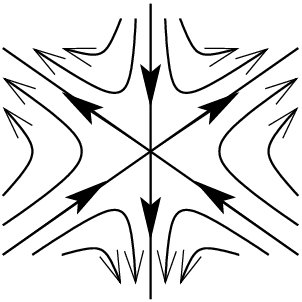}
\end{center}
\caption{The oriented foliation on the spectral cover above a simple zero.\label{Fig:SpectralFoliation}}
\end{figure}

In the case when $e$ and $f$ are not self-folded, a glance at Figure \ref{Fig:SpectralFoliation} shows that the intersection of two such cycles is $\pm 1$ depending on whether $e$ and $f$ occur in clockwise or anticlockwise order.
If  $f$ is  self-folded  and $e$ is the encircling edge of the corresponding self-folded triangle, then $\alpha_e\cdot\alpha_f=0$, since $\alpha_e$ meets $\alpha_f$ twice with opposite signs. Comparing with  \eqref{deg} we see that the claimed relation  holds also in this case.
\end{pf}

In Section \ref{edgelattice} we also considered a modified  basis $\{e\}$ of the edge lattice $\Gamma(T)$, indexed by the edges of $T$. Let us  define classes
\[ \gamma_e=\mu(\{e\})\in \hs,\]
 where $\mu$ is the map of Lemma \ref{move}. We can now give some geometric justification for these classes $\gamma_e$, and hence also the  basis $\{e\}$.

As explained in the proof of Lemma \ref{iv},  any degenerate horizontal strip $h_f$ is enclosed by a non-degenerate strip $h_e$, and taking a generic trajectory from each strip  then gives a self-folded triangle, with self-folded edge $f$ and enclosing edge $e$. Note that   the standard saddle connection in the strip $h_f$ is closed and lifts to a singular curve in $\hS$ which is a bouquet of two circles.
By definition of the basis $\{e\}\in \Gamma(T)$, we have
\[(\gamma_e,\gamma_f)=(\alpha_e,\alpha_e+\alpha_f).\]
These classes are both represented by simple closed curves in $\hS$ obtained by lifting the paths  illustrated by dotted arcs in the two sides of Figure \ref{Fig:Pop}.
   

\subsection{Flips and pops}
\label{wall1}
We can lift the stratification of Section \ref{strat} to the {\'e}tale cover of complete, signed differentials\[\Quad^\pm(\S,\M)_0\to \Quad(\S,\M)_0.\] 
We call the connected components of $B_0$ chambers. The  signed WKB triangulation is constant in each chamber.

The points of the locally-closed subset $F_2=B_2\setminus B_0$ consist of complete,  signed differentials with a single saddle trajectory $\gamma$. We think of the  connected components of $F_2$ as  walls.
 Let us now consider the behaviour of the signed WKB triangulation as we cross such a wall.

 Suppose then that $\phi_0\in F_2\subset \Quad^\pm(\S,\M)_0$ is a complete signed GMN differential lying on a wall, with a unique saddle trajectory $\gamma$.
 By Proposition \ref{more}  we can find $r>0$ such that for all $0<t\leq  r$ the signed differentials \[\phi_\pm(t)=e^{\pm i t} \cdot \phi_0\] are saddle-free and complete.   Consider the WKB triangulations  $T_\pm=\T(\phi_\pm(r))$ with their signings $\epsilon_\pm$. Note that the wall has a natural orientation to it: we make the convention that as we cross from $\phi_-$ to $\phi_+$ the period $Z_\phi(\Hat{\gamma})$ moves in a clockwise direction around the origin.

 There is an isomorphism $F\colon \Gamma(T_-)\to \Gamma(T_+)$ defined by the following commutative diagram
 \begin{equation}\label{flipgeom}\xymatrix@C=1.5em{ \Gamma(T_-) \ar[d]_{\mu_-} \ar[rr]^{F} && \Gamma(T_+)\ar[d]^{\mu_+} \\
\hsinput{-} \ar[rr]^{\operatorname{GM}} &&\hsinput{+}   }\end{equation}
where the bottom arrow is given by  the Gauss-Manin connection. 

\begin{prop} \label{flipandpop}
Take a complete signed GMN differential $\phi_0\in F_2$ with a unique saddle trajectory $\gamma$. There are two possible cases.
\begin{itemize}
\item[(a)] The ends of the saddle trajectory $\gamma$ are distinct. Then the  signed triangulations $(T_\pm,\epsilon_\pm)$ are related by a flip in a non-self-folded edge $e$. Identifying the edges of these triangulations in the standard way, the map $F$ is given by
\[F([f])=[f]+c(e,f) \, [e],\]
for all edges $f$, where $c(-,-)$ is computed in the triangulation $T_+$, and we set $c(e,e)=-2$.
\smallskip

\item[(b)] The saddle trajectory $\gamma$ is closed and forms the boundary of a degenerate ring domain centered on  a double pole $p$ with  real residue. The triangulations $T_-=T_+$ are equal, and the pole $p$ is the centre of a self-folded triangle with encircling and self-folded edges $e$ and $f$ respectively. The signed triangulations  $(T_\pm,\epsilon_\pm)$  are related by a pop at $p$, and the map $F$ exchanges the basis elements $\{e\}$ and $\{f\}$, leaving all other elements of the basis fixed. \end{itemize}
\end{prop}

\begin{proof}
Case (a)  is illustrated in Figure \ref{Fig:GenericFlip} (the four poles represented by the black dots need not be distinct on the surface, however). 
\begin{figure}[ht]
\begin{center}
\includegraphics[scale=0.5]{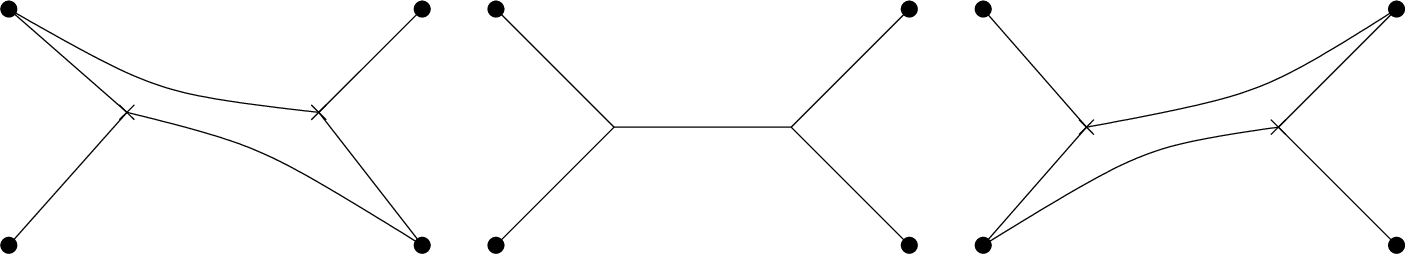}
\end{center}
\caption{The separating trajectories on either side of a flip wall.\label{Fig:GenericFlip}}
\end{figure}
The central picture  represents $\phi_0$ with a single saddle trajectory appearing in the boundary of two neighbouring horizontal strips or half-planes. The wall-crossing is effected by rotating $\phi_0$. Thus to find out what happens to the saddle trajectory, we just need to consider trajectories of small nonzero phase  in the two horizontal strips or half-planes. The result is as illustrated on the two sides of the figure. The associated triangulations $T_\pm$ are related by a flip in a non-self-folded edge $e$, exactly as shown in Figure \ref{flippy}.

Identifying the edges of the two  triangulations $T_\pm$ via the obvious bijection, we see from the picture that the Gauss-Manin connection satisfies
\[\operatorname{GM}(\alpha_f)= \alpha_f + c(e,f) \,\alpha_e\]
for all edges $f$, where $c(-,-)$ is computed in the triangulation $T_+$, and we set $c(e,e)=-2$.

\begin{figure}[ht]
\begin{center}
\includegraphics[scale=0.3]{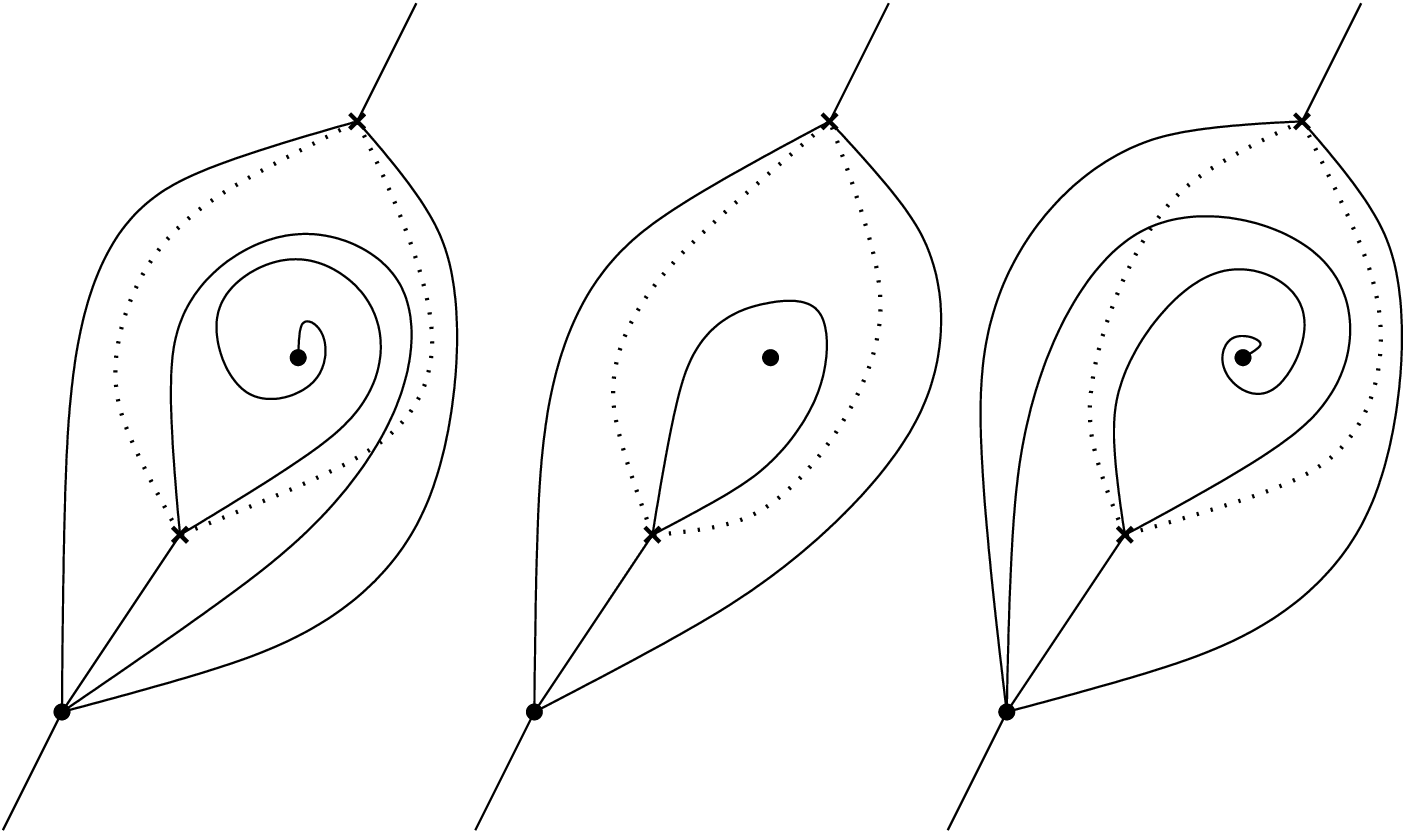}
\end{center}
\caption{The separating trajectories on either side of a pop wall. The hat-homology classes $\gamma_e$ are represented by the lifts of the dotted arcs.\label{Fig:Pop}}
\end{figure}

For (b) note that if $\gamma$ is closed then it is necessarily the boundary of a ring domain, which must be degenerate because there is only one saddle trajectory.  This case is illustrated in Figure \ref{Fig:Pop}. The central picture again represents $\phi_0$ with its degenerate ring domain encased in a horizontal strip. The wall-crossing is effected by rotating the differential, so to find the separating trajectories on either side of the wall it is enough to consider trajectories of small nonzero phase. 
The result on either side of the wall is a degenerate horizontal strip $h_f$, encased in a non-degenerate horizontal strip $h_e$.

The WKB triangulations $T_\pm$ are the same, with a self-folded triangle with self-folded edge $f$ and encircling edge $e$.
The hat-homology class $\Hat{\gamma}=\alpha_f$ is equal to the residue class $\beta_p$, 
where $p$ is the double pole at the centre of the degenerate ring  domain. Since $Z_\phi(\Hat{\gamma})$ crosses the real axis as $\phi$ crosses the wall, the signing $\epsilon(p)$ given by \eqref{condn} changes, and the two signed WKB triangulations on either side of the wall are related by a pop.
The two classes $\{\gamma_e,\gamma_f\}=\{\alpha_e,\alpha_e+\alpha_f\}$ are the same on both sides of the wall, but their labelling by the edges $e,f$ is exchanged (see  the dotted arcs in  Figure \ref{Fig:Pop}). It follows that  $F$ exchanges the two basis elements $\{e\}$ and $\{f\}$ as claimed.\end{proof}


\subsection{Stability conditions from saddle-free differentials}

Let us now assume that our marked bordered surface $(\S,\M)$ is amenable, and  take a saddle-free, complete, signed differential
\[\phi\in B_0\subset \Quad^\pm(\S,\M)_0.\]
 Let $(T,\epsilon)$ denote the  signed WKB triangulation of $\phi$, and consider the   category $\D(T,\epsilon)$ with its canonical heart $\A(T,\epsilon)$. The simple objects $S_e\in \A(T,\epsilon)$  are naturally indexed by the edges of the  triangulation $T$, and so too are the classes $\gamma_e$.

\begin{lemma}
\label{crucaal}
There is  an isomorphism of abelian groups\begin{equation*}\nu\colon K(\D(T,\epsilon))\to \hs,\end{equation*}
taking the class of a simple object $S_e$ to the corresponding class $\gamma_e$, and taking the Euler form to the intersection form.
\end{lemma}

\begin{pf}
This is immediate by combining  Lemma \ref{hope} and Lemma \ref{move}.
\end{pf}

The following result gives the basic link between quadratic differentials and stability conditions.

\begin{lemma}
\label{un}
 There is a unique stability condition $\sigma(\phi)\in \Stab \D(T,\epsilon)$
 whose heart is the standard heart $\A(T,\epsilon)\subset \D(T,\epsilon)$, and whose central charge \[Z\colon K(\D(T,\epsilon))\to\C\] corresponds to the period of $\phi$ under the  isomorphism of Lemma \ref{crucaal}.
\end{lemma}

\begin{pf}
Define $Z$ via the isomorphism $\nu$  of Lemma  \ref{crucaal}.
Then for each  edge $e$  the corresponding central charge
$Z(S_e)=Z_\phi(\gamma_e)$ lies in the upper half-plane. Indeed, by the definition of the basis $\{e\}$  in Section \ref{edgelattice}, the  classes $\gamma_e=\mu(\{e\})$ are positive linear combinations of the classes $\alpha_e=\mu([e])$, whose periods lie in the upper half-plane by definition. 
 Since the standard heart $\A(\T,\epsilon)$ is of finite length, this is enough to give a stability condition. 
\end{pf}


\subsection{Wall-crossing}

We now describe how the stability conditions $\sigma(\phi)$ defined in the last section behave as the differential $\phi$  crosses walls in $\Quad^\pm(\S,\M)_0$ of the sort considered   in Section \ref{wall1}.
Consider a complete, signed GMN differential \[\phi_0\in F_2\subset \Quad^\pm(\S,\M)_0\]  with a unique saddle trajectory $\gamma$.
Take $r>0$ such that for all $0<t\leq  r$ the signed differentials \[\phi_\pm(t)=e^{\pm i t} \cdot \phi_0\] are saddle-free and complete.
Consider the corresponding signed WKB triangulations  $(T_\pm,\epsilon_\pm)=(T(\phi_\pm(r),\epsilon(\phi_\pm(r))$ and their associated   categories \[\D_{\pm}=\D(T_\pm,\epsilon_\pm)\] with their standard hearts $\A_\pm =\A(T_\pm,\epsilon_\pm)$. For $0<t<r$ we set \[\sigma_\pm(t)=\sigma( \phi_\pm(t))\in \Stab(\D_\pm).\]
The following result shows that these different stability conditions glue together in the appropriate way.

\begin{prop}
\label{skateboard}
There is a canonical equivalence
 $\Psi\colon \D_- \isom\D_+$ with the following two properties
\begin{itemize}
\item[(a)] the  diagram of isomorphisms
\begin{equation*}\xymatrix@C=1.5em{ K(\D_-) \ar[d]_{\nu_-} \ar[rr]^{\Psi} && K(\D_+)\ar[d]^{\nu_+} \\
\hsinput{-} \ar[rr]^{\operatorname{GM}} &&\hsinput{+}   }\end{equation*}
commutes, where the bottom arrow is given by the Gauss-Manin connection, and the  vertical arrows are  the isomorphisms of Lemma \ref{crucaal};\smallskip

\item[(b)] the stability conditions $\Psi(\sigma_-(t))$ and $\sigma_+(t)$ on $\D_+$ become arbitrarily close as $t\to 0$.
\end{itemize}
\end{prop}

\begin{pf}
According to Proposition \ref{flipandpop} there are two cases, the flip and the pop. 

In the first case, 
the signed triangulations $(T_\pm,\epsilon_\pm)$ are related by a flip in an edge $e$, and we take $\Psi$ to be the equivalence $\Phi_+$ of Theorem \ref{fl}. Part (a) then follows by comparing the formulae of Lemma \ref{lo} and Proposition \ref{flipandpop} using  Lemma \ref{plop}.
To prove (b), note that Theorem \ref{fl} shows that the heart of the stability condition $\Psi(\sigma_-(t))$ is the tilted heart $\mu_{S_e}^+( \A_+)$. By Lemma \ref{wolf} the regions in $\Stab (\D_+)$ consisting of stability conditions with  hearts $\A_+$ and $\mu_{S_e}^+( \A_+)$ are glued together along a common boundary component to make a larger region on which the period map is still injective. Since part (a) shows that the central charges of the two given stability conditions approach one another, the result follows.

In the  case of the pop, the signed triangulations $(T_\pm,\epsilon_\pm)$ differ by a pop, and we take $\Psi$ to be equivalence of Theorem \ref{pop}. Part (a) then follows by comparing the formulae of  Lemma \ref{popk} and Proposition \ref{flipandpop}. To prove (b) note that since $\Psi$ preserves the canonical hearts, all the stability conditions $\sigma_{\pm}(t)$ have the same  heart, and since their central charges approach each other, they become arbitrarily close.  
\end{pf}






\section{Proofs of the main results}
\label{biggy}

In this section we prove our main results.  Throughout $(\S,\M)$ is a fixed marked bordered surface. For the first five sections  we shall assume that $(\S,\M)$ is amenable.  In  Section \ref{therest} we shall examine what can be said without this assumption. 



\subsection{General set-up}
\label{prep}

 Let us fix a free abelian group $\Gamma$ of rank $n$. We consider the space of framed differentials $\Quad^\Gamma(\S,\M)$. A point of this space corresponds to a GMN differential $\phi$ on a Riemann surface $S$, equipped  with a framing of the extended hat-homology group \[\theta\colon \Gamma\isom \hse\] as in Section \ref{epm}. By abuse of notation, we will simply write $\phi\in \Quad^\Gamma(\S,\M)$.
 
 Let us fix a base-point \[\phi_0\in \Quad^{\Gamma}(\S,\M),\]
which we may as well assume  is complete, saddle-free and generic.
Recall that  the space of framed differentials on $(\S,\M)$ is not usually connected,  so we define \[\Quad^{\Gamma}_*(\S,\M)\subset \Quad^{\Gamma}(\S,\M)\] to be the connected component  containing $\phi_0$.

Let us  choose a signing for the differential $\phi_0$, as in Section \ref{tofin}. We claim that any point in $\Quad^\Gamma_*(\S,\M)$ then also has a canonical signing. To see this, note that to specify a signing of a differential $\phi$ is to specify a choice of sign for the residue class $\beta_p\in \hs$ at each double pole $p$ of $\phi$. Given a framing of such a differential, the classes $\beta_p$ correspond  to fixed classes in $\Gamma$. It follows that if we choose a sign for the $\beta_p$ at the base-point $\phi_0$, then this sign propagates throughout $\Quad^\Gamma_*(\S,\M)$.

We want to study stability conditions on  the \CY triangulated category $\D(\S,\M)$. More precisely, let \[(T_0,\epsilon_0)=(T(\phi_0),\epsilon(\phi_0))\]
be the signed  WKB triangulation    associated to the signed differential $\phi_0$, and  define $\D=\D(T_0,\epsilon_0)$. 
 We identify the Grothendieck group $K(\D)$ with  $\Gamma$   using the isomorphism
\begin{equation}\label{notlast}\nu_0^{-1}\circ \theta_0\colon \Gamma\to K(\D),\end{equation}
obtained by composing the framing $\theta_0$ with the inverse of the map of Lemma \ref{crucaal}.

The distinguished connected component $\Tilt_\triangle(\D)\subset \Tilt(\D)$ of the tilting graph of $\D$ is the one containing the standard heart $\A(T_0,\epsilon_0)$. As explained in Section \ref{chambers},  there is a corresponding distinguished connected component \[\Stab_\triangle(\D)\subset \Stab(\D)\] of the space of stability conditions on $\D$.
The group of reachable autoequivalences $\All(\D)\subset \Aut(\D)$  preserves this connected component.  We  define \[\Allo(\D)\subset \All(\D)\] to be the subgroup of reachable autoequivalences which act by the identity on $K(\D)$.

\begin{remark}
Later, as a consequence of Corollary \ref{reach}, we will see that 
\begin{itemize}
\item[(a)]the group  $\All(\D)$ is  precisely the group of autoequivalences preserving the connected component $\Stab_\triangle(\D)$,
\smallskip

\item[(b)] the group $\Allo(\D)$ is non-trivial, and in fact  contains all even powers of the shift functor.
\end{itemize}
\end{remark}

The subgroup  $\Nil(\D)\subset \All(\D)$ of negligible autoequivalences is precisely the subgroup of elements acting trivially on $\Stab_\triangle(\D)$. The quotient group
\[\uAll(\D)=\All(\D)/\Nil(\D)\]
therefore acts effectively.    Negligible autoequivalences fix the simple objects of hearts $\A\in \Tilt_\triangle(\D)$, and hence act trivially on $K(\D)$, so we can also form the quotient  
\[\uAllo(\D)=\Allo(\D)/\Nil(\D).\]

In the next five sections we shall prove the following result.

\begin{thm}
\label{alltogethernow}
 There is an isomorphism of complex orbifolds
\[\Quadorb(\S,\M)\isom \Stab_\triangle(\D)/\uAll(\D).\]
\end{thm}

This result implies Theorems  \ref{Thm:Main0} and  \ref{Thm:Main1} from the Introduction, except that it doesn't cover the non-amenable  cases (a)--(d) listed after the statement of Theorem  \ref{Thm:Main1}. These exceptional cases will be discussed in Section \ref{therest} below, and some of them are worked out explicitly in Section \ref{applications}.


\subsection{Construction of the map}

In this section
 we  construct a map from framed quadratic differentials to stability conditions.

\begin{prop}
\label{most}
There is  a holomorphic map of complex manifolds $K$ fitting into a commutative diagram
\begin{equation}\label{diag}\xymatrix@C=1.5em{  \Quad^{\Gamma}_*(\S,\M)\ar[rr]^{K} \ar[dr]_\pi &&\Stab_\triangle(\D)/\uAllo(\D) \ar[dl]^\pi\\
&\Hom_\Z(\Gamma,\C)  }\end{equation}
and which commutes with the $\C$-actions on both sides. \end{prop}

\begin{pf}
The space on the left of the diagram \eqref{diag} is a manifold by Proposition \ref{reff} (note that the  surfaces listed for which this result fails are not amenable). The space on the right is a manifold because the action of the group $\uAllo(\D)$ on the connected component $\Stab_\triangle(\D)$ is free. To see this, note that if an  autoequivalence $\Phi$ fixes a stability condition $\sigma\in \Stab_\triangle(\D)$ and acts trivially on $K(\D)$, then, because the period map $\pi$ is a local isomorphism, it must act trivially on a neighbourhood of $\sigma$, and hence on the whole connected component $\Stab_\triangle(\D)$. But then it is negligible and hence defines the identity in $\uAllo(\D)$.

The action of $\C$ on the right of \eqref{diag} is the standard one of Section \ref{stabsumm}; the element $t\in \C$ acts at the level of central charges by $Z(E)\mapsto e^{i\pi t}\cdot  Z(E)$. After the event it will follow from Corollary \ref{reach} that $2\in \C$ acts trivially, so that this factors via a $\C^*$ action, but we don't know this yet. The action of $\C$ on the left is the pullback of the standard $\C^*$ action rescaling the quadratic differential, via the map $\C\to \C^*$ defined by $t\mapsto e^{2\pi i t}$. This action lifts to framed differentials by continuity; note that $1\in \C$ acts trivially on the underlying quadratic differential, but multiplies the framing isomorphism by $-1$.

The maps $\pi$ in  the diagram \eqref{diag} are both local isomorphisms.
The map on the left  is the standard period map on framed differentials. The map  on the right  sends a stability condition to its central charge, which we consider as a group homomorphism $Z\colon \Gamma\to \C$ by composing it with the isomorphism \eqref{notlast}. 
Let $t\in \C$ act on the space of central charges $\Hom_\Z(\Gamma,\C)$ via the map $Z(E)\mapsto e^{i\pi t}\cdot  Z(E)$ as above. Then  because periods of quadratic differentials are given by integrals of $\sqrt{\phi}$, both the maps $\pi$ are $\C$-equivariant.

As soon as we know that $K$ is continuous it is automatically holomorphic, and in fact, a local isomorphism. The $\C$-equivariance is also automatic, just because each of the two maps $\pi$ is a local isomorphism. 
We first  define a map 
\[K_0\colon B_0\to \Stab_\triangle(\D)/\uAllo(\D).\]
The key point is to then use the stratification \[B_0\subset B_1\subset B_2\subset \dots\subset B_{k}\subset \dots\subset \Quad_*^\Gamma(\S,\M)_0\subset \Quad_*^\Gamma(\S,\M)\]
 and inductively extend the domain of definition of the map across larger strata. 

 Let $\phi$ be a saddle-free complete framed differential defining a point in $\Quad^\Gamma_*(\S,\M)$.  Recall  that the signed WKB triangulation $(T(\phi),\epsilon(\phi))$ is  well-defined up the action of the mapping class group $\MCG(\S,\M)$. Take an equivalence
\[\Psi\colon \D(T(\phi),\epsilon(\phi))\lra \D(\T(\phi_0),\epsilon(\phi_0))\]
with the following two properties
\begin{itemize}
\item[(i)] $\Psi$ maps  the distinguished connected components of  the tilting graphs of $\D(T(\phi),\epsilon(\phi))$  and $\D(T(\phi_0),\epsilon(\phi_0))$ one to the other;\smallskip

\item[(ii)] the map $\psi$  on Grothendieck groups induced by $\Psi$
makes the following diagram commute
\begin{equation}\label{blob}\xymatrix@C=1.5em{ K(\D(T(\phi),\epsilon(\phi))) \ar[d]_{\nu} \ar[rr]^{\psi} && K(\D(T(\phi_0),\epsilon(\phi_0)))\ar[d]^{\nu_0}\\
\hs&&\hsinput{0}\\
&\Gamma \ar[ul]^{\theta} \ar[ur]_{\theta_0}  }\end{equation}
where $\nu$ and $\nu_0$ are the isomorphisms of Lemma \ref{crucaal}, and $\theta$ and $\theta_0$ are the framing isomorphisms.
\end{itemize}

To see that such an equivalence $\Psi$ exists, connect $\phi_0$ to $\phi$ by some path in $\Quad^\Gamma_*(\S,\M)$. Since the subset of differentials with simple poles is locally cut out by complex hyperplanes  we can   assume that this path lies in $\Quad^\Gamma_*(\S,\M)_0$. By  Corollary \ref{TamePath} and the fact that $\Quad^\Gamma_*(\S,\M)_0$ is a covering space of $\Quad(\S,\M)_0$, we can then deform the path so that it lies in  $B_2$ and has only  finitely many points in $F_2$. Applying Proposition \ref{skateboard}  to each of these  points, and taking the composite of the given equivalences gives a suitable equivalence $\Psi$.

Any two such equivalences $\Psi$  differ by a reachable autoequivalence acting trivially on $K(\D)$,
so we obtain a well-defined map $K_0$ by 
 setting \[K_0(\phi)=\Psi(\sigma(\phi))\in\Stab_\triangle(\D)/\uAllo(\D).\] 
The diagram \eqref{diag}  then commutes by definition.

Our task is now to successively lift $K$ to the inverse images of the various strata $B_i$.
Let us assume inductively that $K$ is defined and continuous on the open subset $B_{p-1}$. As remarked above, $K$ is invariant under small rotations, i.e.
\[K(e^{i\pi\theta}\cdot \phi)=e^{i\pi\theta} \cdot K(\phi) \text{ for }0<|\theta|\ll 1,\]
because both maps $\pi$ of \eqref{diag} are local isomorphisms.
Take a point $\phi_0\in F_{p}$. By Proposition \ref{more} there is  some $r>0$ such that
\begin{equation}
\label{delta}0<|t|<r \implies e^{it}\cdot \phi_0\in B_{p-1}.\end{equation}
 By the $\C$-equivariance property, the limits
\[\sigma_{\pm}(\phi_0) = \lim_{t\to 0^+} K(e^{\pm i t} \cdot \phi_0)\]
both exist, and the diagram \eqref{diag} shows that they have the same central charge. The point is to show that they are equal.

When  $p=2$ this involves extending from saddle-free differentials to differentials lying on a single wall, and the result follows from Proposition \ref{skateboard}.
Thus we can assume that $p>2$. Note that  the stability conditions $\sigma_\pm (\phi_0)$ vary continuously  on $F_{p}$, by the fact that their central charges do, and using the remark following Proposition \ref{more}. Thus the question of whether they are equal has a constant answer on each connected component of $F_p$. The result then follows from Proposition \ref{moremore}.

The final step is to extend $K$ across the incomplete locus. Suppose  that $\phi\in \Quad^\Gamma_*(\S,\M)$ has simple poles, and fix $0<\epsilon<\frac{1}{8}$. Using Proposition \ref{leaving} we can find complete, generic differentials $\psi$ arbitrarily close to $\phi$ , and such that
\[|Z_\phi(\gamma)-Z_\psi(\gamma)|<\epsilon\, |Z_\psi(\gamma)|,\]
for all classes $\gamma\in \Gamma$ represented by a non-closed saddle connection in $\psi$. Lemma \ref{done} below then shows that this inequality holds for all classes represented by stable objects in the stability condition  $K(\psi)$. The deformation result Proposition \ref{lastnew} then shows that  $K$ extends uniquely over $\phi$.
\end{pf}


\subsection{Saddle trajectories and stable objects}
\label{sadstab}
In this section and the next we relate saddle trajectories for a generic GMN differential  to the existence of stable objects in the corresponding stability condition. 
Let \[\phi\in \Quad^\Gamma_*(\S,\M)_0\] is a complete, framed differential, and  let $\sigma=K(\phi)$ be the  corresponding stability condition on $\D$,
well-defined up to the action of the group $\Allo(\D)$.
Note that any  saddle connection $\gamma$ for $\phi$ has a well-defined hat-homology class in $\hs$, which we can view as an element of $\Gamma$ using the framing isomorphism. 
Similarly,  every object $E\in \D$ has a well-defined class in  $\Gamma$ using the identification \eqref{notlast}. 

We shall start with the following simple result which was used in the final step of the proof of Proposition \ref{most} above. 

\begin{lemma}
\label{done}
If $\sigma$ has a stable object $E$ of class $\alpha\in \Gamma$ then $\phi$ has a saddle trajectory of class proportional to $\alpha$.
\end{lemma}

\begin{proof}
Rotating we can assume that $E$ is of phase 1. Consider a generic differential $\psi$ close to $\phi$ such that $Z_\psi(\alpha)$ remains real. Proposition \ref{fe}(a) shows that the object $E$ remains stable  in the corresponding stability condition $K(\psi)$. By construction therefore $\psi$ cannot be saddle-free, and  has at least one  saddle trajectory $C$, which by genericity must have class proportional to $\alpha$. Since  the class of a saddle connection has divisibility at most 2,  the length of this saddle is at most  $|Z_{\psi}(2\alpha)|$.
 
 Applying  Theorem \ref{finini},  we can now find  a sequence of such differentials $\psi_i$,  converging to $\phi$,  such that the corresponding curves $C_i$ limit to some curve $C$.  Then $C$  is  a geodesic in $\phi$ which must be a union of saddle trajectories. 
\end{proof}

For each class $\alpha\in \Gamma$  we can consider the moduli space $\cM_\sigma(\alpha)$ of $\sigma$-stable objects in $\D$ which have   class $\alpha\in \Gamma$ and phase in the interval $(0,1]$.
It is necessary to constrain the phase, since otherwise all shifts of a given stable object would have to be parameterized.

\begin{lemma}
The moduli space $\cM_\sigma(\alpha)$ is represented by a quasi-projective scheme.
\end{lemma}

\begin{pf} By rotation, we can assume that $\phi$ is saddle-free and therefore defines some signed triangulation $(T,\epsilon)$. The heart of the stability condition $\sigma$ is then equivalent to the category of finite-dimensional modules for the complete Jacobi algebra of the corresponding quiver with potential. This algebra $J(T,\epsilon)$ is known to be finite-dimensional\footnote{This could also be deduced  from Corollary \ref{reach} below together with  an argument of Nagao \cite[Theorem 5.4]{K}.} \cite[Corollary 12.6]{LF4}. The moduli space $\cM_\sigma(\alpha)$  can therefore be identified with the moduli space of $\theta$-stable representations of $J(T,\epsilon)$  with  fixed dimension vector. The claim then follows from the results  of King \cite{king}.
\end{pf}

Recall the notion of a 0-generic differential from Section \ref{strat}. Note that any such differential is, in particular, complete.
In this and the next section we shall prove the following  precise correspondence, which implies Theorem \ref{greenlees} from the Introduction.

\begin{thm}
\label{th}
Assume that $\phi$ is   0-generic  and take a class $\alpha\in \Gamma$ satisfying  $Z_\phi(\alpha)\in \R$. Then  each connected component of the moduli scheme $\cM_\sigma(\alpha)$ is either a point or a copy of $\PP^1$. Moreover 
\begin{itemize}
\item[(a)] the zero-dimensional components of $\cM_\sigma(\alpha)$ are in bijection with the non-closed saddle trajectories for $\phi$ of class $\alpha$;
\smallskip
\item[(b)] the one-dimensional components of $\cM_\sigma(\alpha)$ are in bijection with the non-degenerate ring domains for $\phi$ of  class $\alpha$.
\end{itemize} \end{thm}

 We first prove the result under an additional assumption; the general case will be dealt with in the next section.

\begin{prop}
\label{halfway}
Take assumptions as in Theorem \ref{th}. Suppose moreover that $\phi$  has at most one saddle trajectory. Then the conclusion of Theorem \ref{th} holds.
\end{prop}

\begin{pf}
Suppose that $\phi$ has a unique saddle trajectory $\gamma$  of some class $\alpha\in \Gamma$. Since the surface $(\S,\M)$ is assumed to be amenable, it is not a closed surface with a single puncture, and it follows that there can be no spiral domains. Thus $\phi\in F_2$ and we are in the situation of Proposition \ref{flipandpop}. 

Assume first that $\gamma$ has distinct ends, as in Proposition \ref{flipandpop}(a). Small rotations of $\phi$ are saddle-free, and $\gamma$ is represented by a standard saddle class. The corresponding stability conditions have a unique  simple object  $S$  of class $\alpha$. This is, in particular, stable and spherical, and all semistable objects of  class proportional to $\alpha$  are of the form $S^{\oplus k}$, and hence strictly semistable for  all $k>1$.

Suppose instead that $\gamma$ is a closed saddle trajectory. This is the situation of  Proposition \ref{flipandpop}(b).  By Proposition \ref{skateboard},  stability conditions on either side of the wall obtained by small rotations of $\sigma$ have the same heart. This implies that $\P(0)=(0)$ and  so there are no semistable objects  with class a multiple of $\alpha$.  
 
For the converse, suppose that the stability condition $\sigma$ has a stable object of class $\alpha$. Then, by construction of the map $K$, the differential $\phi$ cannot be saddle-free. Thus $\phi$ has a saddle trajectory $\gamma$ of some class $\beta$ proportional to  $\alpha$.   Applying  what we have proved in the first part, it follows that  $\sigma$ has at most one   stable object of class $\beta$, and no stable objects of any other class proportional to $\beta$.  We conclude that $\alpha=\beta$ and so $\phi$ has a (necessarily unique) saddle trajectory of class $\alpha$.
\end{pf}


\subsection{Ring-shrinking again}

In this section we complete the proof of Theorem \ref{th}. It will be convenient to denote the differential and corresponding stability condition of the statement by $\phi_+$ and $\sigma_+$ respectively.
Thus we consider a complete, framed, 0-generic  differential
\[\phi_+\in \Quad^\Gamma(\S,\M)_0,\]
and let $\sigma_+=K(\phi_+)$ be the corresponding stability condition on $\D$, 
well-defined up to the action of the group $\Allo(\D)$. We may assume that $\phi_+$  has more than one saddle trajectory, since otherwise we are in the situation of Proposition \ref{halfway}.

The surface $(\S,\M)$ is amenable, hence not a once-punctured torus, so according to Section \ref{jug}  there are two possible cases, labelled (J1) and (J2). Shrinking the unique ring domain $A$ as in Section \ref{jug} gives a smooth path in $\Quad^\Gamma(\S,\M)_0$ ending at a non-generic point \[\phi\in \Quad^\Gamma(\S,\M)_0\]
with either 2 or 3 saddle trajectories $\gamma_i$.  Label the  saddle trajectories $\gamma_i$   exactly as in Section \ref{jug}, and write $\alpha_i\in \Gamma$ for the corresponding hat-homology classes. We recall that there  is a linear relation \[\alpha_3=\alpha_1+\alpha_2.\]

Our stategy will be to first understand the stable objects of  phase 1 in the stability condition $\sigma=K(\phi)$ by applying Proposition \ref{halfway} to nearby points on the other side of the wall
\[ \Im Z_{\phi_-}(\alpha_1)/Z_{\phi_-}(\alpha_2)=0.\]
We will then follow this information back through the ring-shrinking operation. 

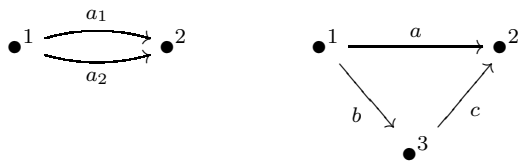
\begin{figure}[ht]
\begin{center}
\begin{equation*}
\xymatrix@C=.4em{\bullet^1 \ar^{a_1}@/^/[rrrr] \ar_{a_2}@/_/[rrrr]&&&& \bullet^2 &&&&   \bullet^1 \ar_b[drr]\ar^a[rrrr] &&&& \bullet^2       \\ 
&&&&&&&& && \bullet^3 \ar[urr]_c }
\end{equation*}
\caption{The quivers relevant to the two cases (J1) and (J2). \label{quivers}}
\end{center}
\end{figure}

Let  $\A=\P(1)$ denote the abelian category of semistable objects of phase 1 in the stability conditon  $\sigma=K(\phi)$. This category has finite-length, so we can model it by the category of representations of the Jacobi algebra of a quiver with potential $(Q,W)$. The relevant quivers $Q$ in the two cases (J1) and (J2) are as shown in Figure \ref{quivers}; in both cases the potential is necessarily zero.

\begin{prop}
The category $\A$  is equivalent to the category of finite-dimensional representations of the corresponding quiver $Q$.
\end{prop}

\begin{pf}
 For definiteness we consider the more difficult  case (J2).  
 
Let $n_{ij}$ denote the number of arrows in the quiver $Q$ connecting vertex $i$ to vertex $j$. To prove the Lemma we must show

\begin{itemize}
\item[(a)]the category $\A$ has exactly 3 stable objects $S_i$;
\smallskip
\item[(b)]  these  objects have classes $[S_i]=\alpha_i$ respectively;
\smallskip
\item[(c)] For all $i,j$ we have $\dim_k \Ext_\A^1(S_i,S_j)=n_{ij}$.\end{itemize}

 It is easy to see that after contracting the ring domain in Figure \ref{Fig:Hatprop}  the intersection multiplicities $\alpha_i\cdot \alpha_j$ coincide with the expressions $n_{ji}-n_{ij}$.  So for the last part it will be enough to show that each object $S_i$ is spherical, and for each pair  $i\neq j$,  we have either
 \[\Ext^1_\A(S_i,S_j)=0 \text{ or }\Ext^1_\A(S_j,S_i)=0.\]

Take  a class $\beta\in \Gamma$  and consider  the wall-and-chamber decomposition of $\Stab(\D)$  with respect to the class $\beta$. Take a chamber  containing $\sigma$ in its closure, and containing points $\sigma_-=K(\phi_-)$ which satisfy
\[ \Im Z_{\phi_-}(\alpha_1)/Z_{\phi_-}(\alpha_2)<0, \quad \Im Z_{\phi_-}(\beta)=0.\]
Let us choose such a point $\sigma_-$ in this chamber, and assume further that the corresponding differential $\phi_-$ is generic, and that it lies in the open subset $U$ of Proposition \ref{jugprop}. Of course, our choice of  $\phi_-=\phi_-(\beta)$ will depend on the class $\beta$ we started with.

The genericity condition implies that all saddle trajectories for $\phi_-$ have classes which are multiples of $\beta$. Proposition \ref{jugprop} then implies that any such saddle trajectory is one of the $\gamma_i$. Since the classes $\alpha_i$ are pairwise non-proportional, it follows in particular that $\phi_-$  has at most one saddle trajectory. Thus  Proposition \ref{halfway} applies to $\phi_-$, and shows that $\phi_-$ has a saddle trajectory of class $\beta$ precisely if $\sigma_-$ has a stable object of class $\beta$.

In the case when $\beta=\alpha_i$ for some $i$, Proposition \ref{halfway} implies that $\sigma_-=\sigma_-(\alpha_i)$ has a unique stable object $S_i$ of class $\beta$, which is moreover spherical. By Proposition \ref{chamberprop}, the  object $S_i$  is then  semistable in $\sigma$, and hence lies in the category $\A=\P(1)$. 

Conversely, suppose that an object $E\in \A$ is stable in $\sigma$ of phase 1.  Proposition \ref{fe}(a) shows that  $E$ is  stable in all nearby stability conditions, and in particular we can assume that this is the case for   $\sigma_-=\sigma_-(\beta)$. Then $\phi_-$ must have a saddle trajectory of class $\beta$, and by Proposition \ref{jugprop} it follows that $\beta=\alpha_i$ for some $i$. Since $S_i$ was the unique stable object in $\sigma_-(\alpha_i)$ of class $\alpha_i$ it follows  that $E=S_i$.

Thus the set of stable objects in $\P(1)$ is some subset of the objects $\{S_1,S_2,S_3\}$. It follows from this that  the objects  $S_1$ and $S_2$  must actually be stable in $\sigma$. For example, if $S_2$ was unstable, it would have a filtration by the objects $S_1$ and $S_3$, which is impossible because $\alpha_2$ is not a positive linear combination of the classes $\alpha_1$ and $\alpha_3=\alpha_1+\alpha_2$. 

Suppose there are extensions in two directions between $S_1$ and $S_2$. Then on both sides of the wall $\Im Z_\phi(\alpha_1)/Z_\phi(\alpha_2)=0$ there would be stable objects of class $[S_1]+[S_2]$, which is not the case for $\sigma_-$. It follows that the object  $S_3$ must also be stable, since the only other possibility is that there is a sequence
\[0\lra S_2\lra S_3\lra S_1\lra 0,\]
which is impossible since $S_3$ is stable on the $\sigma_-$ side of the wall. Using the same argument as before, we can now check that nonzero extensions between one of the objects $S_1$ or $S_2$ and $S_3$ only go in one direction. This completes the proof.\end{pf}

 Consider stability conditions  $W\colon K(\A)\to \C$  on the abelian category $\A$ satisfying  the conditions
 \begin{equation}
\label{cond}Z(S_1)+Z(S_2)\in i\R, \quad  \Im Z(S_1)/Z(S_2)>0.\end{equation}
 In the  case (J2), assume also that $Z(S_3)\in i\R$. Then

\begin{lemma}
\label{nono}
 The  set of stable objects  satisfying  $Z(E)\in i\R$ is independent of the particular choice of stability condition satisfying \eqref{cond}, and  is as follows:
\begin{itemize}
\item[(J1)] a single $\PP^1$ family of objects of dimension vector $(1,1)$;
\smallskip
\item[(J2)]  a single $\PP^1$ family of objects of dimension vector $(1,1,1)$, and unique objects of dimension vectors $(1,1,0)$ and $(0,0,1)$.
\end{itemize}
\end{lemma}

\begin{pf}
In the case (a) we are considering representations of the Kronecker quiver and the result is well-known. In case (b)  we are considering the affine $A_2$ quiver, and since indecomposable representations are completely understood in terms of the real and imaginary roots, the result is again easy. 
\end{pf}

Note the precise correspondence with the finite-length trajectories of $\phi_+$ listed in Section \ref{jug}. The following result then completes the proof of Theorem \ref{th}.

\begin{prop}
An object $E\in \D$ is stable of phase 1 in $\sigma_+$ precisely if it lies in the abelian subcategory $\A$ and is stable with respect to stability conditions $W$ as above.
\end{prop}

\begin{pf}
As  the ring domain $A$ shrinks we move along a path of differentials $\phi_+(t)$ for $0\leq t\leq 1$ with $\phi_+(0)=\phi_+$ and $\phi_+(1)=\phi$. Let $\sigma_+(t)=K(\phi_+(t))$ be the corresponding stability conditions.
The first claim is that  the class of stable objects of phase 1 in the stability condition $\sigma_+(t)$ is constant for $0\leq t<1$.
To prove this it will be enough to show that if $E$ is a stable object of phase 1 in some differential $\sigma_+(t)$, then the  class $\beta$ of $E$ is a multiple of the class $\alpha$ of the ring domain, and hence remains of phase 1 for all $t$. This follows immediately from Lemma \ref{done} and the list of saddle trajectories for $\sigma_+$ given in  Section \ref{jug}.

Suppose that $E\in \D$ is stable in $\sigma_+$ of phase 1. Then  by the above, the class of $E$ is proportional to $\alpha$, and moreover $E$ is stable  in $\sigma_+(t)$ for $0\leq t<1$. It follows that $E$  is at least  semistable in $\sigma$, and hence that $E\in \A=\P(1)$. Thus we are reduced to understanding which objects $E\in \A$ whose classes are proportional to $\alpha$ are stable  in the stability conditions $\sigma_+(t)$ for $0\leq t <1$.

Consider the central charge of  $\sigma_+(t)$ and  rotate it by an angle of $\pi/2$. It then induces a stability function $W$ on $\A$ satisfying  the conditions \eqref{cond}. It follows that an object $E\in \A$ of class proportional to $\alpha$ is stable in $\sigma_+(t)$ precisely if it is stable with respect to the stability conditions $W$ of Lemma \ref{nono}. 
\end{pf}


\subsection{Completion of the proof}
The following result will be enough to complete the proof of Theorem \ref{alltogethernow}.

\begin{prop}
\label{fed}
The map $K$ of Proposition \ref{most} is an isomorphism of complex manifolds.
\end{prop}

\begin{pf}
We begin by showing that $K$  is injective; it then follows that it is an open embedding because it commutes with the period maps, which are local isomorphisms.
Suppose we have two distinct framed differentials
\[\phi_1,\phi_2\in \Quad_*^\Gamma(\S,\M)\]
such that $K(\phi_2)=\Phi(K(\phi_1))$ for some autoequivalence $\Phi\in \Allo(\D)$.   Since the period maps are local isomorphisms, and  $K$ commutes with these maps, if we deform both the $\phi_i$ maintaining the condition that their periods are equal, we will also preserve the condition $K(\phi_2)=\Phi(K(\phi_1))$. Thus we can assume that the $\phi_i$ are saddle-free. Let $(T_i,\epsilon_i)$ be the associated signed WKB triangulations.

 By definition, $K(\phi_i)=\Psi_i(\sigma(\phi_i))$, where the $\Psi_i$ are equivalences
\[\Psi_i\colon \D(T_i,\epsilon_i)\to \D(T_0,\epsilon_0)\]
satisfying the conditions (i) and (ii) of the proof of Proposition \ref{most}.
It follows that $\Psi_2^{-1}\circ \Phi\circ \Psi_1$ takes the canonical heart $\A(T_1,\epsilon_1)\subset \D(T_1,\epsilon_1)$ to the canonical heart $\A(T_2,\epsilon_2)\subset \D(T_2,\epsilon_2)$. In particular, the quivers $Q(T_i)$ are isomorphic. Thus by Proposition \ref{weiwen}, the WKB  triangulations $T_i$ differ by an  orientation-preserving diffeomorphism.

Examining Figure \ref{dualgraph} it is easy to see that this implies  that the differentials $\phi_i$ have the same horizontal strip decomposition in the sense of Section \ref{periodshorizontal}.  The fact that the equivalences $\Psi_i$ satisfy the condition (ii)  of the proof of Proposition \ref{most} implies further  that the horizontal strips of the $\phi_i$ have the same labellings by elements of $\Gamma$.  Proposition \ref{ivans} then shows that $\phi_1=\phi_2\in \Quad^\Gamma(\S,\M)$.

 The fact that the image  of $K$ is closed follows from  Proposition \ref{mum}. Indeed, suppose $\sigma_n\to \sigma\in \Stab_\triangle(\D)$ with $\sigma_n=K(\phi_n)$. We can assume that each $\phi_n$ is complete and generic since these points are dense. Theorem \ref{th} implies that the lengths of non-degenerate saddle connections in $\phi_n$ correspond to the masses of stable objects in $\sigma_n$. 
Thus it will be  sufficient to show that the masses of  objects in the $\sigma_n$ are uniformly bounded below.
By continuity it is enough to check that the masses of  objects in $\sigma$ are bounded below, and this follows from the support property \eqref{support}.
 \end{pf}

To prove Theorem \ref{alltogethernow} it remains to show that $K$ descends to give an isomorphism of complex orbifolds fitting into the commutative diagram
\begin{equation}\xymatrix@C=1.5em{ \Quad_*^{\Gamma}(\S,\M)\ar[d] \ar[rrr]^{K} &&& \Stab_\triangle(\D)/\uAll^0(\D)\ar[d]\\
\Quadorb(\S,\M) \ar[rrr]_{J} &&& \Stab_\triangle(\D)/\uAll(\D) }\end{equation}
Since the map on the left is a covering map, to prove that $K$ descends it is enough to check that all framings of the differential $\phi_0$ give the same stability condition up to the action of $\uAll(\D)$. This is immediate from the definition of $K_0$. 

To prove that the resulting map $J$ is injective, hence an isomorphism, we follow the same argument used above to show that $K$ is injective. Suppose that $J(\phi_1)=J(\phi_2)$. By deforming the differentials $\phi_i$  as before, it is enough to deal with the case when the $\phi_i$ are complete and saddle-free. As above we conclude that the $\phi_i$ have the same horizontal strip decomposition. Since they also  have the same periods, they are equal.  \qed

\smallskip

\begin{corollary}
\label{reach}
An autoequivalence of $\D$ preserves the connected component $\Stab_\triangle(\D)$ precisely if it is reachable. In particular, the  shift functor $[1]$ lies in the group $\Reach(\D)$ of reachable autoequivalences. 
\end{corollary}

\begin{pf}
One implication is always true, by Lemma \ref{wolf}. Conversely, suppose that $\Phi\in \Aut(\D)$ preserves $\Stab_\triangle(\D)$. Take a smooth path connecting $\sigma_0$ and $\Phi(\sigma_0)$ in $\Stab_\triangle(\D)$. Applying $K^{-1}$ this defines a smooth path of framed quadratic differentials. As in the proof of Proposition \ref{most}, we can use Proposition \ref{moremore} to  find a homotopic path consisting entirely of tame differentials, and crossing only finitely many walls. Applying Proposition \ref{skateboard} at each of these walls shows that $\Phi$ is reachable.
\end{pf}




\subsection{Non-amenable cases}
\label{therest}

Suppose now that $(\S,\M)$ is  a marked bordered surface which fails to be amenable because it violates one of the first four conditions of Definition \ref{amenable}.  We shall make some brief comments about what can be said in these various cases.

{\bf Case (a)}. Suppose that $(\S,\M)$ does not satisfy Assumption \ref{asstwo}. We refer the reader to the comments made following Assumption \ref{asstwo}. We cannot deal with the case of a 4-punctured sphere, but the other cases are all described explicitly in Section \ref{applications}.

{\bf Case (b)}. Suppose that $(\S,\M)$ is one of the three surfaces listed in Proposition \ref{weiwen} having distinct triangulations with the same associated  quiver.
The isomorphism of Theorem \ref{kell} still holds and induces an action of the group $\Reach(\D)$ on the  set of tagged triangulations. We say that a reachable  autoequivalence  is \emph{allowable} if the induced action on the quotient set
\[\Tri_{\bowtie}(\S,\M)/\MCG^\pm(\S,\M)\]
is trivial. If we replace the group $\uAll(\D)$ of reachable autoequivalences by the subgroup $\uAll^{\operatorname{allow}}(\D)$ of allowable ones then Proposition \ref{clareshort} holds for $(\S,\M)$ with the same proof.   

In these cases Theorem \ref{alltogethernow} still holds, with the group $\uReach(\D)$ replaced by $\uAll^{\operatorname{allow}}(\D)$. The proof is the same, one just needs to check at several places that certain autoequivalences are allowable. This is easily done, and we omit the details. The case of a once-punctured disc with two points on the boundary is described in detail in Example \ref{egg} below.

{\bf Case (c)}. Consider the two non-closed surfaces listed in Proposition \ref{freely}, for which the action of the mapping class group on ordered ideal triangulations is not free. Proposition \ref{clareshort} also fails in these two cases. These surfaces are dealt with explicitly in Examples \ref{eg} and \ref{eg2} below. Note that  the orbifold $\Quadorb(\S,\M)$ has  non-trivial generic automorphism group $\Z_2$, and Theorem \ref{alltogethernow} continues to hold if we rigidify the orbifold $\Quadorb(\S,\M)$ by  killing this  group.

{\bf Case (d)}. We leave the case of a closed surface  with a single puncture for future research, and restrict ourselves here to some sketchy comments about some of the special features of these surfaces.

The first new feature  is that the graph $\Tri_{\bowtie}(\S,\M)$  has two connected components. This means that  we have two potentially distinct categories $\D(\S,\M)$, depending on the choice of sign $\epsilon=\epsilon(p)$. But by the result 
\cite[Prop. 10.4]{LF3} referred to in the proof of Theorem \ref{clareshort}, the two potentials $W(T,\pm 1)$ are right-equivalent up to scale, so  in fact there is a well-defined category $\D=\D(\S,\M)$.

 As stated, Theorem \ref{kell} is   false for these examples, again because the graph $\Tri_{\bowtie}(\S,\M)$  is disconnected. A related issue is that Corollary \ref{reach} fails, and the shift functor $[1]$ is not reachable\footnote{To see this, note that because there are no self-folded triangles, the residue class $\beta_p$ is always either a strictly positive or strictly negative linear combination of the classes of the simple objects of any heart, and this sign is constant under mutation.}. It might therefore be best to replace the notion of a reachable heart by a \emph{shift-reachable} heart: one for which $\A[i]$  lies in $\Tilt_\triangle(\D)$ for some $i\in \Z$. Of course, in the case of an amenable surface, the notions of reachable and  shift-reachable coincide by Corollary \ref{reach}.

One more feature in this case is that the complete Ginzburg algebra definitely does not coincide with the uncompleted version in general, see Remark \ref{completed}. It may well be that it is more natural to consider the uncompleted Ginzburg algebra in this context.




\section{Examples}
\label{applications}
In this section we consider some special cases of Theorem \ref{Thm:Main1} corresponding to surfaces $(\S,\M)$  of genus $g=0$. This leads to descriptions of spaces of stability conditions on \CY categories associated to  certain simple quivers familiar in representation theory. We adopt a less formal approach in this section, and some of the details are left for the reader.


\subsection{Unpunctured discs: $A_n$ type}

Fix an integer  $n\geq 2$ and let 
 $(\S,\M)$ be an unpunctured disc with $n+3$ points on its boundary. This corresponds to differentials on $\bP^1$ with a single pole of order $n+5$.
 The space $\Quadorb(\S,\M)$ coincides with $\Quad(\S,\M)$ and parameterizes differentials of the form \[\varphi(z)=p_{n+1}(z) \,dz^{\tensor 2}\] where $p_{n+1}(z)$ is a polynomial of degree $n+1$ having simple roots, considered modulo the automorphisms of $\bP^1$ which fix infinity. Taking the sum of the roots to be 0  we can reduce to differentials of the form
\[\phi(z)=\prod_{i=1}^{n+1} (z-a_i) \, dz^{\tensor 2},  \quad \sum_{i=1}^{n+1} a_i = 0,\quad a_i\neq a_j,\]
modulo a residual action of $\Z_{n+3}$ acting by rescaling $z$ by an $(n+3)$rd root of unity. Thus  \[\Quad(\S,\M)\isom \operatorname{Conf}_0^{n+1}(\C)/\Z_{n+3},\]
where $\operatorname{Conf}_0^{n+1}(\C)$ denotes the configuration space of $n+1$ distinct points in $\C$ with centre of mass at the origin, and  the group $\Z_{n+3}$ acts by multiplication by  $(n+3)$rd roots of unity.

\begin{figure}[ht]
\begin{center}
\includegraphics[scale=0.4]{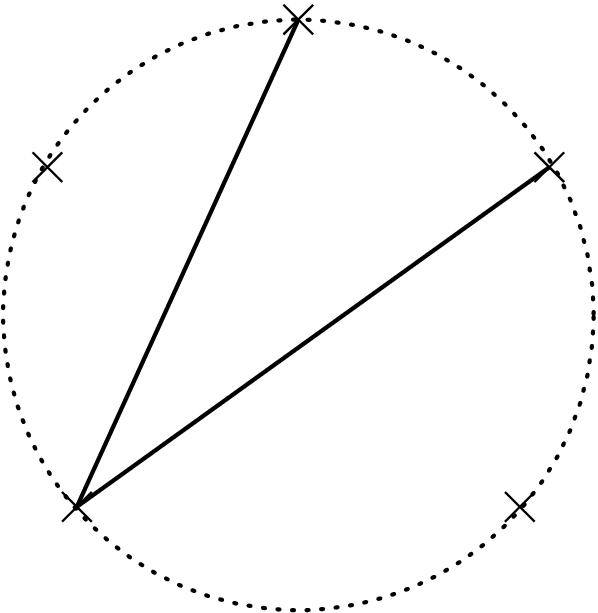}
\end{center}
{\caption{Triangulation of a 5-gon.}}
\end{figure}

The mapping class group  of the surface is \[\MCG(\S,\M)=\Z_{n+3}\] and coincides with the signed mapping class group.  Ideal triangulations  of $(\S,\M)$ correspond to triangulations of an $(n+3)$-gon. For triangulations containing no internal triangles the resulting quivers  are orientations of a Dynkin diagram of $A_n$ type and necessarily have zero potential.

Our main Theorem gives an isomorphism of complex orbifolds
\[\operatorname{Conf}_0^{n+1}(\C)/\Z_{n+3} \isom \Stab_\triangle(\D)/\uReach(\D).\]
 Note that there is a short exact sequence
\[1\lra \operatorname{Br}(A_n) \lra \pi_1 \Quad(\S,\M)\lra \Z_{n+3}\lra 1,\]
where the Artin braid group $\operatorname{Br}(A_n)$ is the fundamental group of the configuration space $\operatorname{Conf}_0^{n+1}(\C)$.  On the other hand, the  group $\Sph(\D)\subset \Aut(\D)$ is isomorphic to $\operatorname{Br}(A_n)$ by  \cite[Theorem 1.3]{ST}. The sequence of Lemma \ref{clareshort} therefore becomes\[ 1\lra \operatorname{Br}(A_n) \lra \uReach(\D)\lra \Z_{n+3}\lra 1.\]
The fact that these two sequences coincide suggests that $\Stab_\triangle(\D)$ is simply-connected.

\begin{Example}
\label{beggar}
The non-amenable case $n=0$ corresponding to an unpunctured disc with 3 marked points on its boundary is very degenerate. The space $\Quad(\S,\M)$ is a single point with  automorphism group $\Z_3$, corresponding to the unique differential $\phi(z)=z\, dz^{\tensor 2}$. The mapping class group is $\MCG(\S,\M)=\Z_3$. There is a unique ideal triangulation, but it contains no edges, and the associated quiver is empty.   \end{Example}

\begin{Example}
\label{eg}
Consider the non-amenable case  $n=1$ corresponding to an unpunctured disc with four marked points on its boundary. The space $\Quad(\S,\M)$ consists of differentials of the form
\[\phi(z)=(z^2+c) \, dz^{\otimes 2}, \quad c\in \C^*, \]
 modulo the action of $\Z_4$ acting on $z$ by multiplication by $i$.
Thus  \[\Quad(\S,\M)\isom \C^*/\Z_4\] with the generator of $\Z_4$ acting by change of sign. Note that the generic stabilizer group is $\Z_2$. 

There are two ideal triangulations of $(\S,\M)$, each with a single edge. These are related by the action of the mapping class group $\Z_4$. Note that Proposition \ref{freely} breaks down in this case: the element of $\Z_4$ of order 2 fixes all triangulations.

The  quiver corresponding to either triangulation has a single vertex and no arrows. The category $\D$ is generated by a single spherical object $S$, and the group $\uReach (\D)$ is generated by the shift functor $[1]$. Proposition \ref{clareshort} also breaks down in this case, since $\Sph(\D)\isom \Z$ is generated by the second shift $[2]$.

The quotient space  $\Stab(\D)/\langle[2]\rangle$ is isomorphic to $\C^*$ with co-ordinate $Z(S)$. Thus
\[ \Stab(\D)/\uReach(\D) \isom \C^*/\Z_2, \] with the generator of $\Z_2$ acting by change of sign. 
There is a  morphism
\[\Quad(\S,\M)\to \Stab(\D)/\uReach(\D)\]
 given by setting  $Z(S)=\pi i c$, but it is not  an isomorphism of orbifolds, since the generic automorphism groups are different. 
\end{Example}


\subsection{Punctured discs: $D_n$ type}

Fix an integer $n\geq 2$ and let $(\S,\M)$ be a once-punctured disc with $n$ points on its boundary.  This corresponds to differentials on $\bP^1$ with  polar type  $(2,n+2)$.

The space $\Quad(\S,\M)$ consists of differentials of the form \[\phi(z)=\prod_{i=1}^n (z-a_i) \,\frac{dz^{\tensor 2}}{z^2}, \quad a_i\neq a_j,\]  modulo the action of $\C^*$ rescaling $z$. Thus
\[\Quad(\S,\M)\isom \operatorname{Conf}^{n}(\C)/\Z_{n},\]
where $\operatorname{Conf}^n(\C)$ denotes the configuration space of $n$ distinct points in $\C$, and the group $\Z_n$ acts by multiplication by  $n$th roots of unity. 

Note that the product  $a=\prod_{i=1}^n a_i$ is invariant under the $\Z_n$ action,  and hence defines a map \[a\colon \Quad(\S,\M)\to \C.\] The cover $\Quad^\pm(\S,\M)$ 
corresponds to choosing a square-root of $a$.

 When  $n\geq 3$ the obvious rotationally-symmetric triangulation of $(\S,\M)$ has an associated  quiver with potential  which consists of a cycle of  $n$ arrows, equipped with a nonzero superpotential of degree $n$. This  is known
 \cite[Example 6.7]{FST}  to be mutation-equivalent to any orientation of the  Dynkin diagram of $D_n$ type, necessarily with zero potential.

\begin{Example}
\label{begg}
In the non-amenable case $n=1$  the space $\Quad(\S,\M)$ consists of differentials of the form
\[\phi(z)=(z+c) \,\frac{ dz^{\otimes 2}}{z^2}, \quad c\in \C.\]
The residue at 0 is $\res_0(\phi)=4\pi i\sqrt{c}$.
 The cover $\Quad^\pm(\S,\M)$ corresponds to choosing a square-root $s=\sqrt{c}$. Thus \[\Quadorb(\S,\M)\isom\C/\Z_2\]
with $\Z_2$ acting on $\C$ by $s\mapsto -s$.

There is only one triangulation, and  the resulting quiver is a single vertex with no arrows. The category $\D$ is generated by a single spherical object. The mapping class group is trivial, and $\MCG^\pm(\S,\M)=\Z_2$. The group $\Reach(\D)$ is generated by the shift $[1]$, and the subgroup $\Sph(\D)$ by the second shift $[2]$. The quotient $\Stab_\triangle(\D)/\Sph(\D)$ is isomorphic to $\C$ with co-ordinate $Z(S)$.
Thus the relation
\[\Quadorb(\S,\M)\isom \Stab_\triangle(\D)/\uReach(\D),\]
also holds in this case.
\end{Example}

\begin{Example}
\label{egg}In the non-amenable case $n=2$ the space $\Quad(\S,\M)$ consists of differentials of the form
\[\phi(z)=(az^2+bz+c) \,\frac{dz^{\otimes 2}}{z^2}, \quad a\in \C^*, \quad b,c\in \C, \quad b^2-4ac\neq 0,\]
modulo the rescaling action of $\C^*$. Using this we can take $a=1$; there is then a residual action of $\Z_2$ acting by $z\mapsto -z$. These differentials have a double pole  at $z=0$ and a fourth order pole at infinity. The respective residues are \[\res_0(\phi)=4\pi i \sqrt{c}, \quad \res_{\infty}(\phi)=2\pi i b.\]
The cover $\Quad^\pm(\S,\M)$ corresponds to choosing a square-root of $c$. Writing the differential as
\[\phi(z)=(z^2+2sz+t^2) \,\frac{dz^{\otimes 2}}{z^2}, \quad s,t \in \C,\quad s^2\neq t^2,\]
we see that $\Quad^\pm(\S,\M)$ is the quotient $ (\C^2\setminus \Delta)/\Z_2$ where $\Z_2$ acts by changing the sign of the first co-ordinate $s$, and $\Delta\subset \C^2$ is the union of the hyperplanes $s=\pm t$. We also have
\begin{equation}
\label{andrei}\Quadorb(\S,\M)=(\C^2\setminus\Delta)/\Z_2^{\oplus 2}\end{equation}
with the two $\Z_2$ factors changing the signs of $s$ and $t$ respectively.
There is a short exact sequence
\[1\lra \Z^{\oplus 2}\lra \pi_1 (\Quadorb(\S,\M))\lra \Z_2^{\oplus 2}\lra 1.\]

 The mapping class group and its signed version are
 \[\MCG(\S,\M)=\Z_2, \quad \MCG^\pm(\S,\M)=\Z_2^{\oplus 2}.\]
There are four tagged triangulations; the two possible taggings of the left-hand triangulation in Figure \ref{fi}, and the tagged triangulations corresponding to the right-hand picture and its rotation. The corresponding quivers all consist of two vertices with no arrows. 

\label{ex}
\begin{figure}[ht]
\begin{center}
\includegraphics[scale=0.4]{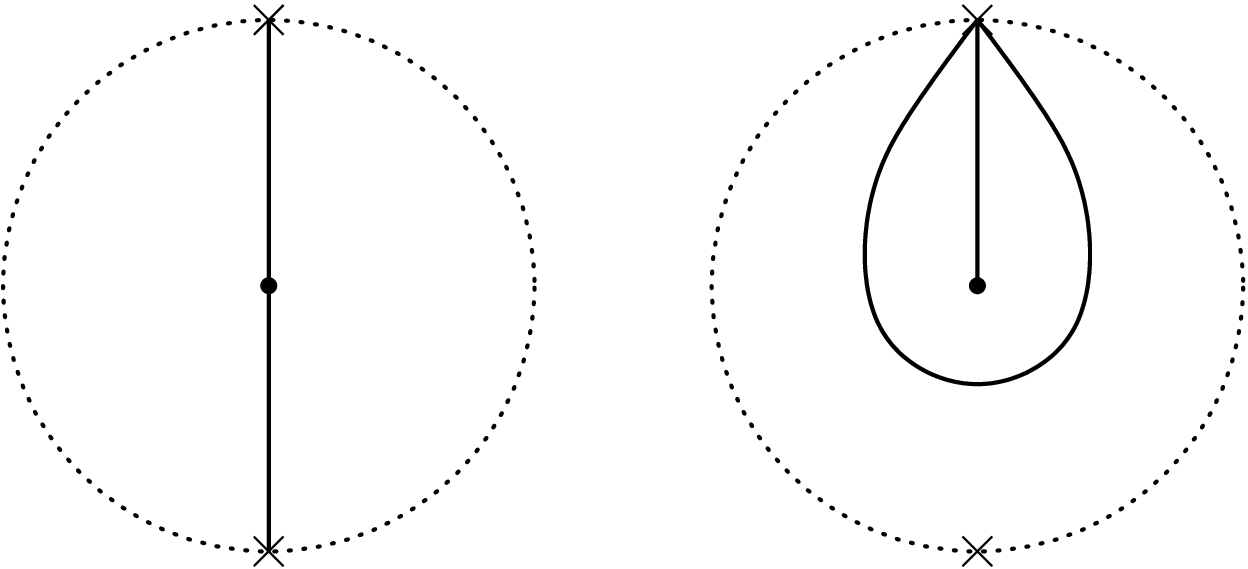}
\end{center}
\caption{Triangulations of a once-punctured disc with two marked points on the boundary.\label{fi}}
\end{figure} 

The category $\D$ is generated by two spherical objects $S_1,S_2$ lying in the heart of the t-structure corresponding to the triangulation on the left in Figure \ref{fi}. They have zero $\Ext$ groups between them. The twist functor $\Tw_{S_i}$ acts by  sending $S_i$ to $S_i[2]$  and leaving the other $S_j$ unchanged.

The group of allowable autoequivalences $\uAll^{\operatorname{allow}}(\D)$ is generated by the spherical twists $\Tw_{S_i}$, together with the autoequivalence swapping $S_1$ and $S_2$, and the shift functor $[1]$. It fits into a short exact seqence
\[1\lra \Z^{\oplus 2}\lra \uAll^{\operatorname{allow}}(\D)\lra \Z_2^{\oplus 2}\lra 1.\]
It has index 2 in the full group $\uAll(\D)$, which also contains the element sending $S_1$ to $S_1[1]$ and leaving $S_2$ fixed. 

The quotient $\Stab_\triangle(\D)/\Sph(\D)$ is isomorphic to $(\C^*)^2$ with co-ordinates $Z(S_i)$.
 The isomorphism of Theorem \ref{Thm:Main1} is given by \[Z(S_1)-Z(S_2)=2\pi i s, \quad Z(S_1)+Z(S_2)=2\pi i t.\]
In the $Z(S_i)$  co-ordinates the discriminant $\Delta$ of \eqref{andrei} is given by $Z(S_1)Z(S_2)=0$. One of the  $\Z_2$ factors acts by exchanging the $Z(S_i)$, and the  other  acts by changing the signs of the $Z(S_i)$. 
\end{Example}


\subsection{Unpunctured annuli: affine $A_n$ type}

Fix integers $p,q\geq 1$ and let  $(\S,\M)$ be an annulus whose boundary components contain  $p$ and $q$ marked points respectively.  This corresponds to differentials on $\bP^1$ with  polar type $(p+2,q+2)$. Let $n=p+q$.

The space $\Quad^\pm(\S,\M)=\Quad(\S,\M)$ consists of differentials of the form \[\phi(z)=\prod_{i=1}^n (z-a_i)\,\frac{dz^{\tensor 2}}{z^{p+2}},\quad a_i\in \C^*, \quad a_i\neq a_j.\]
If $p\neq q$ these are  considered modulo the  action of $\C^*$ rescaling $z$, so  we have
\[\Quad(\S,\M)\isom \operatorname{Conf}^{n}(\C^*)/\Z_{q},\]
where $\operatorname{Conf}^{n}(\C^*)$ denotes the configuration space of $n$ distinct points in $\C^*$, and the group $\Z_q$  acts by multiplication by a $q$th root of unity.
In the case $p=q$   one should quotient by an extra factor of $\Z_2$ acting by $z\leftrightarrow 1/z$.

We remark that despite appearences this answer is symmetric under exchanging $(p,q)$. To see this observe that the relevant quotients can also be viewed as the quotient of $\C^*\times \operatorname{Conf}^n(\C^*)$ by the action of $\C^*$ acting with weight $-p$ or $-q$ on the first factor and weight $1$ on each of the points in the configuration. These two actions are exchanged by  the involution of $\C^*\times \operatorname{Conf}^n(\C^*)$ defined by \[\big(t,\{a_1,\cdots, a_n\}\big)\mapsto \big((a_1\cdots a_n)\cdot t,\{a_1^{-1},\cdots,a_n^{-1}\}\big).\]

For any triangulation of $(\S,\M)$ the resulting quiver  $Q$ is a cycle of $n$ arrows but with a non-cyclic orientation. Thus $Q$ is a non-cyclic orientation of the affine $A_{n-1}$ Dynkin diagram, necessarily with zero potential. 

\begin{Example}
\label{eg2}
Consider the non-amenable case $p=q=1$.   The space $\Quad(\S,\M)$ parameterizes differentials of the form
\[\phi(z)=(tz+2s+tz^{-1}) \frac{dz^{\otimes 2}}{z^2},  \quad s\in \C,\quad t\in \C^*,  \quad t^2\neq s^2\]
modulo an action of $\Z_2^{\oplus 2}$, with one generator acting by changing the sign of $t$, and the other acting trivially, via the automorphism $z\mapsto 1/z$ of $\PP^1$. 
 Writing $a=s\in \C$ and $b=t^2\in \C^*$ we obtain
\[\Quad(\S,\M)\isom (\C\times\C^*\setminus \Delta)/\Z_2, \]
where $\Delta$ is the hypersurface $b=a^2$ and $\Z_2$ acts trivially. Write $\Quad'(\S,\M)$ for the rigidified moduli space obtained by forgetting about the trivial $\Z_2$ action. There is a short exact sequence
\[1\lra \Z*\Z\lra \pi_1 \Quad'(\S,\M)\lra \Z\lra 1\]
obtained from the obvious projection to $\C^*$ whose fibre over $b\in \C^*$ is $\C\setminus\{\pm \sqrt{b}\}$. 

\begin{figure}[ht]
\begin{center}
\includegraphics[scale=0.4]{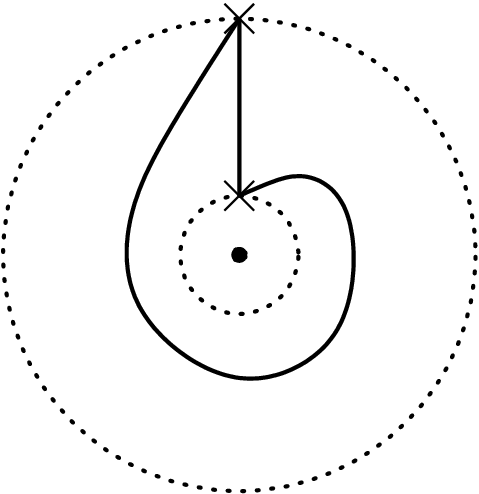}

\begin{equation*}
\xymatrix@C=1em{{^1}\bullet \ar@/^/[rr] \ar@/_/[rr]&& \bullet^2    }\end{equation*}
\caption{Triangulation and quiver for the annulus with one marked point on each boundary component.}
\end{center}
\end{figure}

The mapping class group is $\MCG(\S,\M)=\Z_2\ltimes \Z$ with the $\Z_2$ acting by exchanging the two boundary components, and $\Z$ acting by a Dehn twist around an equatorial curve. There is a single triangulation of $(\S,\M)$ up to diffeomorphism, and the associated quiver is the Kronecker quiver.

The spherical autoequivalence group $\Sph(\D)$ is free on the two generators $\Tw_{S_i}$. The mapping class group does not act freely on ordered triangulations, and Proposition \ref{clareshort} is false  in this case. But the same argument gives
a short exact sequence
\[1\lra \Z * \Z \lra \uAll(\D)\lra \Z\lra 1.\]
Theorem \ref{alltogethernow} continues to hold if we replace $\Quad(\S,\M)$ by $\Quad'(\S,\M)$. The central charges of the two simple objects are given by the elliptic integrals
\[Z(S_i)=\pm 2\int \sqrt{(tz+2s+tz^{-1})} \cdot \frac{dz}{z}\]
where the paths of integration are half-loops connecting the two zeroes of $\phi$.  \end{Example}
 
 
 \subsection{Three-punctured sphere}

\label{threepunctures}
 
 Let $(\S,\M)$ be the  three-punctured sphere. This surface is not amenable, but a version of our Theorem continues to hold. An interesting point is that it seems to be most natural to work with uncompleted Ginzburg algebras in this case (see Remark \ref{completed}).

 The space $\Quad(\S,\M)$ consists of differentials of the form
 \[\phi(z)=\frac{ (az^2+bz+c)\,dz^{\tensor 2}}{z^2 (z-1)^2}\]
 for $a,b,c\in \C$ with $b^2\neq 4ac$, modulo the action of the symmetric group $S_3$ acting via automorphisms of $\PP^1$ permuting $0,1,\infty$. The residues at the points $0,1,\infty$ are 
\[\frac{1}{4\pi i}\res_0 = u= \sqrt{c}, \quad \frac{1}{4\pi i}\res_1=v= \sqrt{a+b+c}, \quad \frac{1}{4\pi i}\res_\infty=w= \sqrt{a}.\]
The space $\Quad^\pm(\S,\M)$ is therefore $\C^3$ with co-ordinates $(u,v,w)$, minus the inverse image of the  discriminant locus $b^2-4ac$, all modulo $S_3$ acting by permutations on $(u,v,w)$. The inverse image of the discriminant  locus is easily seen to be the divisor $z_1 z_2 z_3 z_4=0$, where
\[z_1=-u+v+w, \quad z_2=u-v+w, \quad z_3=u+v-w, \quad z_4=u+v+w.\]
Thus we conclude that \[\Quad^\pm(\S,\M)\isom ((\C^*)^3\setminus \Delta)/S_3\]
where the discriminant $\Delta$ is given by $z_1+z_2+z_3=0$ and  the symmetric group $\Sym_3$ acts by permuting $(z_1,z_2,z_3)$. Hence
\[\Quadorb(\S,\M)\isom ((\C^*)^3\setminus \Delta)/(\Sym_3\ltimes (\Z_2)^{\oplus 3}),\]
where the $\Z_2$ factors change the signs of $u,v,w$ respectively. The fundamental group then sits in a sequence
\[ 1\lra \pi_1((\C^*)^3\setminus \Delta)\lra \pi_1 (\Quadorb(\S,\M))\lra \Sym_3\ltimes \Z_2^{\oplus 3}\lra 1.\]
The space $(\C^*)^3\setminus \Delta$ is a trivial $\C^*$-bundle over its projectivisation, which is the complement of four hyperplanes in $\mathbb{CP}^2$. 
A theorem due to Zariski \cite[Lemma, p. 317]{Zariski} asserts that $\pi_1((\C^*)^3\setminus \Delta)$ is abelian, hence isomorphic to $H_1((\C^*)^3\setminus \Delta) \isom \Z^{\oplus 4}.$

The mapping class group is $\MCG(\S,\M)=S_3$ and permutes the punctures in the obvious way. The signed mapping class group is $\MCG^\pm(\S,\M)=S_3\ltimes \Z_2^{\oplus 3}$. There are two triangulations up to the  mapping class group action: one  has two triangles meeting along three common edges, and the other consists of two self-folded triangles glued along their encircling edges. We claim that the relevant quiver with potential in both cases is the one depicted in Figure \ref{quiv}.

\begin{figure}[ht]
\begin{center}
\begin{equation*}
\xymatrix@C=1em{\bullet^1 \ar^{a}@/^/[rrrr] \ar^{c'}@/^/[ddrr]&&&& \bullet^2 \ar^b@/^/[ddll]    \ar^{a'}@/^/[llll] \\ \\ 
  && \bullet^3 \ar^c@/^/[uull]^c \ar^{b'}@/^/[uurr]}
\end{equation*}
\[W=aa'+bb'+cc' -abc-c'b'a'.\]
\caption{The quiver with potential for the 3-punctured sphere.\label{quiv}}
\end{center}
\end{figure}
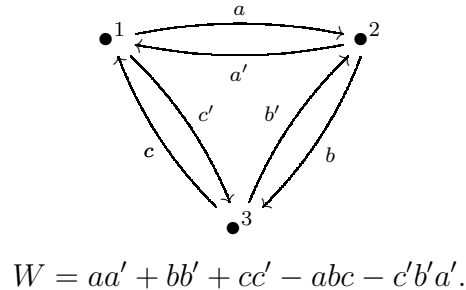
Note that the associated reduced quiver with potential has three verices and no arrows. Let $\A'$ be the category of finite-dimensional representations of the ordinary (incomplete)  Jacobi algebra of the above quiver with potential, and $\A$ the category of finite-dimensional representations of the completed version. Then $\A\subset \A'$ is the full subcategory of nilpotent representations.

\begin{lemma}
The category $\A'$  has exactly four indecomposable objects, all of them simple, namely the three vertex simple objects $S_1,S_2,S_3$ together with the representation $S_4$ of dimension vector $(1,1,1)$ obtained by taking all arrows to be the identity. The subcategory $\A$ is the subcategory consisting of direct sums of  the objects $S_1,S_2,S_3$.
\end{lemma}

\begin{pf}
Differentiating the potential we see that the Jacobi algebra has relations
\[a'=bc, \quad b'=ca, \quad c'=ab,\]
which allow us to eliminate $a',b',c'$. The remaining relations are then
\begin{equation}
\label{late}
a=abca, \quad b=bcab, \quad c=cabc.\end{equation}
Given a representation of the Jacobi algebra we can split the vector space associated to vertex 1 as $\Im(abc)\oplus \Ker(abc)$. Indeed, by the relations \eqref{late}, if $v=(w)abc$ lies in the kernel of the map $abc$ then $(w)a=(w)abca=(v)a=(v)abca=0$ and hence  $v=0$. Similar splittings exist  at the other vertices and it follows easily from the relations \eqref{late}  that all arrows preserve these splittings. Hence, if $E$ is an indecomposable object of $\A'$ we either have $abc=bca=cba=0$ or each of these maps is injective. 

In the first case it follows from the relations \eqref{late} that all arrows are 0, and hence, since $E$ is indecomposable it must be one of the vertex simples. In the second case, each of $a,b,c$ is injective, hence they are all isomorphisms, and \eqref{late} implies that $abc$ is the identity, and similarly for $bca$ and $cab$. We can then choose the gauge so that $a,b,c$ are identity maps, and since $E$ is indecomposable it follows that the dimension vector must be $(1,1,1)$. This unique extra indecomposable $S_4$ is  simple, since there are no maps between it and the vertex simples.
\end{pf}

It follows from the Lemma that  the objects $S_i$ are all spherical and have no  extensions between them, since otherwise there would be more indecomposable objects in the category. Consider now the the  derived category of the  uncompleted Ginzburg algebra of the above quiver with potential, and the subcategory $\D'$ of objects with finite-dimensional cohomology. It is a \CY triangulated category  with a heart $\A'\subset \D$ containing 4 simple, spherical objects, with zero $\Ext$-groups between them.

The Grothendieck group is $K(\D')\isom \Z^{\oplus 4}$ and the dimension vector defines a group homomorphism
\[d\colon K(\D')\to \Z^{\oplus 3}.\]
 Consider the space $\Stab'(\D')\subset \Stab(\D')$ consisting of  those stability conditions whose central charge factors via $d$. This is acted on by the group of autoequivalences $\Aut'(\D')$ whose action on $K(\D')$ preserves the kernel of $d$. 

The group  of all exact autoequivalences of $\D'$ is $\Aut(\D')\isom \Sym_4\ltimes \Z^{\oplus 4} $, with the four generators $\rho_i$ shifting the four  simples $S_i$ respectively, and the symmetric group $\Sym_4$ permuting them. The subgroup $\Aut'(\D')$ contains a subgroup $\Sph(\D')\isom \Z^{\oplus 4}$  generated by the 4 elements $\rho_i^2$.

\begin{lemma}
There is a short exact sequence
 \[1\lra \Sph(\D')\lra \Aut'(\D')\lra  \Sym_3\ltimes\Z_2^{\oplus 3}\lra 1.\]
\end{lemma}

\begin{pf}
An element of $\Aut'(\D')$ is determined by its action on the objects $S_i$, and, up to the action of $\Sph(\D')$, each of these is taken to an object of the form $S_j$ or $S_j[1]$. Consider the  transformation 
\[\tau_{(12)(34)}\colon (S_1,S_2,S_3,S_4)\mapsto (S_2,S_1,S_4[1],S_3[1])\]
along with its two conjugates by permutations of $(S_1,S_2,S_3)$. There is a relation
\[\tau_{(12)(34)}\circ \tau_{(13)(24)}\circ \tau_{(23)(14)}=[1]\in \Aut'(\D')/\Sph(\D').\]
Consider an element $\sigma$ of the quotient group $\Aut'(\D')/\Sph(\D')$. Composing with the above transformations we can asssume that $\sigma$ takes $S_4$ to itself.
The relation $[S_4]=[S_1]+[S_2]+[S_3]$ then forces $\sigma$ to be a permutation of the objects  $(S_1,S_2,S_3)$.
\end{pf}

The  space $\Stab'(\D')/\Sph(\D')$ is equal to $(\C^*)^3\setminus \Delta$ with co-ordinates $z_i=Z(S_i)$, where the discriminant locus $\Delta$ is given by  $z_1+z_2+z_3=0$ as before. Thus we obtain
  an isomorphism
 \[ \Quadorb(\S,\M)\isom \Stab'(\D')/\Aut'(\D')\]
 which can be thought of as a modifed version of Theorem \ref{Thm:Main0}.


\bibliographystyle{amsplain}

\end{document}